\documentclass{amsart}

\usepackage{enumerate}
\usepackage{amssymb}
\usepackage{graphicx}

\title[double cone isomorphic to triple]{The Griffiths double cone group is isomorphic to the triple}
\author{Samuel M. Corson}

\theoremstyle{definition}\newtheorem{theorem}{Theorem}
\theoremstyle{definition}
\theoremstyle{definition}
\theoremstyle{definition}
\theoremstyle{definition}

\theoremstyle{definition}\newtheorem{bigtheorem}{Theorem}

\numberwithin{theorem}{section}
\theoremstyle{definition}\newtheorem{corollary}[theorem]{Corollary}
\theoremstyle{definition}\newtheorem{proposition}[theorem]{Proposition}
\theoremstyle{definition}\newtheorem{definition}[theorem]{Definition}
\theoremstyle{definition}
\theoremstyle{definition}
\theoremstyle{definition}
\theoremstyle{definition}
\theoremstyle{definition}\newtheorem{lemma}[theorem]{Lemma}
\theoremstyle{definition}
\theoremstyle{definition}
\theoremstyle{definition}
\theoremstyle{definition}

\newcommand{\Red}{\operatorname{Red}}
\newcommand{\pro}{\operatorname{proj}}

\newcommand{\GS}{\mathbb{G}\mathbb{S}}
\newcommand{\E}{\operatorname{E}}
\newcommand{\proj}{\operatorname{proj}}
\newcommand{\Co}{\mathcal{C}}
\newcommand{\Ao}{\mathcal{A}}
\newcommand{\diam}{\operatorname{diam}}
\newcommand{\im}{\operatorname{im}}
\newcommand{\R}{\operatorname{R}}
\newcommand{\Aut}{\operatorname{Aut}}

\newcommand{\HA}{\mathbb{H}\mathbb{A}}
\newcommand{\Ho}{\operatorname{H}}
\newcommand{\W}{\mathcal{W}}

\newcommand{\pchunk}{\operatorname{p-chunk}}
\newcommand{\Pfine}{\operatorname{Pfine}}
\newcommand{\pindex}{\operatorname{p-index}}
\newcommand{\Close}{\operatorname{Close}}
\newcommand{\coi}{\operatorname{coi}}
\newcommand{\dom}{\operatorname{dom}}
\newcommand{\first}{\operatorname{h}}
\newcommand{\Ea}{\mathbb{E}}
\newcommand{\Pure}{\operatorname{Pure}}

\begin{document}

\address{School of Mathematics, University of Bristol, Fry Building, Woodland Road, Bristol, BS8 1UG, United Kingdom.}
\email{sammyc973@gmail.com}

\keywords{fundamental group, Griffiths space, harmonic archipelago, Hawaiian earring, topological cone}
\subjclass[2010]{Primary 03E75, 20A15, 55Q52; Secondary 20F10, 20F34}
\thanks{This work was supported by European Research Council grant PCG-336983, the Severo Ochoa Programme for Centres of Excellence in R\&D SEV-20150554, and by the Heilbronn Institute for Mathematical Research, Bristol, UK.}

\begin{abstract}  It is shown that the fundamental group of the Griffiths double cone space is isomorphic to that of the triple cone.  More generally if $\kappa$ is a cardinal such that $2 \leq \kappa \leq 2^{\aleph_0}$ then the $\kappa$-fold cone has the same fundamental group as the double cone.  The isomorphisms produced are non-constructive, and no isomorphism between the fundamental group of the $2$- and of the $\kappa$-fold cones, with $2 < \kappa$, can be realized via continuous mappings.  We also prove a conjecture of James W. Cannon and Gregory R. Conner which states that the fundamental group of the Griffiths double cone space is isomorphic to that of the harmonic archipelago.
\end{abstract}

\maketitle

\begin{section}{Introduction}

The Griffiths double cone over the Hawaiian earring, which we denote $\GS_2$, was introduced by H. B. Griffiths in \cite{G} and has long stood as an interesting example in topology (Figure \ref{doubleconefig}).  Although $\GS_2$ is a path connected, locally path connected compact metric space (a \emph{Peano continuum}) which embeds as a subspace of $\mathbb{R}^3$, it has some subtle properties.  Despite being a wedge of two contractible spaces, $\GS_2$ is not itself contractible, and more surprisingly the fundamental group of $\GS_2$ is uncountable.  The fundamental group is freely indecomposable and includes a copy of the additive group of the rationals and of the fundamental group of the Hawaiian earring.  This group has found use in defining cotorsion-free groups in the non-abelian setting \cite{EF} and continues to serve as a counterexample \cite{Z} and as a test model for notions of infinitary abelianization \cite{BG}.

It is easy to see that analogous behavior is exhibited when one uses more cones in the wedge, as in the triple wedge $\GS_3$ of cones over the Hawaiian earring.  A natural question is whether the isomorphism type of the fundamental group changes with this change in subscript.  In light of the intuitive fact that no spacial isomorphism can be defined (see the forthcoming Theorem \ref{continuous}), the following answer is surprising.

\begin{bigtheorem}\label{bigisomorphism}  If $\kappa$ is a cardinal such that $2 \leq \kappa \leq 2^{\aleph_0}$ then $\pi_1(\GS_2) \simeq \pi_1(\GS_{\kappa})$.
\end{bigtheorem}

The bounds on $\kappa$ in the statement of Theorem \ref{bigisomorphism} are the best possible.  The spaces $\GS_0$ and $\GS_1$ both strongly deformation retract to a point and therefore have trivial fundamental group, and when $\kappa>2^{\aleph_0}$ one has $|\pi_1(\GS_{\kappa})| > 2^{\aleph_0} = |\pi_1(\GS_2)|$ (Theorem \ref{howmany}).  Using techniques of \cite{EF} or \cite{HH} one can compute the abelianizations of $\pi_1(\GS_2)$ and $\pi_1(\GS_3)$ and see that these abelianizations are isomorphic.

A notable point of comparison is that the wedge of $2$, $3$, etc. Hawaiian earrings (without cones) is again homeomorphic to the Hawaiian earring, and so these spaces have isomorphic fundamental groups.  However the fundamental group of a wedge of $\aleph_0$ Hawaiian earrings, under the topology that we are considering, will not have isomorphic fundamental group.  This follows since the $\aleph_0$-wedge of Hawaiian earrings retracts to a subspace which is the $\aleph_0$-wedge of circles each having diameter $1$, and this shows that the fundamental group of the $\aleph_0$-wedge homomorphically surjects onto an infinite rank free group, which the fundamental group of the Hawaiian earring cannot do \cite{Hig}.

The isomorphism given in Theorem \ref{bigisomorphism} is produced combinatorially by a back-and-forth argument, using the axiom of choice.  One can ask whether an isomorphism can be given more explicitly using constructive methods, perhaps via continuous maps between spaces.  This is impossible because of the following theorem.

\begin{bigtheorem}\label{continuous}  If $1 \leq n< \kappa$ with $n$ finite the following hold:

\begin{enumerate}
\item  If $f: \GS_n \rightarrow \GS_{\kappa}$ is continuous then $f_*(\pi_1(\GS_n))$ is of uncountable index in $\pi_1(\GS_{\kappa})$.

\item  If $f: \GS_{\kappa} \rightarrow \GS_n$ is continuous then $\ker(f_*)$ is uncountable.
\end{enumerate}
\end{bigtheorem}

A comparable situation in the setting of topological groups is that $\mathbb{R}$ and $\mathbb{R}^2$ are isomorphic as abstract groups, since by picking a Hamel basis over $\mathbb{Q}$ one sees that both are isomorphic to $\bigoplus_{2^{\aleph_0}} \mathbb{Q}$.  There is no continuous, or even Baire measurable, isomorphism between these topological groups.  By contrast Theorem \ref{bigisomorphism} does not seem to follow by producing isomorphisms to an easily understood third group like $\bigoplus_{2^{\aleph_0}} \mathbb{Q}$.

Another curiosity worth mentioning is that despite the necessary constraints on the cardinality of $\kappa$ in Theorem \ref{bigisomorphism}, the first-order logical theory of $\pi_1(\GS_2)$ and $\pi_1(\GS_{\kappa})$ are the same whenever $\kappa \geq 2$.

\begin{bigtheorem}\label{elementaryequiv}   If $2 \leq \gamma\leq \kappa$ then $\pi_1(\GS_{\gamma})$ elementarily embeds in $\pi_1(\GS_{\kappa})$.  Thus for $\kappa \geq 2$ the groups $\pi_1(\GS_2)$ and $\pi_1(\GS_{\kappa})$ are elementarily equivalent.
\end{bigtheorem}

Of course when $\kappa$ is $0$ or $1$ the fundamental group $\pi_1(\GS_{\kappa})$ is trivial and therefore not elementarily equivalent to $\pi_1(\GS_2)$.  The proof of Theorem \ref{elementaryequiv} utilizes Theorem \ref{bigisomorphism} and the action of the automorphism group, and no previous knowledge of first-order logic is required to understand the proof.

\begin{figure}
\includegraphics[height = 5cm]{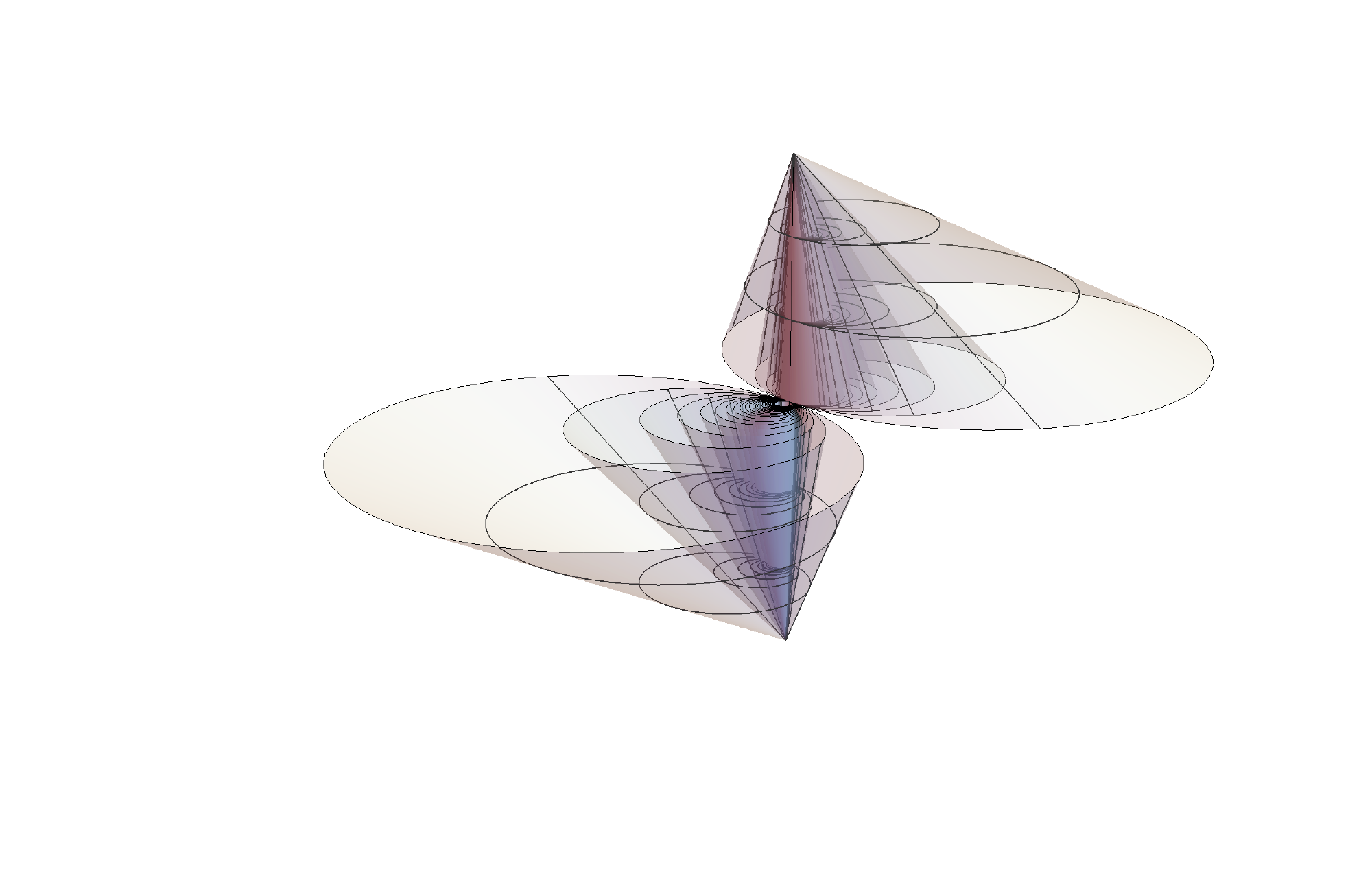}
\caption{The Griffiths double cone $\GS_2$}
\label{doubleconefig}
\end{figure}

The ideas used in proving Theorem \ref{bigisomorphism} seem to have very broad applications, and we state one now.  Another space that is often mentioned along with the Griffiths space is the harmonic archipelago $\HA$ of Bogley and Sieradski \cite{BS} (see Figure \ref{harmonicarchipelagofig}).  The spaces $\GS_2$ and $\HA$ share many common properties.  Each embeds as a subspace of $\mathbb{R}^3$, both contain a distinguished point at which every loop can be homotoped to have arbitrarily small image, and both have uncountable fundamental group.  Cannon and Conner have conjectured that the two spaces share a further property, namely that they have isomorphic fundamental group \cite{C2}.  We show that this is the case.

\begin{bigtheorem} \label{bigisomorphism2}  The groups $\pi_1(\GS_2)$ and $\pi_1(\HA)$ are isomorphic.
\end{bigtheorem}

One can quickly convince oneself that there cannot be a continuous function from one space to the other which induces an isomorphism on fundamental groups.  The abelianizations of these groups are known to be isomorphic \cite{KR}, \cite{EF}, \cite{HH}.  The proof of Theorem \ref{bigisomorphism2} uses modifications of that of Theorem \ref{bigisomorphism}.  It seems clear that by further reworking these ideas one can produce a correct proof of the main theorem of \cite{CHM} (some errors have been pointed out by K. Eda) as well as answer many of the questions of that paper in the affirmative.

\begin{figure}
\includegraphics[height = 5cm]{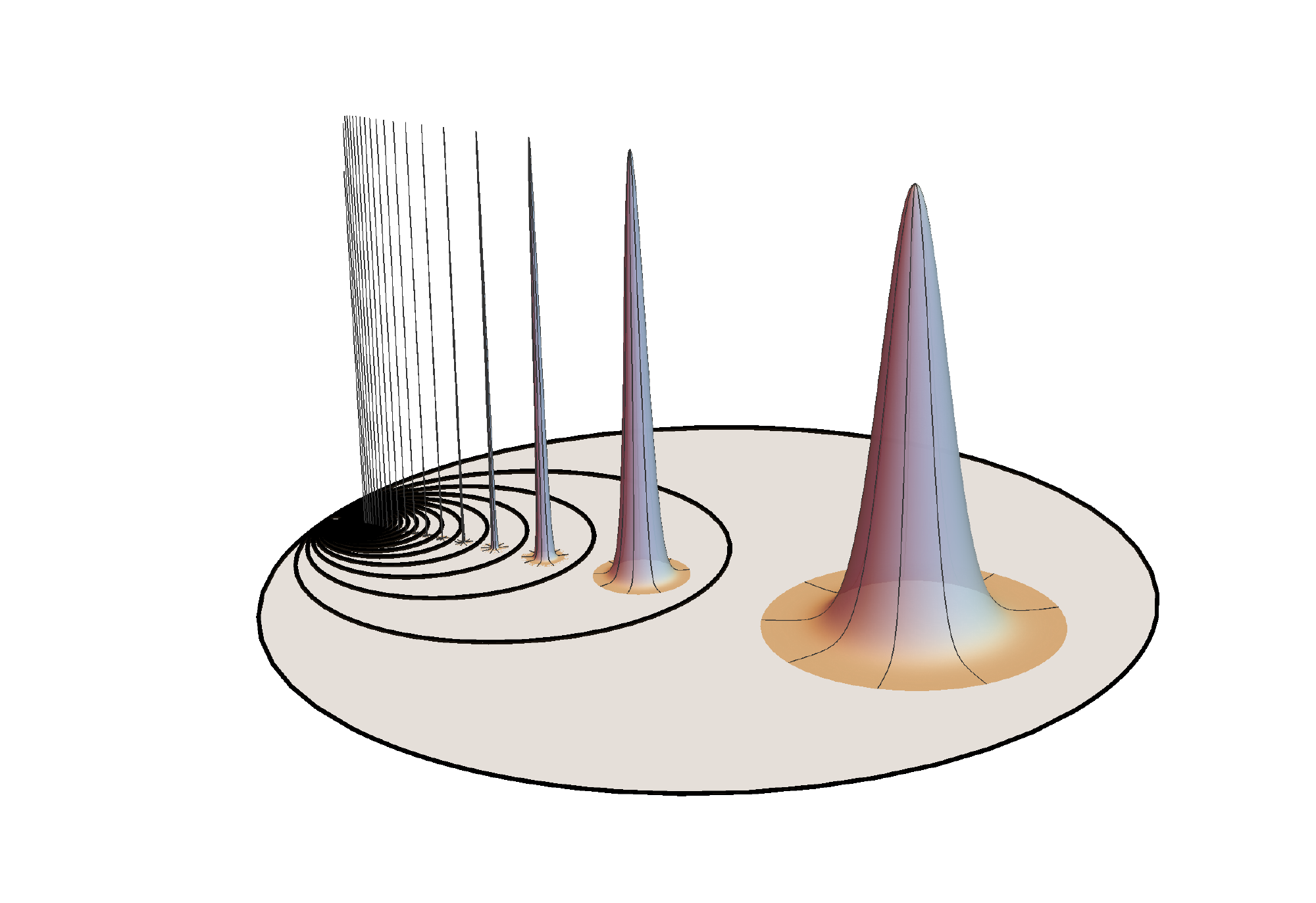}
\caption{The harmonic archipelago $\HA$}
\label{harmonicarchipelagofig}
\end{figure}

We describe the layout of this paper.  In Section \ref{conegroups} we give the formal definition of the Griffiths space and its $\kappa$-fold analogues.  We also present some combinatorially defined groups $\Co_{\kappa}$ and show them to be isomorphic to the fundamental groups $\pi_1(\GS_{\kappa})$.  We also prove Theorem \ref{continuous}.  In Section \ref{isom} we prove Theorems \ref{bigisomorphism} and \ref{elementaryequiv}.  In Section \ref{last} we prove Theorem \ref{bigisomorphism2}.
\end{section}

\begin{section}{The cone groups}\label{conegroups}

We give a construction of $\GS_2$ and more generally of the $\kappa$-fold Griffiths space $\GS_{\kappa}$ for any cardinal $\kappa$.  We consider each cardinal number $\kappa$ as being the set of all ordinals below it in the standard way.  Thus $0 = \emptyset$, $n = \{0, \ldots, n-1\}$ for each $n\in \omega$, $\omega + 2 = \{0, 1, \ldots, \omega, \omega + 1\}$, etc.  Let $2^{\aleph_0}$ denote the cardinal of the continuum.  Given a point $p\in \mathbb{R}^2$ and $r\in [0, \infty)$ we let $C(p, r)$ denote the circle centered at $p$ of radius $r$ (in case $r=0$ we obtain the degenerate circle consisting only of the point $p$).  The \emph{Hawaiian Earring} is the subspace $\Ea=\bigcup_{n\in \omega} C((0, \frac{1}{n+3}),\frac{1}{n+3})$ of $\mathbb{R}^2$.  Let $\GS_1 \subseteq \mathbb{R}^3$ be the subspace $\bigcup_{r\in [0, 1]}(\bigcup_{n\in \omega} C((0, \frac{1 - r}{n+3}), \frac{r}{n+3})\times \{r\})$.  The space $\GS_1$ may also be viewed as the space obtained by first taking the Hawaiian earring sitting in the xy-plane $\Ea \times \{0\}$ and joining each point of $\Ea \times \{0\}$ to the point $(0, 0, 1)$ by a geodesic line segment.  A third, topological way of viewing $\GS_1$ is by simply taking the topological cone over the Hawaiian earring.  In other words, $\GS_1$ is homeomorphic to the quotient space obtained by beginning with $\Ea\times [0, 1]$ and identifying all points which have $1$ in the last coordinate.

We define $\GS_0$ to be the metric space consisting of the single point $\circ_0$.  Let $\kappa\geq 1$ be a cardinal.  We take $\GS_{\kappa}$ to be the set obtained by taking $\kappa$-many disjoint isometric copies $\bigsqcup_{\alpha<\kappa} X_{\alpha}$ of $\GS_1$ and identifying all copies of $(0, 0, 0)$ to a single point $\circ_{\kappa}$.  Thus we consider $\circ_{\kappa} \in X_{\alpha}$ for all $\alpha <\kappa$.  Metrize $\GS_{\kappa}$ by letting

\[
d(x,y) = \left\{
\begin{array}{ll}
d_{\alpha}(x, y)
                                            & \text{if } x, y\in X_{\alpha}, \\
d_{\alpha}(x, \circ_{\kappa}) + d_{\alpha'}(\circ_{\kappa}, y)                                           & \text{if } x\in X_{\alpha}\setminus\{\circ_{\kappa}\} \text{ and }y\in X_{\alpha'}\setminus \{\circ_{\kappa}\}, \alpha \neq \alpha'.
\end{array}
\right.
\]

\noindent  We note that this definition yields an isometric copy of $\GS_1$ when $\kappa = 1$ and so the definition is consistent.  When $\kappa$ is finite, the space $\GS_{\kappa}$ is a Peano continuum and $\GS_{\kappa}$ is homeomorphic to the topological wedge of $\kappa$-many copies of $\GS_1$ with the copies of the point $(0, 0, 0)$ identified.  When $\kappa\geq \aleph_0$ the space $\GS_{\kappa}$ is neither compact nor homeomorphic to the quotient space obtained by identifying all copies of $(0, 0, 0)$ in the topological disjoint union of $\kappa$-many copies of $\GS_1$.

Next we give a description of what we call the \emph{cone group} $\Co_{\kappa}$ for each cardinal $\kappa$.  The description involves infinitary word combinatorics.  Fix a cardinal $\kappa$.  We start with a set $\Ao_{\kappa}=  \{a_{\alpha, n}^{\pm 1 }\}_{\alpha<\kappa, n< \omega}$ equipped with formal inverses.  We call the elements of $\Ao_{\kappa}$ \emph{letters} and a letter is \emph{positive} if it has superscript $1$.  For convenience we shall usually leave off the superscript $1$ on positive letters.  A letter which is not positive is \emph{negative}.  Let $\pro_0$, respectively $\pro_1$, be the functions defined on $\Ao_{\kappa}$ which project the first, resp. second, subscript of a letter.  Thus $\pro_0(a_{\alpha, n}^{-1}) = \alpha$ and $\pro_1(a_{\alpha, n}^{-1}) = n$.

A \emph{word} in $\Ao_{\kappa}$ is a function $W: \overline{W} \rightarrow \Ao_{\kappa}$ such that $\overline{W}$ is a totally ordered set and for each $N\in \omega$ the set $\{i\in \overline{W}\mid \pro_1(W(i)) \leq N\}$ is finite.  The domain of a word is necessarily countable.  We write $W_0 \equiv W_1$ if there exists an order isomorphism $\iota: \overline{W_0} \rightarrow \overline{W_1}$ such that $W_1(\iota(i)) = W_0(i)$ for all $i\in \overline{W_0}$, and write $\iota: W_0 \equiv W_1$ in this case.  Let $E$ denote the word with empty domain.

Let $\W_{\kappa}$ denote the set of all $\equiv$ classes of words in $\Ao_{\kappa}$.  For $W \in \W_{\kappa}$ we let $d(W) = \min\{\pro_1(W(i))\mid i\in \overline{W}\}$ and $d(E) = \infty$.  There is a natural associative binary operation on $\W_{\kappa}$ given by word concatenation, defined by letting $W_0W_1$ be the word $W$ such that $\overline{W} = \overline{W_0} \sqcup \overline{W_1}$ has the ordering that extends the orders of $\overline{W_0}$ and $\overline{W_1}$, placing elements in $\overline{W_0}$ below those of $\overline{W_1}$, and

\[
W(i) = \left\{
\begin{array}{ll}
W_0(i)
                                            & \text{if } i\in \overline{W_0}, \\
W_1(i)                                           & \text{if } i\in \overline{W_1}.
\end{array}
\right.
\]

\noindent There is similarly a notion of infinite concatenation.  If $\Lambda$ is a totally ordered set and $\{W_{\lambda}\}_{\lambda\in \Lambda}$ is a collection of words such that for every $N\in \omega$ the set $\{\lambda\in \Lambda: d(W_{\lambda}) \leq N\}$ is finite then we can take a concatenation $\prod_{\lambda\in \Lambda}W_{\lambda}$ whose domain is the disjoint union $\bigsqcup_{\lambda\in \Lambda} \overline{W_{\lambda}}$ ordered in the natural way and whose outputs are given by $(\prod_{\lambda\in \Lambda}W_{\lambda})(i) = W_{\lambda}(i)$ where $i\in \overline{W_{\lambda}}$.  We also use this notation for the concatenation of ordered sets.  If $\{\Lambda_{\lambda}\}_{\lambda\in \Lambda}$ is a collection of ordered sets and $\Lambda$ is itself ordered we let $\prod_{\lambda\in \Lambda}\Lambda_{\lambda}$ be the ordered set obtained by taking the disjoint union of the $\Lambda_{\lambda}$ and ordering the elements in the obvious way.  To further abuse notation we write $\Lambda \equiv \Theta$ if $\Lambda$ is order isomorphic to $\Theta$.

We also have an inversion operation on words given by letting $W^{-1}$ have domain $\overline{W}$ under the reverse order and letting $W^{-1}(i) = (W(i))^{-1}$.  For each $N\in \omega$ and word $W$ we let $p_N(W)$ be the restriction $W\upharpoonright\{i\in \overline{W}\mid \pro_1(W(i)) \leq N\}$.  Thus $p_N(W)$ is a finite word in the alphabet $\Ao_{\kappa}$.  We write $W_0 \sim W_1$ if for every $N\in \omega$ the words $p_N(W_0)$ and $p_N(W_1)$ are equal when considered as elements in the free group on positive elements of $\Ao_{\kappa}$.  As an example, the word $W \equiv a_{0, 0}a_{0, 0}^{-1}a_{0, 1}a_{0, 1}^{-1}\cdots$ satisfies $W \sim E$ since $p_N(W) \equiv a_{0, 0}a_{0, 0}^{-1}a_{0, 1}a_{0, 1}^{-1}\cdots a_{0, N}a_{0, N}^{-1}$ is freely equal to $E$ for each $N\in \omega$.  Let $[W]$ denote the $\sim$ equivalence class of $W$.  We obtain a group structure on $\W_{\kappa}/\sim$ by letting $[W_0][W_1] = [W_0W_1]$, from which one gets inverses defined by $[W]^{-1} = [W^{-1}]$ and $[E]$ as the identity element.  Let $\Ho_{\kappa}$ denote this group.  Define a word $W$ to be \emph{$\alpha$-pure} if $p_0\circ W(i) = \alpha$ for all $i\in \overline{W}$.  More generally a word is \emph{pure} if it is $\alpha$-pure for some $\alpha$.  The empty word $E$ is $\alpha$-pure for every $\alpha$.  Define the group $\Co_{\kappa}$ to be the quotient of $\Ho_{\kappa}$ by the smallest normal subgroup including the set of $\sim$ equivalence classes of pure words.

We work towards the proof that $\Co_{\kappa} \simeq \pi_1(\GS_{\kappa}, \circ_{\kappa})$.  Recall that the Hawaiian earring $\Ea \times \{0\}$ is a subspace of $\GS_1$.  Each copy $X_{\alpha}$ of $\GS_1$ which appears in the wedge $\GS_{\kappa}$ therefore has such a copy of the Hawaiian earring, which we denote $E_{\alpha}$, at its ``base.''  Let $\Ea_{\kappa}$ denote the union of all of these copies $E_{\alpha}$ of the Hawaiian earring.

In \cite{CC} is a description of an isomorphism of $\Ho_1$ with the fundamental group of the Hawaiian earring $\pi_1(\Ea_1, \circ_1)$, which we give and generalize here.  Let $\mathcal{I}$ denote the set of maximal open intervals in the closed interval $[0, 1]$ minus the Cantor ternary set.  The natural ordering on $\mathcal{I}$ is order isomorphic to that of the rationals, and so every countable order type embeds in $\mathcal{I}$.  For each $n\in \omega$ let $L_n$ be a loop based at $\circ_1$ which passes exactly once around the circle $C((0, \frac{1}{n+3}), \frac{1}{n+3})$ and is injective except at $0$ and $1$.  Given a word $W\in \W_1$ we let $\iota: \overline{W} \rightarrow \mathcal{I}$ be an order embedding.  Let $\R_{\iota}(W): [0, 1] \rightarrow \Ea_1$ be the loop given by

\[
\R_{\iota}(W)(t) = \left\{
\begin{array}{ll}
L_n(\frac{t - \inf I}{\sup I - \inf I})
                                            & \text{if } W(i) = a_{0, n}\text{ and }t\in I = \iota(i), \\
L_n^{-1}(\frac{t - \inf I}{\sup I - \inf I})                                           & \text{if } W(i) = a_{0, n}^{-1} \text{ and }t\in I = \iota(i), \\

\circ_{1} & \text{otherwise}.
\end{array}
\right.
\]

If $\iota_0: \overline{W} \rightarrow \mathcal{I}$ is a distinct order embedding, then $\R_{\iota}(W)$ and $\R_{\iota_0}(W)$ are homotopic via a straightforward homotopy whose image lies inside the common image $\R_{\iota}(W)([0, 1]) = \R_{\iota_0}(W)([0, 1])$.  Thus we have a well defined map $\R: \W\rightarrow \pi_1(\Ea_1, \circ_1)$.  Less obvious is the fact that $W \sim U$ implies $\R(W) = \R(U)$, so that $\R$ descends to a map, which we also name $\R$, from $\Ho_1$ to $\pi_1(\Ea_1, \circ_1)$ which is in fact an isomorphism.  Each loop at $\circ_1$, moreover, can be homotoped in its image to a loop which is precisely $\R_{\iota}(W)$ for some $\iota$ and $W$.

We'll use these facts to produce such a map $\R$ for larger values of $\kappa$.  To simplify the work we introduce the notion of reduced words.  As is the case with finitary words, there is a notion of reducedness for words in $\W_{\kappa}$.  We say $W \in \W_{\kappa}$ is \emph{reduced} if $W \equiv W_0W_1W_2$ and $W_1 \sim E$ implies $W_1\equiv E$.  We state the following, whose proof would follow in precisely the same way as that of \cite[Theorem 1.4, Corollary 1.7]{E}.

\begin{lemma}\label{reduced}  Given $W\in \W_{\kappa}$ there exists a reduced word $W_0 \in \W_{\kappa}$ such that $[W] = [W_0]$ and this $W_0$ is unique up to $\equiv$.  Moreover letting $W$ and $U$ be reduced there exist unique words $W_0, W_1, U_0, U_1$ such that 

\begin{enumerate}  \item $W \equiv W_0W_1$;

\item $U \equiv U_0U_1$;

\item $W_1\equiv U_0^{-1}$;

\item $W_0U_1$ is reduced.
\end{enumerate}

\end{lemma}

Let $\Red_{\kappa}$ denote the set of reduced words in $\W_{\kappa}$ and for each $W\in \W_{\kappa}$ let $\Red(W)$ be the reduced word such that $W \sim \Red(W)$.

\begin{lemma}\label{reducedconsequence}  Given $W \in \W_{\kappa}$ and $U\in \W_{\kappa}$ we have $\Red(WU)\equiv \Red(\Red(W)\Red(U))$.  Similarly, given $W_0, W_1, W_2 \in \W_{\kappa}$ we have $\Red(W_0W_1W_2) \equiv \Red(W_0\Red(W_1W_2)) \equiv \Red(\Red(W_0W_1)W_2)$.
\end{lemma}

\begin{proof}  Since $W\sim \Red(W)$ and $U\sim \Red(U)$ we have $WU \sim \Red(W)\Red(U)$ and by the uniqueness of the reduced word in its $\sim$ class we see that $\Red(WU) \equiv \Red(\Red(W)\Red(U))$.  The claim in the second sentence follows along the same lines.
\end{proof}

Lemma \ref{reducedconsequence} implies the group $\Ho_{\kappa}$ is isomorphic to the set $\Red_{\kappa}$ under the group operation $W*U =\Red(WU)$.  We give the following definition (see \cite[Definition 3.4]{CC}):

\begin{definition}\label{cancellation}  Given a word $W\in \W_{\kappa}$ we say $\mathcal{S} \subseteq \overline{W} \times\overline{W}$ is a \emph{cancellation} provided

\begin{enumerate}\item for $\langle i_0, i_1 \rangle \in \mathcal{S}$ we have $i_0 < i_1$;

\item  if $\langle i_0, i_1\rangle \in \mathcal{S}$ and $\langle i_0, i_2\rangle \in \mathcal{S}$ then $i_2 = i_1$;

\item  if $\langle i_0, i_1\rangle \in \mathcal{S}$ and $\langle i_2, i_1\rangle\in \mathcal{S}$ then $i_2 = i_0$;

\item  if $\langle i_0, i_1 \rangle \in \mathcal{S}$ and $i_2\in (i_0, i_1) \subseteq \overline{W}$ there exists $i_3\in (i_0, i_1)$ such that either $\langle i_2, i_3\rangle \in \mathcal{S}$ or $\langle i_3, i_2\rangle \in \mathcal{S}$;

\item  if $\langle i_0, i_1\rangle \in \mathcal{S}$ then $W(i_0) = (W(i_1))^{-1}$.

\end{enumerate}

\end{definition}

The $\langle \cdot , \cdot \rangle$ notation for ordered pairs is used here in order to avoid confusion with parenthetical notation $( \cdot, \cdot)$ which can be interpreted as an open interval.  We shall use $\langle \cdot \rangle$ to denote a generated subgroup, and the lack of a comma makes this use unambiguous.

A cancellation may be understood as a transfinite strategy for freely reducing a word.  Conditions (2) and (3) imply that a cancellation is a pairing of elements in a subset of elements of $\overline{W}$.  Condition (5) says that the pairing requires the associated letters in $W$ to be inverses of each other.  Condition (4) requires the pairing to be complete in the sense that each element between paired elements must also be paired by $\mathcal{S}$.  Condition (4) also requires that the pairing is noncrossing in the sense that if an element $i$ lies between two paired elements $i_0$ and $i_1$, then the element with which $i$ is paired must also be between $i_0$ and $i_1$.

Zorn's Lemma implies that each cancellation $\mathcal{S}$ in a word $W$ is included in a maximal cancellation $\mathcal{S}'$; that is, $\mathcal{S} \subseteq \mathcal{S}'$ and $\mathcal{S}'$ is not a proper subset of a cancellation in $W$.  It turns out that a maximal cancellation reveals the reduced word representative, as happens with freely reducing a finitary word until free reductions are no longer possible.  We omit the proof of the following, but it follows in precisely the same manner as \cite[Theorem 3.9]{CC}:

\begin{lemma}\label{cancellationreduces}  If $\mathcal{S}$ is a maximal cancellation for $W\in \W_{\kappa}$ then 
\begin{center}
$W\upharpoonright \{i\in \overline{W}\mid (\neg \exists i')(\langle i, i'\rangle \in \mathcal{S} \text{ or }\langle i, i'\rangle \in \mathcal{S})\} \equiv\Red(W)$.
\end{center}

Thus a word has only trivial cancellation if and only if that word is reduced.  As a consequence, if $W \in \W_{\kappa}$ with $W \equiv \prod_{\lambda \in \Lambda} W_{\lambda}$ then $\Red(W) \equiv \Red(\prod_{\lambda \in \Lambda}\Red(W_{\lambda}))$.
\end{lemma}

Now we define our homomorphism from $\Red_{\kappa}$ to $\pi_1(\Ea_{\kappa}, \circ_{\kappa})$.  For each $\alpha < \kappa$ and $n<\omega$ we let $L_{\alpha, n}$ be a loop based at $\circ_{\kappa}$ which goes exactly once around the $n$-th circle of $E_{\alpha}$ and is injective except at $0, 1$.  One can use an isometry between $\Ea_1$ and $E_{\alpha}$ to define $L_{\alpha, n}$ from $L_n$ if wished.  Given a reduced word $W\in \Red_{\kappa}$ and an order embedding $\iota: \overline{W} \rightarrow \mathcal{I}$ we get a loop $\R_{\iota}(W)$ defined by

\[
\R_{\iota}(W)(t) = \left\{
\begin{array}{ll}
L_{\alpha, n}(\frac{t - \inf I}{\sup I - \inf I})
                                            & \text{if } W(i) = a_{0, n}\text{ and }t\in I = \iota(i), \\
L_{\alpha, n}^{-1}(\frac{t - \inf I}{\sup I - \inf I})                                           & \text{if } W(i) = a_{0, n}^{-1} \text{ and }t\in I = \iota(i), \\

\circ_{\kappa} & \text{otherwise}.
\end{array}
\right.
\]

\noindent The check that this function on $[0, 1]$ is continuous is straightforward.  Given some other order embedding $\iota_0: \overline{W} \rightarrow \mathcal{I}$ we obtain a different loop $\R_{\iota_0}$ which is homotopic to $\R_{\iota}$ via a homotopy which is a reparametrization.  Explicitly, letting 

\begin{center}
$j_{\min}(s)(i) = s\inf \iota(i) + (1-s)\iota_0(s)$
\end{center}

\noindent and 
\begin{center}
$j_{\max}(s)(i) = s\sup\iota(i) + (1-s)\sup\iota_0(i)$
\end{center}

\noindent a homotopy $H: [0, 1]\times [0,1] \rightarrow \GS_{\kappa}$ is given by $H(t, s) =$

\[
\left\{
\begin{array}{ll}
L_{\alpha, n}(\frac{t - j_{\min}(s)(i)}{j_{\max}(s)(i) - j_{\min}(s)(i)})
                                            & \text{if } W(i) = a_{\alpha, n}\text{ and }t\in (j_{\max}(s)(i), j_{\min}(s)(i)), \\
L_{\alpha, n}^{-1}(\frac{t - j_{\min}(s)(i)}{j_{\max}(s)(i) - j_{\min}(s)(i)})                                           & \text{if } W(i) = a_{\alpha, n}^{-1} \text{ and }t\in (j_{\max}(s)(i), j_{\min}(s)(i)), \\

\circ_{\kappa} & \text{otherwise}.
\end{array}
\right.
\]

In particular we have a well-defined map $\R: \Red_{\kappa} \rightarrow \pi_1(\Ea_{\kappa}, \circ_{\kappa})$.  To see that this is a homomorphism, we let $W, U \in \Red_{\kappa}$ and let $W_0, W_1, U_0, U_1$ be as in Lemma \ref{reduced}.  The loop $\R(W_1)$ is readily seen to be the inverse of $\R(U_2)$.  The word $W_0U_1$ is reduced and therefore we have

$$
\begin{array}{ll}
\R(W*U) & = \R(\Red(WU))\\
& = \R(W_0U_1)\\
& = \R(W_0)\R(U_0)^{-1}\R(U_0)\R(U_1)\\
& = \R(W_0)\R(W_1)\R(U_0)\R(U_1)\\
& = \R(W_0W_1)\R(U_0U_1)\\
& = \R(W)\R(U).
\end{array}
$$

Suppose now that $W \in \Red_{\kappa}$ is in the kernel of $\R$. Suppose for contradiction that $W \not\equiv E$.  We'll construct a cancellation $\mathcal{S}$ of $W$ to obtain a contradiction.  Fix an order embedding $\iota: \overline{W} \rightarrow \mathcal{I}$.  Let $H: [0, 1] \times [0, 1] \rightarrow \Ea_{\kappa}$ be a nullhomotopy of $\R_{\iota}(W)$.  That is, $H(t, 0) = \R_{\iota}(W)(t)$ and $H(0, s) = H(1, s) = H(t, 1)$ for all $t, s\in [0, 1]$.  For each $I\in \mathcal{I}$ we let $m(I)$ signify the midpoint $m(I) = \frac{\sup I + \inf I}{2}$.  Consider the set of points $M = \{(m(\iota(i)), 0)\}_{i\in\overline{W}} \subseteq [0, 1]\times[0, 1]$.  For each point $p\in M$ we consider its path component $P_p$ in $[0, 1]\times [0, 1] \setminus H^{-1}(\circ_{\kappa})$.  Each $p\in M$ is associated with a unique interval $\iota(i_p)$ and therefore with a unique element $i_p\in \overline{W}$, and each $i\in \overline{W}$ is in turn associated with a unique point $p\in M$.  Moreover, the natural order on points in $M$ is isomorphic with the elements of $\overline{W}$ in this association.

Fixing $p\in M$ the set $P_p \cap M$ is necessarily finite, because each element of $P_p \cap M$ corresponds to exactly one occurrence of a loop $L_{\alpha, n}$ or of its inverse, for a fixed $\alpha$ and $n$, and there are only finitely many such occurrences since there are finitely many occurences of $a_{\alpha, n}^{\pm 1}$ in $W$.  Write $P_p \cap M = \{p_0, p_1, \ldots, p_j\}$ listing elements in the natural order.  By modifying $H$ to have output $\circ_{\kappa}$ outside of $P_p$, we see that $H$ witnesses a nulhomotopy of the loop $\R_{i}(W\upharpoonright\{i_{p_0}, \ldots, i_{p_j}\})$, which lies entirely in the $n$-th circle of $E_{\alpha}$.  Then there are exactly as many $i_{p_k}$ for which $W(i_{p_k}) = a_{\alpha, n}$ as there are for which $W(i_{p_k}) = a_{\alpha, n}^{-1}$.  Select neighboring points $p_k, p_{k+1}$ which are of opposite parity and let $\langle i_{p_k}, i_{p_{k+1}}\rangle \in \mathcal{S}$.  Among the remaining points $P_p\cap M \setminus \{p_k, p_{k+1}\}$ select two which are neighboring under the new order and add this ordered pair to $\mathcal{S}$.  Continue in this way until all elements of $P_p\cap M$ are used.  Perform this procedure on all path components $P_p$ for $p\in M$.  It is straightforward to check that $\mathcal{S}$ satisfies the rules of a cancellation.  We have obtained our contradiction.  Thus $\R$ is an injection.

We check that $\R$ is a surjection.  Let $L: [0, 1] \rightarrow \Ea_{\kappa}$ be a loop at $\circ_{\kappa}$.  Let $\mathcal{J}$ be the set of maximal open intervals in $[0, 1] \setminus L^{-1}(\circ_{\kappa})$.  This set is countable and has a natural ordering.  For each restriction $L\upharpoonright \overline{J}$, where $J \in \mathcal{J}$, there is a homotopy $H_J: \overline{J} \times [0, 1] \rightarrow L(\overline{J})$ to a loop $L_J: \overline{J} \rightarrow L(\overline{J})$ which is either constant, or $L_{\alpha, n}(\frac{t-\inf J}{\sup J - \inf J})$ or $L_{\alpha, n}(\frac{t-\inf J}{\sup J - \inf J})$.  By gluing these homotopies together we get a homotopy of $L$ to a loop whose restriction to each nonconstant interval $\overline{J}$ is of the form $L_{\alpha, n}(\frac{t-\inf J}{\sup J - \inf J})$ or $L_{\alpha, n}(\frac{t-\inf J}{\sup J - \inf J})$.

Thus assuming $L$ is of this form, we define a word $W: \mathcal{J} \rightarrow \Ao_{\kappa}$ by letting $W(J) = a_{\alpha, n}^{\pm 1}$ where the $\alpha$, $n$ and superscript are determined in the straightforward way.  That the mapping $W$ is indeed a word (no $n$ in the subscript occurs infinitely often) follows from the fact that $L$ is continuous.  Let $\mathcal{S}$ be a maximal cancellation on $W$.  This $\mathcal{S}$ can be used to homotope $L$ so that the new associated word is $\Red(W)$.  More explicitly, we define $H: [0, 1] \times [0, 1] \rightarrow \Ea_{\kappa}$ by having $H(t, s) = L(t)$ if $t$ does not lie inside an interval $(\inf J_0, \sup J_1)$ where $\langle J_0, J_1\rangle \in \mathcal{S}$.  If a point $(t, s) \in [0, 1]\times [0, 1]$ lies on the semicircle determined by points $(t_0, 0)$ and $(t_1, 0)$ which is perpendicular to $[0, 1]\times \{0\}$ where $t_0\in J_0$, $t_1\in J_1$ and $\langle J_0, J_1\rangle \in \mathcal{S}$ with $L(t_0) = L(t_1)$ we let $H(t, s) = L(t_0) = L(t_1)$.  Give $H$ output $\circ_{\kappa}$ everywhere else.  That $H$ is continuous and produces a loop $H(t, 1)$ as described is intuitive but tedious to check.  Thus we may now assume that the associated word $W$ is reduced.  By reparametrizing $L$ we may make it so that all the intervals in $\mathcal{J}$ are elements in $\mathcal{I}$, which immediately gives an order embedding $\iota$ of $\overline{W}$ to $\mathcal{I}$ for which $L = \R_{\iota}(W)$.  We have shown surjectivity and finished the proof of the following:

\begin{lemma}\label{Hawaiianpart}  The function $\R: \Red_{\kappa}\rightarrow \pi_1(\Ea_{\kappa}, \circ_{\kappa})$ is an isomorphism.
\end{lemma}

We now approach the isomorphism $\Co_{\kappa} \simeq \pi_1(\GS_{\kappa}, \circ_{\kappa})$.  For finite values of $\kappa$ this can be done by a straighforward argument in which van Kampen's Theorem is iterated finitely many times, as is done in \cite[Section 4]{EF}.  We present an argument which works for every cardinal $\kappa$.

\begin{lemma}\label{smallloops}  Given $\epsilon>0$ and a loop $L: [0, 1] \rightarrow \GS_{\kappa}$ based at $\circ_{\kappa}$ there is a loop homotopic to $L$ whose image is of diameter at most $\epsilon$.
\end{lemma}

\begin{proof}  Let $\mathcal{J}$ be the set of maximal open intervals in $[0, 1] \setminus L^{-1}(\circ_{\kappa})$.  There are only finitely many intervals $J \in \mathcal{J}$ for which the diameter of the image $\diam(L\upharpoonright J)$ is at least $\epsilon/2$.  But for every $J\in \mathcal{J}$ the loop $L\upharpoonright \overline{J}$ lies entirely in a contractible space, a homeomorph of $\GS_1$.  In particular each restriction $L\upharpoonright \overline{J}$ is nulhomotopic.  Thus letting $\mathcal{J}' \subseteq \mathcal{J}$ the set of those intervals whose images are of diameter $\geq \epsilon/2$ we have $L$ homotopic to the loop $L': [0, 1]\rightarrow \GS_{\kappa}$ given by

\[
L'(t) = \left\{
\begin{array}{ll}
L(t)
                                            & \text{if } t\notin \bigcup \mathcal{J}', \\
\circ_{\kappa}                                           & \text{if } t\in\bigcup \mathcal{J}'.
\end{array}
\right.
\]

\noindent which has diameter at most $\epsilon$.
\end{proof}

\begin{lemma}\label{smalldeformation}  The space $\GS_1 \setminus \{(0, 0, 1)\}$ strongly deformation retracts to $\Ea_1$.
\end{lemma}

\begin{proof}  We recall that $\GS_1$ is homeomorphic to the quotient space of $\Ea \times [0, 1]$ which identifies points whose third coordinate is $1$.  Under this homeomorphism the point $(0, 0, 1)$ is mapped to the identified point whose third coordinate is $1$.  Letting $h: (\GS_1\setminus \{(0, 0, 1)\}) \times [0, 1] \rightarrow \GS_1$ be given by $((x, y, z), s)\mapsto (x, y, (1-s)z)$ it is easy to see that $h$ is a strong deformation retraction of $\GS_1$ to $\Ea\times \{0\}$.
\end{proof}

Let each copy of $(0, 0, 1)$ in the copies of $\GS_1$ whose wedge forms $\GS_{\kappa}$ be called a ``cone tip.''  Let $\GS_{\kappa}'$ denote the space $\GS_{\kappa}$ minus the set of cone tips.

\begin{lemma}\label{bigdeformation}  The space $\GS_{\kappa}'$ strongly deformation retracts to $\Ea_{\kappa}$. 
\end{lemma}

\begin{proof}  Let $h_{\alpha}: X_{\alpha} \times [0, 1] \rightarrow X_{\alpha}$ be the homotopy given by Lemma \ref{smalldeformation} on each isometric copy $X_{\alpha}$ of $\GS_{1}$ whose union gives $\GS_{\kappa}$.  Let $H: \GS_{\kappa}'\times [0, 1] \rightarrow \GS_{\kappa}'$  be given by

\[
H(p, s) = \left\{
\begin{array}{ll}
h_{\alpha}(p, s)
                                            & \text{if } p\in X_{\alpha}\setminus\{\circ_{\kappa}\}, \\
\circ_{\kappa}                                        & \text{if } p = \circ_{\kappa}.
\end{array}
\right.
\]

\noindent This map $H$ is a strong deformation retraction to $\Ea_{\kappa}$.
\end{proof}

\begin{lemma}\label{thingsareinE}  Each loop in $\GS_{\kappa}$ based at $\circ_{\kappa}$ is homotopic to a loop in $\Ea_{\kappa}$.  In particular the inclusion map $\Ea_{\kappa} \rightarrow \GS_{\kappa}$ induces an onto homomorphism of fundamental groups.
\end{lemma}

\begin{proof}  Letting $L$ be a loop in $\GS_{\kappa}$ based at $\circ_{\kappa}$ we homotope $L$ to a loop $L'$ which is of diameter $1/2$ by Lemma \ref{smallloops}.  This $L'$ lies in $\GS_{\kappa}'$  and so by Lemma \ref{bigdeformation} we can homotope $L'$ to have image in $\Ea_{\kappa}$.
\end{proof}

\begin{theorem}\label{conegroupdoesthetrick}  The isomorphism $\R: \Red_{\kappa} \rightarrow \pi_1(\Ea_{\kappa}, \circ_{\kappa})$ descends to an isomorphism $\R_{\Co_{\kappa}}: \Co_{\kappa} \rightarrow \pi_1(\GS_{\kappa}, \circ_{\kappa})$.
\end{theorem}

\begin{proof}  We have by Lemma \ref{thingsareinE} that the inclusion $\Ea_{\alpha} \rightarrow \GS_{\kappa}$ induces a surjection $\pi_1(\Ea_{\kappa}, \circ_{\kappa}) \rightarrow \pi_1(\GS_{\kappa}, \circ_{\kappa})$.  Thus by composing with $\R$ we obtain an epimorphism $\R': \Red_{\kappa} \rightarrow \pi_1(\GS_{\kappa}, \circ_{\kappa})$.  Moreover each pure word $W$ maps to a loop which is contained entirely in a copy of $\GS_1$ and is therefore in the kernel.  Then $\R'$ descends to an epimorphism $\R_{\Co_{\kappa}}:\Co_{\kappa} \rightarrow \pi_1(\GS_{\kappa}, \circ_{\kappa})$.  We shall be done when we show that $\R_{\Co_{\kappa}}$ has trivial kernel.

Suppose that $W$ is in the kernel of $\R'$.  Fix an order injection $\iota: \overline{W} \rightarrow \mathcal{I}$ and let $\R_{\iota}(W): [0, 1]\rightarrow \Ea_{\kappa}$ be the corresponding loop.  Let $H: [0,1]\times[0, 1] \rightarrow \GS_{\kappa}$ be a nulhomotopy.  That is, $H(t, 0) = \R_{\iota}(W)(t)$, $H(0, s) = H(1, s) = H(t, 1) = \circ_{\kappa}$ for all $t, s\in [0, 1]$.  For each $I\in \mathcal{I}$ we again let $m(I)$ be the midpoint $m(I) = \frac{\sup I + \inf I}{2}$ and $M = \{(m(\iota(i)), 0)\}_{i \in \overline{W}} \subseteq [0, 1]\times [0, 1]$.  For $p\in M$ let $P_p$ signify the path component of $p$ in $[0, 1]\times [0, 1]\setminus H^{-1}(\circ_{\kappa})$.

We claim that there are only finitely many path components $P_{p_0}, \ldots, P_{p_j}$ for which there exists a point $z\in P_{p_m}$ such that $H(z)$ is a cone tip.  Supposing this is false, we obtain by compactness of $[0, 1]\times[0, 1]$ a sequence of points $\{z_m\}_{m\in \omega}$ for which each $H(z_m)$ is a cone tip, each $z_m$ is in a distinct path component $P_{p_m}$ and the $z_m$ converge to a point $z\in [0, 1]\times [0, 1]$.  Let $\rho: [0, 1] \rightarrow [0, 1] \times [0, 1]$ be a function such that $\rho\upharpoonright [1-\frac{1}{m+1}, 1 - \frac{1}{m+2}]$ follows the geodesic from $z_m$ to $z_{m+1}$ and $\rho(1) = z$.   Such a function is obviously continuous.  However $H \circ \rho$ is not continuous at the point $1$, for there are points $t$ arbitrarily close to $1$ for which $H(\rho(t))$ is a cone tip and there are $t$ arbitrarily close to $1$ for which $H(\rho(t)) = \circ_{\kappa}$, a contradiction.

We next notice that for each of these finitely many path components $P_{p_m}$ including a point which maps under $H$ to a cone tip that all elements of $P_{p_m}\cap M$ map into the same cone $X_{\alpha_m}$.  This is clear since any two points in $P_{p_m}\cap M$ are joined by a path which avoids $H^{-1}(\circ_{\kappa})$, and so their images under $H$ are joined by a path which avoids $\circ_{\kappa}$.  In particular their images lie in the same cone.

Next, for each path component $P_{p_m}$ which includes a point which maps under $H$ to a cone tip there exist finitely many intervals $\Lambda_{m, 0}, \Lambda_{m, 1}, \ldots, \Lambda_{m, {j_m}}$ in $\overline{W}$ such that $m(\iota(i))\in P_{p_m}$ if and only if $i\in \Lambda_{m, n}$ for some $0\leq n\leq j_m$.  Were this not the case, there would exist a nonempty interval $\Lambda' \subseteq \overline{W}$ for which all $i\in \Lambda'$ are such that $P_{m(\iota(i))}$ does not contain a point mapping under $H$ to a cone tip and such that any interval properly including $\Lambda'$ contains an $i$ for which $m(\iota(i)) \in P_{p_m}$.  This follows from the fact that there are only finitely many path components $P_{p_0}, \ldots, P_{p_j}$ which contain a point mapping under $H$ to a cone tip.  The map $H$ witnesses that $\R(W\upharpoonright \Lambda')$ is nulhomotopic in $\GS_{\kappa}'$.  Thus by Lemma \ref{bigdeformation} we know $\R(W\upharpoonright \Lambda')$ is nulhomotopic in $\Ea_{\kappa}$.  By Lemma \ref{Hawaiianpart} we therefore have $W \upharpoonright \Lambda' \equiv E$, contrary to $\Lambda'$ being a nonempty interval.

Finally, we write $W \equiv W_0W_1\cdots W_l$ as the decomposition of $W$ such that each $\overline{W_q}$ is one of the intervals $\Lambda_{m, n}$ or is a maximal interval not intersecting any of the $\Lambda_{m, n}$.  Let $q_0<q_1< \cdots< q_r$ be the subscripts for which $\overline{W_{q_d}}$ is not a $\Lambda_{m, n}$.  The function

\[
H'(t, s) = \left\{
\begin{array}{ll}
H(t, s)
                                            & \text{if } (t, s) \notin \bigcup_{m=0}^jP_{p_j}, \\
\circ_{\kappa}                                        & \text{otherwise}.
\end{array}
\right.
\]

\noindent witnesses a nulhomotopy of the concatenation of loops $\R(W_{q_0})\R(W_{q_1})\cdots \R(W_{q_r})$ taking place entirely inside of $\GS_{\kappa}'$.  Thus $\R(W_{q_0})\R(W_{q_1})\cdots \R(W_{q_r})$ is nulhomotopic in $\Ea_{\kappa}$ by Lemma \ref{bigdeformation}, and by Lemma \ref{Hawaiianpart} we know that in fact $\Red(W_{q_0}\cdots W_{q_r}) = E$.  Thus by deleting finitely many intervals of $\overline{W}$ over each of which the letters have the same first coordinate we get a word which reduces to $E$.  Then $W$ is in the kernel of $\Red_{\kappa} \rightarrow \Co_{\kappa}$ and we are done.
\end{proof}

The above proof immediately gives us the following (cf. \cite[Theorem 8.1]{BZ}):

\begin{corollary}\label{deletesomewords}  A reduced word $W$ is in the kernel of the map $\Red_{\kappa}\rightarrow \Co_{\kappa}$ if and only if there exist finitely many intervals $I_0, \ldots, I_p$ such that $W\upharpoonright I_j$ is pure for each $j$ and $\Red(W\upharpoonright(\overline{W} \setminus \bigcup_{j=0}^pI_j)) = E$.
\end{corollary}

\begin{lemma}\label{foranontrivialelement}  Suppose that we have a word $V \equiv \prod_{n \in \omega} V_n$ with $V \in \Red_{\kappa}$ and

\begin{enumerate}

\item any interval $I \subseteq \overline{V}$ such that $V \upharpoonright I$ is pure is a subinterval of $\overline{\prod_{n = 0}^m V_n}$ for some $m \in \omega$; and

\item for each $n \in \omega$ there exists $j_n \in \omega$ such that $|\{i \in \overline{V_n} \mid \pro_1(V_n(i)) = j_n\}| > \sum_{m \neq n} |\{i \in \overline{V_n} \mid \pro_1(V_n(i)) = j_n\}|$.

\end{enumerate}
\noindent Then $[[V]] \neq [[E]]$ in $\Co_{\kappa}$.

\end{lemma}

\begin{proof}  Suppose for contradiction that $[[V]] = [[E]]$, so by Corollary \ref{deletesomewords} we obtain a finite collection of intervals $I_0, \ldots, I_p$ in $\overline{V}$ such that $V \upharpoonright I_k$ is pure for each $0 \leq k \leq p$ and $\Red(V \upharpoonright (\overline{V} \setminus \bigcup_{k = 0}^p I_k)) = E$.  Let $\mathcal{S}$ be a maximal cancellation of $V \upharpoonright (\overline{V} \setminus \bigcup_{k = 0}^p I_k)$.  We know by (1) that $\bigcup_{k = 0}^p I_k \subseteq \overline{\prod_{n = 0}^m V_n}$ for some $m \in \omega$.  All elements of $Z = \{i \in \overline{V_{m + 1}} \mid \pro_1(V_{m + 1}(i)) = j_{m + 1}\}$ must participate in $\mathcal{S}$ since $\Red(V \upharpoonright (\overline{V} \setminus \bigcup_{k = 0}^p I_k)) = E$, but since $V_{m + 1}$ is reduced we know that the elements of $Z$ are paired with elements of $\overline{V} \setminus (\overline{V_{m + 1}} \cup \bigcup_{k = 0}^p I_k)$, but this is impossible by condition (2).
\end{proof}

For a reduced word $W$ we let $[[W]]$ denote the equivalence class of $W$ in $\Co_{\kappa}$ and if $[[W]]=[[U]]$ we write $W\approx U$.

\begin{theorem}\label{howmany}  For each cardinal $\kappa$ we have

\[
|\Co_{\kappa}| = \left\{
\begin{array}{ll}
1
                                            & \text{if } \kappa = 0, \\
\kappa^{\aleph_0}                                        & \text{if }\kappa \geq 1.
\end{array}
\right.
\]

\end{theorem}

\begin{proof}  We have already seen that the formula holds in case $\kappa =0, 1$.  Suppose $\kappa \geq 2$.  Notice that the space $\GS_{\kappa}$ has $2^{\aleph_0}\cdot \kappa = \max\{2^{\aleph_0}, \kappa\}$ points in it.  Every continuous function from $[0, 1]$ to the metric space $\GS_{\kappa}$ is totally determined by the restriction to $[0, 1]\cap \mathbb{Q}$.  Thus there are at most $(\max\{2^{\aleph_0}, \kappa\})^{\aleph_0} = \kappa^{\aleph_0}$ loops in the space, so in particular $|\Co_{\kappa}| \leq \kappa^{\aleph_0}$.  We must show $|\Co_{\kappa}| \geq \kappa^{\aleph_0}$.

If $2 \leq \kappa \leq 2^{\aleph_0}$ then let $\Sigma$ be a collection of infinite subsets of $\omega$ such that for distinct $X, Y \in \Sigma$ we have $X\cap Y$ finite and such that $|\Sigma| = 2^{\aleph_0}$.  Such a construction is straightforward, see for example
 \cite[II.1.3]{K}.  For each $X\in \Sigma$ let $X = \{n_{0, X}, n_{1, X}, \ldots\}$ be the enumeration of $X$ in the natural order.  Let $$W_X \equiv a_{0, n_{0, X}}a_{1, n_{1, X}}a_{0, n_{2, X}}a_{1, n_{3, X}}\cdots .$$  Since $W_X$ uses only positive letters it is clear that $W_X$ and also any deletion of finitely many letters of $W_X$ is a reduced word.  By the conditions on $\Sigma$ is is clear that $[[W_X]] \neq [[W_Y]]$ if $X\neq Y$.  Then $\kappa^{\aleph_0} \leq |\Co_{\kappa}|$.

Suppose that $2^{\aleph_0}<\kappa$ and that $\kappa^{\aleph_0} = \kappa$.  Let $f: \kappa \times \omega \rightarrow \kappa$ be an injection and for each $\alpha<\kappa$ we define $W_{\alpha} \equiv a_{f(\alpha, 0), 0}a_{f(\alpha, 1), 1}\cdots$.  It is clear that $[[W_{\alpha}]] \neq [[W_{\beta}]]$ for distinct $\alpha, \beta<\kappa$.

Suppose finally that $2^{\aleph_0}<\kappa$ and that $\kappa^{\aleph_0} > \kappa$.  Let $X$ be the set of all sequences from $\omega$ to $\kappa$ and consider two sequences $\sigma_0, \sigma_1\in X$ to be equivalent if they are eventually identical: for some $m \in \omega$ we have $\sigma_0(m + n) = \sigma_1(m + n)$ for all $n\in \omega$.  Each equivalence class is of cardinality $\kappa$, so there are exactly $\kappa^{\aleph_0}$ distinct equivalence classes.  Letting $Y \subset X$ be a selection from each equivalence class we define a map $Y \rightarrow \Co_{\kappa}$ by letting $\sigma \mapsto W_{\sigma}$ where $W_{\sigma} \equiv a_{f(\sigma(0), 0), 0}a_{f(\sigma(1), 1), 1}\cdots$ and again $f: \kappa \times \omega \rightarrow \kappa$ is an injection.  It is easy to see that for distinct elements of $Y$ the assigned words are not equivalent in $\Co_{\kappa}$.
\end{proof}

An interval $I$ in a totally ordered set $\Lambda$ is \emph{initial} if it is a union of intervals of the form $(-\infty, i]$ and is \emph{terminal} if a union of intervals of form $[i, \infty)$ (an initial or terminal interval may be empty).  Given a nonempty word $W\in \Red_{\kappa}$ there exists a unique maximal initial interval $I_0$ of $\overline{W}$ for which there exists a terminal interval $I_1 \subseteq \overline{W}$ such that $W\upharpoonright I_0\equiv (W\upharpoonright I_1)^{-1}$.  By the proof of \cite[Corollary 1.6]{E} the maximal such initial interval $I_0$ and the accompanying $I_1$ are disjoint and $\overline{W} \setminus (I_0 \cup I_1)$ is nonempty, and this set is clearly an interval, say $I_2$.  Thus $W \equiv (W\upharpoonright I_0) (W\upharpoonright I_2)  (W\upharpoonright I_0)^{-1}$ and we call the word $W\upharpoonright I_2$ the \emph{cyclic reduction} of $W$.  Clearly if $U$ is the cyclic reduction of $W$ then the cyclic reduction of $U$ is again $U$, so cyclic reduction is an idempotent operation.  A word whose cyclic reduction is itself is called \emph{cyclically reduced}.  It is clear from Lemma \ref{cancellationreduces} that word $U$ is cyclically reduced if and only if the word $U^n$ is reduced for all $n \geq 1$ if and only if $U^2$ is reduced.

\begin{proof}[Proof of Theorem \ref{continuous}] By Theorem \ref{howmany} we know that when $n = 1$ any homomorphism from $\Co_n$ to $\Co_{\kappa}$ has trivial image and is therefore of uncountable index.  Any homomorphism from $\Co_{\kappa}$ to $\Co_n$ is trivial and therefore has uncountable kernel by Theorem \ref{howmany}.  We may therefore assume $2 \leq n, \kappa$.  We will pause for some general discussion and a couple of lemmas, finally returning to finish our proof.

Suppose that $2\leq \kappa_0, \kappa_1$ and that $f: \GS_{\kappa_0} \rightarrow \GS_{\kappa_1}$ is continuous (no assumption on how $\kappa_0$ compares with $\kappa_1$ or whether either of $\kappa_0$, $\kappa_1$ is finite).  We notice that if $f(\circ_{\kappa_0}) \neq \circ_{\kappa_2}$ then the induced map is trivial.  This can be seen by letting $\delta = d(f(\circ_{\kappa_0}), \circ_{\kappa_1})$ and selecting $\epsilon>0$ such that $d(x, \circ_{\kappa_0})<\epsilon$ implies $d(f(x), f(\circ_{\kappa_1}))<\delta$.  Given any loop $L$ at $\circ_{\kappa_0}$ in $\GS_{\kappa_0}$ we can homotope $L$ to have diameter less than $\epsilon$ by Lemma \ref{smallloops}, and the image $f\circ L$ will lie entirely in a copy of the contractible space $\GS_1$, and therefore be trivial in $\pi_1(\GS_{\kappa_1})$.  Thus when proving either (1) or (2) we may without loss of generality assume that the wedge point of the domain is mapped by $f$ to the wedge point of the codomain.

Suppose again that $2\leq \kappa_0, \kappa_1$ without any assumptions on how $\kappa_0$ and $\kappa_1$ compare or whether either is finite.  Also suppose we have a continuous function $f: \GS_{\kappa_0} \rightarrow \GS_{\kappa_1}$ with $f(\circ_{\kappa_0}) = \circ_{\kappa_1}$.  Select $\epsilon>0$ such that $d(x, \circ_{\kappa_0})<\epsilon$ implies $d(f(x), \circ_{\kappa_1})<1$.  Select $N \in\omega$ large enough that the circle $C((0, \frac{1}{N+3}),\frac{1}{N+3})$ is of diameter less than $\epsilon$.  For each $\alpha<\kappa_0$ we let $E_{N , \alpha}\leq E_{\alpha}$ be the union of all circles $C((0, \frac{1}{n+3}),\frac{1}{n+3})$ for $n\geq N$ in the copy of the Hawaiian earring $E_{\alpha}\subseteq \E_{\kappa_0}$.  Let $\Ea_{\kappa_0, N} = \bigcup_{\alpha<\gamma} E_{N, \alpha}$.  The image $f(\Ea_{\kappa_0, N})$ has trivial intersection with the cone tips of $\GS_{\kappa_1}$, so by Lemma \ref{bigdeformation} the restriction map $f\upharpoonright \Ea_{\kappa_0, N}$ can be homotoped to a map $g_1:\Ea_{\kappa_0, N} \rightarrow \Ea_{\kappa_1}$.  Extend $g_1:\Ea_{\kappa_0, N} \rightarrow \Ea_{\kappa_1}$ to a map $g: \Ea_{\kappa_0} \rightarrow \Ea_{\kappa_1}$ by letting all circles $C((0, \frac{1}{n+3}), \frac{1}{n+3})$ with $n<N$ in each $E_{\alpha} \subseteq \Ea_{\kappa_0}$ map to $\circ_{\kappa_1}$.

Now it is clear that the map $g: \Ea_{\kappa_0} \rightarrow \Ea_{\kappa_1}$ satisfies $(\iota_{\kappa_1} \circ g)_* = (f \circ \iota_{\kappa_0})_*$ where $\iota_0: \Ea_{\kappa_0} \rightarrow \GS_{\kappa_0}$ is the inclusion map and similarly for $\iota_1$.  From the isomorphisms $\R_{\kappa_0}: \Red_{\kappa_0} \rightarrow \pi_1(\Ea_{\kappa_0}, \circ_{\kappa_0})$ and $\R_{\kappa_1}: \Red_{\kappa_1} \rightarrow \pi_1(\Ea_{\kappa_1}, \circ_{\kappa_1})$ we obtain a homomorphism $h: \Red_{\kappa_0} \rightarrow \Red_{\kappa_1}$ defined by $g_* \circ \R_{\kappa_0} = \R_{\kappa_1} \circ h$.  Because $g$ is continuous we have that if $W \in \W_{\kappa_0}$ with $W \equiv \prod_{\lambda \in \Lambda} W_{\lambda}$ then $\prod_{\lambda \in \Lambda}h(\Red(W_{\lambda})) \in \W_{\kappa_1}$ and $h(\Red(W)) \equiv \Red(\prod_{\lambda \in \Lambda}h(\Red(W_{\lambda})))$.

\begin{lemma}\label{eventualpurity}  For each $\alpha < \kappa_0$ there exists $N_{\alpha} \in \omega$ such that if $d(W) \geq N_{\alpha}$ and $W$ is $\alpha$-pure then $h(W)$ is pure.
\end{lemma}

\begin{proof}  Suppose that the claim is false.  Select $\alpha < \kappa$ and sequence of $\alpha$-pure words $\{W_n \}_{n \in \omega} \subseteq \Red_{\kappa_0}$ such that $d(W_n) \rightarrow \infty$ and for each $n \in \omega$ we have $h(W_n)$ not pure.  By continuity of $g$ we know that $d(h(W_n)) \rightarrow \infty$ as well.  We inductively define sequences $\{n_k\}_{k \in \omega}$, $\{m_k\}_{k \in \omega}$, and $\{j_k\}_{k \in \omega}$ of natural numbers.

Let $n_0 = m_0 = 1$ and select $j_0$ such that $h(W_{n_0})$ has a letter with second subscript $j_0$.  Suppose that we have already defined $n_0, \ldots, n_k$ and $m_0, \ldots, m_k$ and $j_0, \ldots, j_k$.  We know that $h(W_{n_k})$ is not pure, so it has letters $a_{\alpha_0, l_0}^{\pm 1}$ and $a_{\alpha_1, l_1}^{\pm 1}$ where $\alpha_0, \alpha_1$ are distinct ordinals below $\kappa_1$.  Pick $n_{k + 1}$ large enough that $d(h(W_{n_{k + 1}})) > l_0, l_1, j_k$.  Since $h(W_{n_{k + 1}})$ is nontrivial it has a nontrivial cyclic reduction $U_{1, k}$.  Select $j_{k + 1}$ such that $U_{1, k}$ has a letter whose second subscript is $j_{k + 1}$.  Select $m_{k + 1}$ large enough that $2 + 2 \sum_{r = 0}^k m_r |\{i \in \overline{h(W_{n_r})} \mid \pro_1(h(W_{n_r})(i)) = j_{k + 1}\}| < m_{k + 1}$.

Let $U_k \equiv \Red((h(W_{n_k}))^{m_k}) \equiv U_{0, k}U_{1, k}^{m_k}U_{0, k}^{-1}$, where $U_{1, k}$ is the cyclic reduction of $h(W_{n_k})$ and $h(W_{n_k}) \equiv U_{0, k}U_{1, k}U_{0, k}^{-1}$.  Notice that the concatenation $U \equiv \prod_{k \in \omega} U_k$ is a word in $\W_{\kappa_1}$.  Moreover $\Red(U) = h(\Red(\prod_{k \in \omega}W_{n_k}^{m_k}))$ by continuity of $g$ and how $h$ is defined.  Let $\mathcal{S}$ be a maximal cancellation of $U$.  Since each $U_k$ is reduced, $\mathcal{S}$ cannot pair elements of $\overline{U_k}\subseteq \overline{U}$ with elements in $\overline{U_k}$.  Moreover

$$
\begin{array}{ll}
|\{i \in \overline{U_{1, k}^{m_k}} \mid \pro_1(U_{1, k}^{m_k}(i)) = j_k\}| & \geq m_k|\{i \in \overline{U_{1, k}} \mid \pro_1(U_{1, k}(i)) = j_k\}|\\
& \geq m_k;
\end{array}
$$

$$
\begin{array}{ll}
|\{ i \in \overline{\prod_{q = k + 1}^{\infty} U_q} \mid \pro_1((\prod_{q = k + 1}^{\infty} U_q)(i)) = j_k\}|& = 0;\text{ and}
\end{array}
$$

\begin{center}

$|\{ i \in \overline{\prod_{q = 0}^{k - 1} U_q} \mid \pro_1((\prod_{q = 0}^{k - 1} U_q)(i)) = j_k\}|$

$\leq \sum_{r = 0}^{k - 1} m_r |\{i \in \overline{h(W_{n_r})} \mid \pro_1(h(W_{n_r})(i)) = j_k\}|$

\end{center}

\noindent hold for each $k \in \omega$.  Thus for each $k \in \omega$ there is a (possibly empty) initial interval $I_k \subseteq \overline{U_k}$, nonempty interval $I_k' \subseteq \overline{U_k}$, and (possibly empty) terminal interval $I_k'' \subseteq \overline{U_k}$ such that $\overline{U_k} \equiv I_kI_k'I_k''$ and the elements of $I_k$ are second coordinates of elements in $\mathcal{S}$ and the elements of $I_k''$ are first coordinates of elements in $\mathcal{S}$ and elements of $I_k'$ do not appear in $\mathcal{S}$.  We can say furthermore from the above inequalities that $U_k \upharpoonright I_k'$ includes a subword which is $\equiv$ to $U_{1, k}$.  By construction there exist $i_k, i_k' \in I_k' \cup \overline{U_{1, k}^{m_k} U_{0, k}^{-1}}$ such that  $d(\prod_{q = k + 1}^{\infty} U_q) > \pro_1(U_k(i_k)), \pro_1(U_k(i_k'))$ and with $\pro_0(U_k(i_k)) \neq \pro_0(U_k(i_k'))$.

Now let $V_k \equiv U_k \upharpoonright I_k'$ and $V \equiv \prod_{k \in \omega} V_k$, so $V \equiv \Red(U) \equiv h(\Red(\prod_{k \in \omega}W_{n_k}^{m_k}))$.  Clearly the hypotheses of Lemma \ref{foranontrivialelement} apply, and so $[[V]] \neq [[E]]$ in $\Red_{\kappa_1}$.  However $\R_{\kappa_0}(\Red(\prod_{k \in \omega}W_{n_k}^{m_k})) \in \ker(\iota_{0*})$, which implies that $V \in \ker(\iota_{1*} \circ \R_{\kappa_1})$ since $\iota_{1*} \circ \R_{\kappa_1} \circ h = f_* \circ \iota_{0*} \circ \R_{\kappa_0}$, so $[[V]] = [[E]]$, a contradiction.

\end{proof}

\begin{lemma}\label{eventualspecialpurity}  For each $\alpha < \kappa_0$ there exist $\beta_{\alpha} < \kappa_1$ and $M_{\alpha} \in \omega$ such that if $d(W) \geq M_{\alpha}$ and $W$ is $\alpha$-pure then $h(W)$ is $\beta_{\alpha}$-pure.
\end{lemma}

\begin{proof}  By Lemma \ref{eventualpurity} we can select $N_{\alpha}$ such that if $d(W) \geq N_{\alpha}$ and $W$ is $\alpha$-pure then $h(W)$ is pure.  If our current lemma is false then there exists a sequence $\{W_n\}_{n \in \omega}$ with $d(W_n) > N_{\alpha}$ and $h(W_n)$ being nontrivial and $\beta_n$-pure, $d(W_n) \rightarrow \infty$, and $\beta_n \neq \beta_{n + 1}$ for each $n \in \omega$.  Letting $V_n \equiv h(W_n)$ and $V \equiv \prod_{n \in \omega} V_n$ it is clear that $V \in \Red_{\kappa_1}$ since $\beta_n \neq \beta_{n + 1}$ (any reduction would require that some parts of a word $V_n$ will cancel with parts of $V_{n + 1}$, and this is impossible).  Also, by the continuity of $g$ we know $V \equiv h(\Red(\prod_n W_n))$.

Since $\R_{\kappa_0}(\Red(\prod_{n \in \omega} W_n)) \in \ker(\iota_{0*})$ we have $[[V]] = [[E]]$.  Then by Corollary \ref{deletesomewords} we select intervals $I_0, \ldots, I_p$ in $\overline{V}$ such that $V \upharpoonright I_k$ is pure for each $0 \leq k \leq p$ and $\Red(V \upharpoonright (\overline{V} \setminus \bigcup_{k = 0}^p I_k)) = E$.  Each $I_k$ must be a subinterval of some $\overline{V_{n_k}} \subseteq \overline{V}$ since $V \upharpoonright I_k$ is pure.  Let $M > n_0, \ldots, n_p$.  Let $V' = \Red(V\upharpoonright (\overline{\prod_{n = 0}^{M-1} V_n} \setminus \bigcup_{k = 0}^p I_k))$.  The subword $\prod_{n = M}^{\infty} V_n$ is reduced and $V'$ is also reduced and $E \equiv \Red(V \upharpoonright (\overline{V} \setminus \bigcup_{k = 0}^p I_k)) \equiv \Red(V' \prod_{n = M}^{\infty} V_n)$.  Thus by Lemma \ref{reduced} we have $(V')^{-1} \equiv \prod_{n = M}^{\infty} V_n$.  However, $V'$ is clearly the concatenation of at most $M$ pure words, whereas $\prod_{n = M}^{\infty} V_n$ is not, and this is a contradiction.

\end{proof}

Now we are ready to finish the proof of Theorem \ref{continuous}.  For (1), if $\kappa_0 = n < \kappa = \kappa_1$ in the notation above, with $n$ finite, then we can select by Lemma \ref{eventualspecialpurity} an $M \in \omega$ large enough and $\beta_{\alpha}$ for each $\alpha < n = \kappa_0$ so that if $d(W) > M$ and $W$ is $\alpha$-pure then $h(W)$ is $\beta_{\alpha}$-pure.  Then we may select $\beta < \kappa_1$ such that $\beta \notin \{\beta_{\alpha}\}_{\alpha < n}$.  The continuous function $f_1: \GS_{\kappa_1} \rightarrow \GS_2$ given by mapping the $\beta$-cone homeomorphically to the $1$-cone of $\GS_2$ and mapping each other cone homeomorphically to the $0$-cone of $\GS_2$ is clearly such that $f_{1*}: \pi_1(\GS_{\kappa_1}, \circ_{\kappa_1}) \rightarrow \pi_1(\GS_2, \circ_2)$ is surjective, but also the image $f_*(\pi_1(\GS_{\kappa_0}))$ is included in the kernel $\ker(f_{1*})$, and so claim (1) follows since $\ker(f_{1*})$ has index at least $2^{\aleph_0}$ in $\pi_1(\GS_{\kappa_1}, \circ_{\kappa_1})$.

For (2) we let $\kappa_1 = n < \kappa = \kappa_0$ in the notation used above.  To prove that $\ker(f_*)$ is uncountable it is sufficient to show that $\ker(f_{2*})$ is uncountable, where $f_2$ is the restriction $f \upharpoonright \GS_{n + 1}$ since the subspace $\GS_{n + 1}$ is a retract subspace of $\GS_{\kappa_0}$ (so, in particular, $\pi_1(\GS_{n + 1}, \circ_{n + 1})$ includes into $\pi_1(\GS_{\kappa_0}, \circ_{\kappa_0})$ as a retract subgroup).  Thus we will assume that $\kappa = \kappa_0 = n + 1$ and that $f_2 = f$.  By Lemma \ref{eventualspecialpurity}, since $n + 1$ is finite we select an $M \in \omega$ large enough and $\beta_{\alpha}$ for each $\alpha < n + 1 = \kappa_0$ so that if $d(W) > M$ and $W$ is $\alpha$-pure then $h(W)$ is $\beta_{\alpha}$-pure.  By the pigeonhole principle, since $n < n + 1$, there are $\alpha_0, \alpha_1 < n + 1$ such that $\beta_{\alpha_0} = \beta_{\alpha_1}$.  But now any words in $\Red_{n + 1} = \Red_{\kappa_0}$ which utilize only letters whose first coordinate is in $\{\alpha_0, \alpha_1\}$ will represent elements in $\ker(f_*)$, and this implies that $\ker(f_*)$ is of cardinality at least $2^{\aleph_0}$.

\end{proof}

\end{section}

\begin{section}{Theorem \ref{bigisomorphism}}\label{isom}

We begin with a description of the overall strategy and then describe the structure of this section.  An isomorphism between two cone groups $\Co_{\kappa_0}$ and $\Co_{\kappa_1}$ will be constructed by induction on specially defined subgroups.  We cannot expect that such an isomorphism will be imposed by a homomorphism $\Red_{\kappa_0} \rightarrow \Red_{\kappa_1}$, because of the arguments of Section \ref{conegroups}.  However, the idea is that establishing careful correspondences between certain words in $\Red_{\kappa_0}$ and certain words in $\Red_{\kappa_1}$ will allow us to ultimately produce homomorphisms $\phi_0: \Red_{\kappa_0} \rightarrow \Co_{\kappa_1}$ and $\phi_1: \Red_{\kappa_1} \rightarrow \Co_{\kappa_0}$ which will descend to isomorphisms $\Phi_0: \Co_{\kappa_0} \rightarrow \Co_{\kappa_1}$ and $\Phi_1: \Co_{\kappa_1} \rightarrow \Co_{\kappa_0}$ with $\Phi_1 = \Phi_0^{-1}$.

What sort of correspondences between words should be produced?  They should not be so rigid as to produce a homomorphism $\Red_{\kappa_0} \rightarrow \Red_{\kappa_1}$.  Rather, they should be forgiving enough to produce the homomorphisms $\phi_0$ and $\phi_1$ described above.  The correspondences should also agree with each other so that the $\phi_0$ and $\phi_1$ are well-defined.

Each word in $\Red_{\kappa_0}$ and $\Red_{\kappa_1}$ may be decomposed in a natural way as a concatenation of maximal pure subwords (the index over which concatenation is written is unique up to order isomorphism and is called the \emph{p-index}).  Taking concatenations over subintervals of the p-index gives us words which are recognizable pieces of the original word (which we will call \emph{p-chunks}).  There is a natural way of comparing certain words $W \in \Red_{\kappa_0}$ with other in $U \in \Red_{\kappa_1}$ via an order isomorphism between a subset of the p-index of $W$ and that of $U$.  These subsets will be large enough to ``capture'' any interval of the p-index, up to deletion of finitely many elements, and there will be a correspondence between the p-chunks of $W$ and those of $U$.  The bijections between the subsets of the p-indices will honor word concatenation (up to finite deletion of pure subwords) and will allow us to define isomorphisms between the subgroups of $\Co_{\kappa_0}$ and $\Co_{\kappa_1}$ which are generated by the p-chunks of the words on which we have defined such bijections.

In order to have the isomorphisms be well-defined, it is essential that the imposed correspondences between p-chunks are in agreement with each other.  That is- suppose that $W_0, W_1 \in \Red_{\kappa_0}$ and $U_0, U_1 \in \Red_{\kappa_1}$ and $W_i$ is made to correspond to $U_i$ for $i = 0, 1$.  If $W \in \Red_{\kappa_0}$ is a p-chunk of each of $W_0$ and $W_1$ then we should be able to make $W$ correspond to a word $U \in \Red_{\kappa_0}$ in a way that honors the correspondences $W_i \leftrightarrow U_i$, so any choice of such a $U$ should  be independent of whether we are considering $W$ as a p-chunk of $W_0$ or of $W_1$, up to the equivalence $\approx$.

It will be necessary to be able to define many such correspondences between words, so as to make the isomorphism between subgroups of $\Co_{\kappa_0}$ and $\Co_{\kappa_1}$ have larger and larger domain and range.  Keeping such new correspondences in agreement with the previously defined ones requires us to consider concatenations of words on which such bijections have already been defined, concatenations of order type $\omega$ and of order type $\mathbb{Q}$ are of particular concern.  If we can continue to do this for sufficiently many steps ($2^{\aleph_0}$ steps will suffice) then we can succeed in the construction.

This section is organized into subsections for the sake of clarity.  We introduce and prove some basic properties of p-chunks in subsection \ref{pchunks}.  In subsection \ref{closesubsets} we will make precise the concept of a ``sufficiently large'' subset of an ordered set.  In subsection \ref{coherentcoitriplessubsect} we define what it means for bijections between sufficiently large subsets of p-indices to honor word concatenation (up to deletion of finitely many pure subwords).  In subsection \ref{extendingcoherence} we give some baby steps towards defining such bijections on more words, and in subsections \ref{omegaconcat} and \ref{Qconcat} we show how to extend such notions for $\omega$- and $\mathbb{Q}$-type concatenations, respectively.  Finally in subsection \ref{arbitraryext} we combine all the previous ideas to prove Theorems \ref{bigisomorphism} and \ref{elementaryequiv}.

\begin{subsection}{P-chunks}\label{pchunks}   Let $\kappa$ be a cardinal.  For each word $W\in \Red_{\kappa}$ we have a decomposition of the domain $\overline{W} \equiv  \prod_{\lambda \in \Lambda} \Lambda_{\lambda}$ such that each $\Lambda_{\lambda}$ is a nonempty maximal interval such that $W\upharpoonright \Lambda_{\lambda}$ is pure.  We'll call this decomposition the \emph{pure decomposition of the domain of $W$}.   Write $W \equiv_p \prod_{\lambda \in \Lambda} W_{\lambda}$ to express that $\overline{W}\equiv \prod_{\lambda \in \Lambda}\overline{W_{\lambda}}$ is the p-decomposition of the domain of $W$, and call this writing $W \equiv_p \prod_{\lambda \in \Lambda} W_{\lambda}$ the \emph{p-decomposition of $W$} and $\Lambda$ the \emph{p-index}, denoted $\pindex(W)$.  By definition we therefore have $E \equiv_p \prod_{\lambda\in \Lambda}W_{\lambda}$ with $\Lambda = \emptyset$.  If $W \equiv_p \prod_{\lambda \in \pindex(W)}W_{\lambda}$ and $I$ is an interval in $\pindex(W)$ then let $W\upharpoonright_p I$ denote the word $\prod_{\lambda \in I} W_{\lambda}$.  Call a word $W'$ a \emph{p-chunk} of $W$ if for some interval $I \subseteq \pindex(W)$ we have $W' \equiv W \upharpoonright_p I$.  For a given $W\in \Red_{\kappa}$ we let $\pchunk(W)$ denote the set of p-chunks of $W$.  A pure p-chunk of a word $W \equiv_p \prod_{\lambda \in \Lambda} W_{\lambda}$ will, of course, either be empty or one of the $W_{\lambda}$.  Notice as well that an equivalence $W \equiv U$ immediately gives an order isomorphism from $\pindex(W)$ to $\pindex(U)$.

\begin{lemma}\label{pchunkmultiplication}  Suppose that $W \equiv_p \prod_{\lambda \in \Lambda} W_{\lambda}$ and $U \equiv_p \prod_{\lambda' \in \Lambda'} U_{\lambda'}$.  Then there exists a (possibly empty) initial interval $I \subseteq \Lambda$, a (possibly empty) terminal interval $I' \subseteq \Lambda'$ such that either:

\begin{enumerate}[(i)] \item $\Red(WU) \equiv_p \prod_{\lambda \in I}W_{\lambda}\prod_{\lambda'\in I'}U_{\lambda'}$; or

\item there exist $\lambda_0\in \Lambda$ which is the least element strictly above all elements in $I$, $\lambda_1\in \Lambda'$ which is the greatest element strictly below all elements of $I'$ and 

\begin{center}
$\Red(WU) \equiv_p (\prod_{\lambda \in I}W_{\lambda})V(\prod_{\lambda'\in I'}U_{\lambda'})$ 
\end{center}

\noindent where $V \equiv \Red(W_{\lambda_0}U_{\lambda_1})\not\equiv E$ is pure.

\end{enumerate}

\end{lemma}

\begin{proof}  Since both $W$ and $U$ are reduced we have reduced words $W_0$, $W_1$, $U_0$, $U_1$  as in the conclusion of Lemma \ref{reduced}.  Select $I_0 \subseteq \Lambda$ to be a maximal initial interval for which $\bigcup_{\lambda \in I}\overline{W_{\lambda}} \subseteq \overline{W_0}$.  Select $I_1' \subseteq \Lambda'$ to be a maximal terminal interval such that $\bigcup_{\lambda' \in I'}\overline{U_{\lambda'}} \subseteq \overline{U_1}$.

Suppose $\prod_{\lambda \in I_0} W_{\lambda}\equiv W_0$ and $\prod_{\lambda' \in I_1'} U_{\lambda'} \equiv U_1$.  If $I_0$ has a maximal element $\lambda_0$ and $I_1'$ has a minimal element $\lambda_1$ such that the words $W_{\lambda_0}$ and $U_{\lambda_1}$ are both $\alpha$-pure for some $\alpha$, then we let $I = I_0 \setminus \{\lambda_0\}$ and $I' = I_1'\setminus \{\lambda_1\}$ and $V \equiv W_{\lambda_0}U_{\lambda_1}$ and obviously condition (ii) holds.  If there are no such maximal and minimal elements then condition (i) holds.

Suppose that $\prod_{\lambda \in I_0} W_{\lambda}\not\equiv W_0$.  Then there exists some $\lambda_0$ which is the least element strictly above all elements in $I_0$ and nonempty words $W_{\lambda_0, 0}$ and $W_{\lambda_0, 1}$ such that

$$
\begin{array}{ll}
W_{\lambda_0} & \equiv W_{\lambda_0, 0}W_{\lambda_0, 1};\\
W_0 & \equiv_p (\prod_{\lambda \in I_0} W_{\lambda}) W_{\lambda_0, 0};\\
W_1 & \equiv_p W_{\lambda_{0}, 1} (\prod_{\lambda \in \Lambda \setminus (I_0\cup \{\lambda_0\})}).
\end{array}
$$

\noindent If in addition $\prod_{\lambda'\in I_1} U_{\lambda'} \equiv U_1$ then $\Lambda' \setminus I_1$ has a maximum element $\lambda_1$ which satisfies $U_{\lambda_1} \equiv W_{\lambda_0, 1}^{-1}$.  Thus we let $I = I_0\setminus \{\lambda_0\}$ and $I' = I_1$ and $V \equiv W_{\lambda_0, 0} \equiv \Red(W_{\lambda_0}U_{\lambda_1})$ and we have condition (ii).  On the other hand, if in addition we have $\prod_{\lambda'\in I_1} U_{\lambda'} \not\equiv U_1$ then $\Lambda' \setminus I_1$ has a maximum element $\lambda_1$ and there exist nonempty words $U_{\lambda_1, 0}$ and $U_{\lambda_1, 1}$ for which

$$
\begin{array}{ll}
U_{\lambda_1} & \equiv U_{\lambda_1, 0}U_{\lambda_1, 1};\\
U_0 & \equiv_p (\prod_{\lambda' \in \Lambda' \setminus I_1}U_{\lambda'})U_{\lambda_1, 0};\\
U_1 & \equiv_p U_{\lambda_1, 1}(\prod_{\lambda'\in I_1} U_{\lambda'}).
\end{array}
$$

\noindent Then we let $V \equiv W_{\lambda_0, 0}V_{\lambda_1, 1} \equiv \Red(W_{\lambda_0}U_{\lambda_1})$ and $I = I_0$ and $I' = I_1$ and condition (ii) holds.

The case where $\prod_{\lambda \in I_0} W_{\lambda}\equiv W_0$ and $\prod_{\lambda' \in I_1'} U_{\lambda'} \not\equiv U_1$ follows from dualizing the proof of an earlier case, and so we are done.
\end{proof}

\begin{lemma}\label{elementsofthegeneratedsubgroup}  Suppose that $X \subseteq \Red_{\kappa}$.  For each nonempty element $W$ of the subgroup $\langle \bigcup_{U\in X}\pchunk(U) \rangle \leq \Red_{\kappa}$ if $W \equiv_p \prod_{\lambda \in \Lambda} W_{\lambda}$ then there exist nonempty intervals $I_0, \ldots, I_n$ in $\Lambda$ such that

\begin{enumerate}[(i)]

\item $\Lambda \equiv \prod_{i = 0}^n I_i$; and

\item for each  $0 \leq i \leq n$ at least one of the following holds:

\begin{enumerate}[(a)]

\item $I_i$ is a singleton $\{\lambda\}$ such that $W_{\lambda}$ is the reduction of a finite concatenation of pure p-chunks of elements in $X^{\pm 1}$;

\item $\prod_{\lambda \in I_i} W_{\lambda}$ is a p-chunk of some element in $X^{\pm 1}$.

\end{enumerate}

\end{enumerate}
\end{lemma}

\begin{proof}  The elements of $\langle \bigcup_{U \in X} \pchunk(U) \rangle$ are of form $\Red(U_0\cdots U_l)$ where each $U_i$ is a p-chunk of an element of $X^{\pm 1}$.  The claim will follow by an induction on the number $l$.  If $l = 0$ or $l = 1$ then we are already done.  Supposing that the claim holds for $l$, we suppose $W \equiv \Red(U_0\cdots U_{l + 1}) \equiv \Red(\Red(U_0\cdots U_l)U_{l+1})$ and let $W' \equiv \Red(U_0\cdots U_l)$ and $U \equiv U_{l+1}$.  Let $W' \equiv W_0W_1$ and $U \equiv U_0U_1$ as in Lemma \ref{reduced} for performing the reduction $\Red(W'U)$.  Let $W' \equiv_p \prod_{\lambda \in \Lambda} W_{\lambda}$ and $U \equiv_p \prod_{\lambda' \in \Lambda'}U_{\lambda}$. By induction we have for the word $W'$ a decomposition $I_0, \ldots, I_{n'}$ as in the conclusion of this lemma.  We can select an initial interval $I \subseteq \Lambda$ and a terminal interval $I'\subseteq \Lambda'$ as in the conclusion of Lemma \ref{pchunkmultiplication}.  Consider the two possible cases in Lemma \ref{pchunkmultiplication} for the word $W \equiv \Red(W'U)$.  If case (i) holds then we can decompose the p-chunk total order for $W$ into at most $n' + 1$ intervals as in (i)-(iii) of the statement of the lemma that we are proving.  If case (ii) holds then we can decompose the p-chunk total order for $W$ into at most $n'+ 2$ intervals, at least one of which will be a singleton.  Thus we are done.
\end{proof}

We say a subgroup $G$ of $\Red_{\kappa}$ is \emph{p-fine} if each p-chunk $U$ of each $W\in G$ is also in $G$ (cf. \cite[page 600]{E2}).

\begin{lemma}\label{fine}  If $X \subseteq \Red_{\kappa}$ then the subgroup $\langle \bigcup_{U\in X} \pchunk(U)\rangle \leq \Red_{\kappa}$ is p-fine.  This is the smallest p-fine subgroup including the set $X$.
\end{lemma}

\begin{proof}  This follows immediately from the characterization in Lemma \ref{elementsofthegeneratedsubgroup}.
\end{proof}

Given a set $X \subseteq \Red_{\kappa}$ we'll denote the subgroup $\langle \bigcup_{U\in X} \pchunk(U)\rangle \leq \Red_{\kappa}$ by $\Pfine(X)$.

\begin{lemma}\label{notmanypures}  If $X \subseteq \Red_{\kappa}$ then there are at most $(|X| +1) \cdot \aleph_0$ pure p-chunks of elements in $\Pfine(X)$.
\end{lemma}

\begin{proof}  This is also immediate from Lemma \ref{elementsofthegeneratedsubgroup}, since the pure p-chunks in $\Pfine(X)$ are reductions of finite concatenations of pure p-chunks of elements in $X^{\pm 1}$.
\end{proof}

\end{subsection}

\begin{subsection}{Close Subsets}\label{closesubsets}  We take a diversion through a concept which will be useful in later subsections.

\begin{definition}  Let $\Lambda$ be a totally ordered set.  We say $\Lambda_0 \subseteq \Lambda$ is \emph{close in $\Lambda$}, and write $\Close(\Lambda_0, \Lambda)$, if every infinite interval in $\Lambda$ has nonempty intersection with $\Lambda_0$.  
\end{definition}

\begin{lemma}\label{basiccloseproperties}  The following hold:

\begin{enumerate}[(i)]

\item  If $\Close(\Lambda_0, \Lambda)$ then for any infinite interval $I \subseteq \Lambda$ the set $I \cap \Lambda_0$ is infinite.

\item  If $\Lambda_2 \subseteq \Lambda_1 \subseteq \Lambda_0$ with $\Close(\Lambda_{i+1}, \Lambda_i)$  for $i = 0, 1$, then $\Close(\Lambda_2, \Lambda_0)$.

\item  If $\Lambda \equiv \prod_{\theta \in \Theta} \Lambda_{\theta}$, $\Close(\{\theta \in \Theta \mid \Lambda_{\theta, 0} \neq \emptyset\}, \Theta)$, and $\Close(\Lambda_{\theta, 0}, \Lambda_{\theta})$ for each $\theta \in \Theta$ then $\Close(\bigcup_{\theta \in \Theta} \Lambda_{\theta, 0}, \Lambda)$.

\item  If  $I_0$ is an interval in $\Lambda$ and $\Close(\Lambda_0, \Lambda)$ then $\Close(\Lambda_0 \cap I_0, I_0)$

\end{enumerate}

\end{lemma}

\begin{proof} (i)  If instead $I \cap \Lambda_0 = \{\lambda_0, \lambda_1, \ldots, \lambda_n\}$ with $\lambda_i < \lambda_{i + 1}$ then at least one of the intervals $I \cap (-\infty, \lambda_0)$, $(\lambda_0, \lambda_1)$, $\ldots$, $(\lambda_{n-1}, \lambda_n)$, $I \cap (\lambda_n, \infty)$ in $\Lambda$ is infinite, but each has empty intersection with $\Lambda_0$ and this is a contradiction.

\noindent (ii) Let $I \subseteq \Lambda_0$ be an infinite interval.  Notice that $I \cap \Lambda_1$ is infinite by (i) and so $I \cap \Lambda_1$ is an infinite interval in $\Lambda_1$, so $I \cap \Lambda_2 = (I \cap \Lambda_1) \cap \Lambda_2 \neq \emptyset$.

\noindent (iii)  Let $I \subseteq \Lambda$ be an infinite interval.  The set $I_0 = \{\theta \in \Theta \mid I \cap \Lambda_{\theta} \neq \emptyset\}$ is an interval in $\Theta$.  If $I_0$ is finite then as $I = \bigsqcup_{\theta \in I_0} (I \cap \Lambda_{\theta})$ there is some $\theta_0 \in I_0$ for which $|I \cap \Lambda_{\theta_0}| = \infty$, and as $I \cap \Lambda_{\theta_0}$ is an infinite interval in $\Lambda_{\theta_0}$ we see that $I \cap \Lambda_{\theta_0, 0} \neq \emptyset$, so $I \cap \bigcup_{\theta \in \Theta} \Lambda_{\theta, 0} \neq \emptyset$.  If $I_0$ is infinite then $I_0 \cap \{\theta \in \Theta \mid \Lambda_{\theta, 0} \neq \emptyset\}$ is infinite by (i), as we are assuming $\Close(\{\theta \in \Theta \mid  \Lambda_{\theta, 0} \neq \emptyset\}, \Theta)$.  Then there exists some $\theta_0 \in I_0 \cap \{\theta \in \Theta \mid  \Lambda_{\theta, 0} \neq \emptyset\}$ for which $I \supseteq \Lambda_{\theta_0}$.  Thus $I \cap \Lambda_{\theta_0, 0} \neq \emptyset$.

\noindent (iv)  This is obvious.
\end{proof}

If $\Close(\Lambda_0, \Lambda)$ then for each interval $I \subseteq \Lambda$ we let $\varpropto(I, \Lambda_0)$ denote the smallest interval in $\Lambda$ which includes the set $I \cap \Lambda_0$.  In other words $\varpropto(I, \Lambda_0) =  \bigcup_{\lambda_0, \lambda_1 \in I \cap \Lambda_0, \lambda_0 \leq \lambda_1} [\lambda_0, \lambda_1]$ where the intervals $[\lambda_0, \lambda_1]$ are being considered in $\Lambda$.

\begin{lemma}\label{prettyclose} Let $\Close(\Lambda_0, \Lambda)$ and $I \subseteq \Lambda$ be an interval.

\begin{enumerate}[(i)]

\item The inclusion $I \supseteq \varpropto(I, \Lambda_0)$ holds and $\varpropto(I, \Lambda_0) = \varpropto(\varpropto(I, \Lambda_0), \Lambda_0)$.

\item  The set $I \setminus \varpropto(I, \Lambda_0)$ is the disjoint union of an initial and terminal subinterval $I_0, I_1 \subseteq I$ (either subinterval could be empty) with $|I_0|, |I_1| < \infty$.
\end{enumerate}

\end{lemma}

\begin{proof} (i)  The claimed inclusion is obvious.  For the claimed equality it is therefore sufficient to prove that $\varpropto(I, \Lambda_0) \subseteq \varpropto(\varpropto(I, \Lambda_0), \Lambda_0)$.  We let $\lambda \in \varpropto(I, \Lambda_0)$ be given.  Select $\lambda_0, \lambda_1 \in I \cap \Lambda_0$ such that $\lambda_0 \leq \lambda \leq \lambda_1$.  Then $\lambda_0, \lambda_1 \in \varpropto(I, \Lambda_0) \cap \Lambda_0$ and $\lambda_0 \leq \lambda \leq \lambda_1$, so $\lambda \in \varpropto(\varpropto(I, \Lambda_0), \Lambda_0)$.

(ii)   If $I \cap \Lambda_0 = \emptyset$ then $I$ is finite (since $\Close(\lambda_0, \Lambda)$) and we can let $I_0 = \emptyset$ and $I_1 = I$.  If $I \cap \Lambda_0 \neq \emptyset$ then we let $I_0 = \{\lambda \in I \mid (\forall \lambda_0 \in I \cap \Lambda_0)\lambda < \lambda_0\}$ and $I_1 = \{\lambda \in I \mid (\forall \lambda_0 \in I \cap \Lambda_0)\lambda > \lambda_0\}$.  Clearly $I \equiv I_0\varpropto(I, \Lambda_0)I_1$.  Each of $I_0$ and $I_1$ is a subinterval of $I$ and therefore a subinterval of $\Lambda$ as well.  If, say, $I_0$ is infinite then $I_0 \cap \Lambda_0 \neq \emptyset$ but this is an obvious contradiction.

\end{proof}

We will say that two totally ordered sets $\Lambda$ and $\Theta$ are \emph{close-isomorphic} if there exist $\Lambda_0 \subseteq \Lambda$ and $\Theta_0 \subseteq \Theta$ with $\Close(\Lambda_0, \Lambda)$, $\Close(\Theta_0, \Theta)$ and $\Lambda_0$ order isomorphic to $\Theta_0$; and if $\iota$ is an order isomorphism between such a $\Lambda_0$ and $\Theta_0$ then we will call $\iota$ a \emph{close order isomorphism from $\Lambda$ to $\Theta$}.  It is obvious that the inverse of a close order isomorphism from $\Lambda$ to $\Theta$ is a close order isomorphism from $\Theta$ to $\Lambda$.

From a close order isomorphism (abbreviated \emph{coi}) between totally ordered sets one obtains a reasonable way of identifying intervals in one totally ordered set with intervals in the other, which we now describe.  Given coi $\iota$ between $\Lambda$ and $\Theta$, with $\Lambda_0$ and $\Theta_0$ being the respective domain and range of $\iota$, and an interval $I \subseteq \Lambda$ we let $\varpropto(I, \iota)$ denote the smallest interval in $\Theta$ which includes the set $\iota(I \cap \Lambda_0)$.  Thus  $\varpropto(I, \iota) = \bigcup_{\theta_0, \theta_1 \in \iota(I \cap \Lambda_0), \theta_0 \leq \theta_1} [\theta_0, \theta_1]$, where each interval $[\theta_0, \theta_1]$ is being considered in $\Theta$.

\begin{lemma}\label{almostidentified}  If $\iota$ is a coi between $\Lambda$ and $\Theta$ and $I \subseteq \Lambda$ is an interval then $\varpropto(\varpropto(I, \iota), \iota^{-1}) = \varpropto(I, \Lambda_0)$, where $\iota: \Lambda_0 \rightarrow \Theta_0$.
\end{lemma}

\begin{proof}  Straightforward.
\end{proof}

We point out that a coi $\iota$ between $\Lambda$ and $\Theta$ also induces a coi between the reversed orders $\Lambda^{-1}$ and $\Theta^{-1}$ in the obvious way.

\begin{lemma}\label{coilemma}  Let $I \equiv I_0 \cdots I_n$ and $\iota$ a coi from $I$ to $I'$.  Then there exist (possibly empty) finite subintervals $I_0', \ldots, I_{n+1}'$ of $\varpropto(I, \iota)$ such that 

\begin{center}

$\varpropto(I, \iota) \equiv I_0'\varpropto(I_0, \iota) I_1' \varpropto(I_1, \iota) I_2' \cdots \varpropto(I_n, \iota) I_{n+1}'$.

\end{center}
\end{lemma}

\begin{proof}  Assume the hypotheses and let $\Close(\Lambda, I)$ and $\Close(\Lambda', I')$ with $\iota: \Lambda \rightarrow \Lambda'$ being an order isomorphism.  Clearly each $\varpropto(I_j, \iota)$ is a subinterval of $\varpropto(I, \iota)$, and it is easy to see that all elements of $\varpropto(I_j, \iota)$ are strictly below all elements of $\varpropto(I_{j+1}, \iota)$ for $0 \leq j < n$.  Thus we may indeed write 

\begin{center}

$\varpropto(I, \iota) \equiv I_0'\varpropto(I_0, \iota) I_1' \varpropto(I_1, \iota) I_2' \cdots \varpropto(I_n, \iota) I_{n+1}'$

\end{center}

\noindent and we conclude by pointing out that $I_l' \cap \Lambda' = I_l' \cap \iota(\Lambda) = I_l' \cap (\bigcup_{j = 0}^n \iota(I_j \cap \Lambda)) \subseteq \bigcup_{j = 0}^n (I_l' \cap \varpropto(I_j, \iota)) = \emptyset$ for each $0 \leq l \leq n+1$, and since $\Close(\Lambda', I')$ we have $I_l'$ finite.
\end{proof}

\begin{lemma}\label{coilemma2}  Let $\iota$ be a coi from $I$ to $I'$.  If $I_0 \subseteq I$ is finite then $\varpropto(I_0, \iota)$ is finite.
\end{lemma}

\begin{proof}  Let $\Close(\Lambda, I)$ and $\Close(\Lambda', I')$ and $\iota: \Lambda \rightarrow \Lambda'$ be an order isomorphism.  Since $I_0$ is finite, we know $I_0 \cap \Lambda$ is finite.  Clearly we have $\varpropto(I_0, \iota) \cap \Lambda' = \iota(I_0 \cap \Lambda)$, so $\varpropto(I_0, \iota)$ is an interval in $I'$ having finite intersection with $\Lambda'$.  Thus $\varpropto(I_0, \iota)$ is finite by Lemma \ref{basiccloseproperties} (i).
\end{proof}

\end{subsection}

\begin{subsection}{Coherent coi triples}\label{coherentcoitriplessubsect}  Suppose that $\kappa_0$ and $\kappa_1$ are cardinal numbers greater than or equal to $2$.  For words $W \in \Red_{\kappa_0}$ and $U \in \Red_{\kappa_1}$ we'll write $\coi(W, \iota, U)$ to denote that $\iota$ is a coi between $\pindex(W)$ and $\pindex(U)$ and say that $\coi(W, \iota, U)$ is a \emph{coi triple from $\Red_{\kappa_0}$ to $\Red_{\kappa_1}$}.  We will often abuse language and say that $\iota$ is a coi from $W$ to $U$ when really $\iota$ is a coi from $\pindex(W)$ to $\pindex(U)$.

\begin{definition}  A collection $\{\coi(W_x, \iota_x, U_x)\}_{x\in X}$ of coi triples from $\Red_{\kappa_0}$ to $\Red_{\kappa_1}$ is \emph{coherent} if for any choice of $x_0, x_1 \in X$, intervals $I_0 \subseteq \pindex(W_{x_0})$ and $I_1 \subseteq \pindex(W_{x_1})$ and $i \in \{-1, 1\}$ such that $W_{x_0} \upharpoonright_p I_0 \equiv (W_{x_1}\upharpoonright_p I_1)^i$ we get

\begin{center}

$[[U_{x_0}\upharpoonright_p \varpropto(I_0, \iota_{x_0})]] = [[(U_{x_1}\upharpoonright_p \varpropto(I_1, \iota_{x_1}))^i]]$

\end{center}
 
\noindent and similarly for any choice of $x_2, x_3 \in X$, intervals $I_2 \subseteq \pchunk(U_{x_2})$ and $I_3 \subseteq \pchunk(U_{x_3})$ and $j \in \{-1, 1\}$ such that $U_{x_2} \upharpoonright_p I_2 \equiv (U_{x_3} \upharpoonright_p I_3)^j$ we get

\begin{center}

$[[W_{x_2} \upharpoonright_p \varpropto(I_2, \iota_{x_2}^{-1})]] = [[(W_{x_3} \upharpoonright_p \varpropto(I_3, \iota_{x_3}^{-1}))^j]]$.

\end{center}

\noindent It is clear from the symmetric nature of this definition that if collection of coi triples $\{\coi(W_x, \iota_x, U_x)\}_{x\in X}$  from $\Red_{\kappa_0}$ to $\Red_{\kappa_1}$ is coherent then so also is the collection of coi triples $\{\coi(U_x, \iota_x^{-1} W_x)\}_{x\in X}$ from $\Red_{\kappa_1}$ to $\Red_{\kappa_0}$.  We emphasize that a word can appear multiple times in a coherent collection.  For example, if each element of $\{W_x\}_{x \in X}$ is pure then the collection $\{(W_x, \iota_x, E)\}_{x \in X}$ is obviously coherent (each $\iota_x$ is the empty function).

\end{definition}

\begin{lemma}\label{ascendingchaincoi}  Suppose that $\Theta$ is a totally ordered set and that $\{\mathcal{T}_{\theta}\}_{\theta \in \Theta}$ is a collection of coherent collections of coi triples from $\Red_{\kappa_0}$ to $\Red_{\kappa_1}$ such that $\theta \leq \theta'$ implies $\mathcal{T}_{\theta} \subseteq \mathcal{T}_{\theta'}$.  Then $\bigcup_{\theta \in \Theta} \mathcal{T}_{\theta}$ is coherent.
\end{lemma}

\begin{proof}  Supposing that $\coi(W_{x_0}, \iota_{x_0}, U_{x_0}), \coi(W_{x_1}, \iota_{x_1}, U_{x_1}) \in \bigcup_{\theta \in \Theta} \mathcal{T}_{\theta}$ and intervals $I_0 \subseteq \pindex(W_{x_0})$ and $I_1 \subseteq \pindex(W_{x_1})$ and $i\in \{-1, 1\}$ are such that $W_{x_0} \upharpoonright_p I_0 \equiv (W_{x_1}\upharpoonright_p I_1)^i$, we select $\theta \in \Theta$ such that $\coi(W_{x_0}, \iota_{x_0}, U_{x_0}), \coi(W_{x_1}, \iota_{x_1}, U_{x_1}) \in \mathcal{T}_{\theta}$.  As $\mathcal{T}_{\theta}$ is coherent we get 

\begin{center}

$[[U_{x_0}\upharpoonright_p \varpropto(I_0, \iota_{x_0})]] = [[(U_{x_1}\upharpoonright_p \varpropto(I_1, \iota_{x_1}))^i]]$

\end{center}

\noindent  The comparable check for words $U_{x_2}, U_{x_3} \in \Red_{\kappa_1}$ is analogous.

\end{proof}

\begin{lemma}\label{gettingstarted}  Suppose $\{\coi(W_x, \iota_x, U_x)\}_{x\in X}$ is coherent, $x\in X$, $I \subseteq \pindex(W_x)$ is an interval, $I \equiv I_0I_1 \cdots I_n$.  Suppose also that for each $0 \leq j \leq n$ we have an $x_j \in X$, an interval $I_j'$ in $\pindex(W_{x_j})$ and $i_j \in \{-1, 1\}$ such that $W_x \upharpoonright_p I_j \equiv (W_{x_j} \upharpoonright_p I_j')^{i_j}$.  Then

\begin{center}  $[[U_x \upharpoonright_p \varpropto(I, \iota_x)]] = \prod_{j = 0}^n [[(U_{x_j} \upharpoonright_p \varpropto(I_j', \iota_{x_j}))^{i_j}]]$.

\end{center}

Furthermore, if $L = \{0 \leq j \leq n \mid |I_j| > 1\}$ we have

\begin{center}  $[[U_x \upharpoonright_p \varpropto(I, \iota_x)]] = \prod_{j \in L} [[(U_{x_j} \upharpoonright_p \varpropto(I_j', \iota_{x_j}))^{i_j}]]$.

\end{center}

\end{lemma}

\begin{proof}  For each $0 \leq j \leq n$ we have $W_x \upharpoonright_p I_j \equiv (W_{x_j} \upharpoonright_p I_j')^{i_j}$, so that by the fact that $\{\coi(W_x, \iota_x, U_x)\}_{x\in X}$ is coherent we see that $$[[U_x \upharpoonright_p \varpropto(I_j, \iota_x)]] = [[(U_{x_j} \upharpoonright_p \varpropto(I_j', \iota_{x_j}))^{i_j}]]$$ for all $0 \leq j \leq n$.  In particular we have $$\prod_{j = 0}^n [[U_x \upharpoonright_p \varpropto(I_j', \iota_x)]] = \prod_{j = 0}^n [[(U_{x_j} \upharpoonright_p \varpropto(I_j', \iota_{x_j}))^{i_j}]]$$ and so we will be done with the first claim if we show that $[[U_x \upharpoonright_p \varpropto(I, \iota_x)]] = \prod_{j = 0}^n [[U_x \upharpoonright_p \varpropto(I_j, \iota_x)]]$.  But this is true since by Lemma \ref{coilemma} the (possibly unreduced) word $\prod_{j = 0}^n U_{x} \upharpoonright_p \varpropto(I_j, \iota_x)$ is obtained from $U_x \upharpoonright_p \varpropto(I, \iota_x)$ by deleting finitely many pure subwords.

Next we let $L$ be as in the statement of the lemma.  Notice that for each $0 \leq j \leq n$ with $j \notin L$ we have $|I_j| = |I_j'| \leq 1$ and so $\varpropto(I_j', \iota_{x_j})$ is a finite interval, by Lemma \ref{coilemma2}.  Thus for each such $j$ we have $[[(U_{x_j} \upharpoonright_p \varpropto(I_j', \iota_{x_j}))^{i_j}]] = [[E]]$ since $U_{x_j} \upharpoonright_p \varpropto(I_j', \iota_{x_j})$ is a finite concatenation of pure words.  Thus removing all such $j$ from the multiplication expression $\prod_{j = 0}^n [[(U_{x_j} \upharpoonright_p \varpropto(I_j', \iota_{x_j}))^{i_j}]]$ will not change the value in the group, and so we are done with the second claim.

\end{proof}

What follows is a rather technical result that will allow us to conclude that certain natural maps are well-defined despite certain choices that are made.

\begin{lemma}\label{tedious}  Let $\{\coi(W_x, \iota_x, U_x)\}_{x\in X}$ be coherent and $W \in \Pfine(\{W_x\}_{x \in X})$.  Let $I_0, \ldots, I_n$ be a finite set of subintervals of $\pindex(W)$ as in the conclusion of Lemma \ref{elementsofthegeneratedsubgroup} and let $J = \{0 \leq j \leq n \mid |I_j| > 1\}$.  For each $j \in J$ select $x_j \in X$, $i_j \in \{-1, 1\}$, and interval $\Lambda_j \subseteq W_{x_j}$ such that $W \upharpoonright_p I_j \equiv (W_{x_j} \upharpoonright_p \Lambda_j)^{i_j}$.  Again, let $I_0', \ldots, I_{n'}'$ be a finite set of subintervals of $\pindex(W)$ as in the conclusion of Lemma \ref{elementsofthegeneratedsubgroup} and let $J' = \{0 \leq j' \leq n' \mid |I_{j'}'| > 1\}$.  For each $j' \in J'$ select $y_{j'} \in X$, $m_{j'} \in \{-1, 1\}$, and interval $\Lambda_{j'}' \subseteq W_{y_{j'}}$ such that $W \upharpoonright_p I_{j'}' \equiv (W_{y_{j'}} \upharpoonright_p \Lambda_{j'}')^{m_{j'}}$.  Then

\begin{center}  $\prod_{j \in J} [[(U_{x_j} \upharpoonright_p \varpropto(\Lambda_j, \iota_{x_j}))^{i_j}]] = \prod_{j' \in J'} [[(U_{y_{j'}} \upharpoonright_p \varpropto(\Lambda_{j'}', \iota_{y_{j'}}))^{m_{j'}}]]$.
\end{center}

\end{lemma}

\begin{proof}  Assume the hypotheses.  Take $\mathbb{I}$ to be the set of nonempty intervals obtained by intersecting an $I_j$ with an $I_{j'}'$.  For each $0 \leq j \leq n$ we can write $I_j \equiv I_{(j, 0)}I_{(j, 1)}\cdots I_{(j, n_j)}$ where each $I_{(j, q)}$ is an element of $\mathbb{I}$.  Similarly for each $0 \leq j' \leq n'$ we write $I_{j'}' \equiv I_{(j', 0)}' \cdots I_{(j', n_{j'}')}'$ where each $I_{(j', r)}'$ is an element of $\mathbb{I}$.

We have $\mathbb{I} = \{I_{(j, q)}\}_{0 \leq j \leq n, 0 \leq q \leq n_j} = \{I_{(j', r)}'\}_{0 \leq j' \leq n', 0 \leq r \leq n_{j'}'}$.  Let $F: \mathbb{I} \rightarrow \{(j, q) \mid 0 \leq j \leq n, 0 \leq q \leq n_j\}$ be the unique order isomorphism between the domain and codomain where the domain is given the lexicographic order, comparing the leftmost coordinate first and define $F':\mathbb{I} \rightarrow \{(j', r) \mid 0 \leq j' \leq n', 0 \leq r \leq n_{j'}'\}$ similarly.  Let $\first: \{(j, q) \mid 0 \leq j \leq n, 0 \leq q \leq n_j\} \rightarrow \{0, \ldots, n\}$ denote projection to the first coordinate, and similarly define $\first':  \{(j', r) \mid 0 \leq j' \leq n', 0 \leq r \leq n_{j'}'\} \rightarrow \{0, \ldots, n'\}$.  Let $\mathbb{J} \subseteq \mathbb{I}$ denote the set of intervals in $\mathbb{I}$ which are of cardinality at least $2$; that is, $\mathbb{J} = \{I_{(j, q)}| 0 \leq j \leq n, 0 \leq q \leq n_j , |I_{(j, q)}| \geq 2\}$.

For each $j \in J$ and each $I_{(j, q)} \in \mathbb{J}$ we know that $W \upharpoonright_p I_{(j, q)} \in \pchunk(W_{x_j}^{i_j})$, so select an interval $\Lambda_{(j, q)} \subseteq \pindex(W_{x_j})$ such that $W \upharpoonright_p I_{(j, q)} \equiv (W_{x_j} \upharpoonright_p \Lambda_{(j, q)})^{i_j}$.  Now

$$
\begin{array}{ll}
\prod_{j \in J} [[(U_{x_j} \upharpoonright_p \varpropto(\Lambda_j, \iota_{x_j}))^{i_j}]] & = \prod_{j \in J} \prod_{0 \leq q \leq n_j, I_{(j, q)} \in \mathbb{J}} [[(U_{x_j} \upharpoonright_p \varpropto(\Lambda_{(j, q)}, \iota_{x_j}))^{i_j}]]\\
& =  \prod_{I \in \mathbb{J}} [[(U_{x_{\first\circ F(I)}} \upharpoonright_p \varpropto(\Lambda_{F(I)}, \iota_{x_{\first\circ F(I)}}))^{i_{\text{first}\circ F(I)}}]] \\
& = \prod_{j' \in J'} \prod_{0 \leq r \leq n_{j'}', I_{(j', r)}\in \mathbb{J}}  [[(U_{x_{\first\circ F\circ (F')^{-1}(j', r))}} \upharpoonright_p\\
& \varpropto(\Lambda_{F\circ (F')^{-1}(j', r)}, \iota_{x_{\first \circ F\circ (F')^{-1}(j', r)}}))^{i_{\first\circ F\circ (F')^{-1}(j', r)}}]]\\
& = \prod_{j' \in J'} [[(U_{y_{j'}} \upharpoonright_p \varpropto(\Lambda_{j'}', \iota_{y_{j'}}))^{m_{j'}}]]
\end{array}
$$

\noindent where the first equality holds by Lemma \ref{gettingstarted}, the second and third equalities are simply a rewriting of the order index, and the last equality holds by another application of Lemma \ref{gettingstarted}.  This completes the proof.
\end{proof}

Now we may conclude that a coherent collection of coi's produces well-defined homomorphisms.  For each $i \in \{0, 1\}$ we let $\beth_{\kappa_i}: \Red_{\kappa_i} \rightarrow \Co_{\kappa_i}$ denote the surjection given by $W \mapsto [[W]]$.

\begin{proposition}\label{welldefined}  Let $\{\coi(W_x, \iota_x, U_x)\}_{x\in X}$ be coherent.  By selecting for each $W \in \Pfine(\{W_x\}_{x \in X})$ a finite set of subintervals $I_0, \ldots, I_n$ of $\pindex(W)$ as in the conclusion of Lemma \ref{elementsofthegeneratedsubgroup}, letting $J = \{0 \leq j \leq n \mid |I_j| > 1\}$, selecting for each $j \in J$ an element $x_j \in X$, $i_j \in \{-1, 1\}$, and interval $\Lambda_j \subseteq \pindex(W_{x_j})$ such that $W \upharpoonright_p I_j \equiv (W_{x_j} \upharpoonright_p \Lambda_j)^{i_j}$ we obtain a homomorphism

\begin{center}
$\phi_0:  \Pfine(\{W_x\}_{x \in X}) \rightarrow \beth_{\kappa_1}(\Pfine(\{U_x\}_{x \in X}))$
\end{center}

\noindent whose definition is independent of the choices made of the set of subintervals $I_0, \ldots, I_n$, elements $x_j \in X$ and $i_j \in \{-1, 1\}$, and intervals $\Lambda_j \subseteq \pindex(W_{x_j})$.  The comparable map 

\begin{center}

$\phi_1: \Pfine(\{U_x\}_{x \in X}) \rightarrow \beth_{\kappa_0}(\Pfine(\{W_x\}_{x \in X}))$

\end{center}

\noindent similarly is a homomorphism whose definition is independent of the various selections made.

\end{proposition}

\begin{proof}  From Lemma \ref{tedious} we see that the described function $\phi_0$ is well-defined and independent of the numerous choices made.  We must check that $\phi_0$ is a homomorphism.

We note first that if $W\in \Pfine(\{W_x\}_{x\in X})$ and $\pindex(W)$ has a first or last element, say $\lambda = \max(\pindex(W))$, then $\phi_0(W) = \phi_0(W\upharpoonright_p \pindex(W)\setminus \{\lambda\})$.  This is easily seen by selecting the set of intervals $I_0, \ldots, I_n$ for $W$ to be such that $I_n = \{\lambda\}$.  The fact that $|I_n| = 1$ and therefore $I_n \notin J$ completes the argument.

Suppose that $W \in \Pfine(\{W_x\}_{x \in X})$ and $W \equiv W_0W_1$ where also both $W_0, W_1 \in \Pfine(\{W_x\}_{x\in X})$.  Choose subintervals $I_0, \cdots I_n$ in $\pindex(W_0)$ as in Lemma \ref{elementsofthegeneratedsubgroup}, let $J = \{0 \leq j \leq n \mid |I_j| > 1\}$, select $x_j \in X$ and $i_j \in \{-1, 1\}$ and intervals $\Lambda_j \subseteq \pindex(W_{x_j})$ with $W \upharpoonright_p I_j \equiv (W_{x_j} \upharpoonright_p \Lambda_j)^{i_j}$.  Similarly choose intervals $I_0', \ldots, I_{n'}'$ in $\pindex(W_1)$ and define $J'$ and choose $y_{j'} \in X$, $m_{j'} \in \{-1, 1\}$ and $\Lambda_{j'}' \subseteq \pindex(W_{y_{j'}})$ for each $j' \in J'$.  Notice that $\pindex(W) \equiv I_0\cdots I_nI_0'\cdots I_{n'}'$ is a decomposition as in Lemma \ref{elementsofthegeneratedsubgroup} and $J \cup J'$ is precisely the set of indices whose accompanying interval is of cardinality at most one.  Then

$$
\begin{array}{ll}
\phi_0(W) & = (\prod_{j \in J} [[(U_{x_j}\varpropto(\Lambda_j, \iota_{x_j}))^{i_j}]]) (\prod_{j' \in J'} [[(U_{y_{j'}}\varpropto(\Lambda_{j'}, \iota_{x_j}))^{m_j}]])\\
& = \phi_0(W_0)\phi_0(W_1)
\end{array}
$$

Next we suppose that $W \in \Pfine(\{W_x\}_{x \in X})$ and let subintervals $I_0, \cdots I_n$ in $\pindex(W_0)$ be as in Lemma \ref{elementsofthegeneratedsubgroup}, let $J = \{0 \leq j \leq n \mid |I_j| > 1\}$, select $x_j \in X$ and $i_j \in \{-1, 1\}$ and intervals $\Lambda_j \subseteq \pindex(W_{x_j})$ with $W \upharpoonright_p I_j \equiv (W_{x_j} \upharpoonright_p \Lambda_j)^{i_j}$.  Notice that $\pindex(W^{-1})$ may be written as $\pindex(W^{-1}) \equiv I_n'\cdots I_0'$ as in Lemma \ref{elementsofthegeneratedsubgroup}, where $I_j'$ is order isomorphic to the ordered set $(I_j)^{-1}$, and $W \upharpoonright_p I_j \equiv (W^{-1} \upharpoonright_p I_j')^{-1}$.  Also, $\{0 \leq j \leq n \mid |I_j'| > 1\}$ is equal to the set $J$.  Then

$$
\begin{array}{ll}
\phi_0(W) & = \prod_{j \in J} [[(U_{x_j}\varpropto(\Lambda_j, \iota_{x_j}))^{i_j}]]\\
& = (\prod_{j \in J^{-1}}  [[(U_{x_j}\varpropto(\Lambda_j, \iota_{x_j}))^{- i_j}]])^{-1}\\
& = (\phi_0(W^{-1}))^{-1}
\end{array}
$$

\noindent where we use $J^{-1}$ to denote the set $J$ under the reverse order.  Thus $\phi_0(W^{-1}) = (\phi_0(W))^{-1}$.

Finally we let $W_0, W_1 \in \Pfine(\{W_x\}_{x \in X})$ be given.  As in Lemma \ref{reduced} we write $W_0 \equiv W_{00}W_{01}$ and $W_1 \equiv W_{1 0}W_{11}$ with $W_{01} \equiv W_{10}^{-1}$ and the word $W_{00}W_{11}$ reduced.  We will give the argument in the most difficult case and sketch how the argument goes in the less difficult ones.  Suppose that $W_{00}$ ends with a nonempty $\alpha$-pure word and $W_{11}$ begins with a nonempty $\alpha$-pure word, and also that $W_{01}$ begins with a nonempty $\alpha$-pure word.  From this last assumption we know that $W_{10}$ ends with a nonempty $\alpha$-pure word.

We have $W_{00}W_{11} \equiv W_{00}'W_aW_{11}'$ where we denote $\lambda_0 = \max(\pindex(W_{00}))$, $\lambda_1 = \min(\pindex(W_{11}))$ and

$$
\begin{array}{ll}
W_{00}' & \equiv W_{00} \upharpoonright_p \{\lambda \in \pindex(W_0) \mid \lambda < \lambda_0\}\\
W_{11}' & \equiv W_1 \upharpoonright_p \{\lambda \in \pindex(W_1) \mid \lambda > \lambda_1\}\\
W_a & \equiv (W_{00} \upharpoonright_p \{\lambda_0\})(W_{11} \upharpoonright_p \{\lambda_1\})
\end{array}
$$

Note that $W_{00}', W_a, W_{11}' \in \Pfine(\{W_x\}_{x\in X})$, whereas for example $W_{00} \upharpoonright_p \{\lambda_0\}$ might not be in $\Pfine(\{W_x\}_{x\in X})$.  Furthermore let $\lambda_2 = \min(\pindex(W_{01}))$ and $\lambda_3 = \max(\pindex(W_{10}))$ and define

$$
\begin{array}{ll}
W_{01}' & \equiv W_{01} \upharpoonright_p (\pindex(W_{01}) \setminus \{\lambda_2\})\\
W_b & \equiv (W_{00}\upharpoonright_p \{\lambda_0\})(W_{01}\upharpoonright_p \{\lambda_2\})\\
W_{10}' & \equiv W_{10} \upharpoonright_p (\pindex(W_{10}) \setminus \{\lambda_3\})\\
W_c & \equiv (W_{10}\upharpoonright_p \{\lambda_3\})(W_{11}\upharpoonright_p \{\lambda_1\})
\end{array}
$$

\noindent Notice that $W_{01}' \equiv (W_{10}')^{-1}$ and that $W_{01}', W_b, W_{10}', W_c \in \Pfine(\{W_x\}_{x\in X})$.

By our work so far we get

$$
\begin{array}{ll}
\phi_0(W_{00}W_{11}) & = \phi_0(W_{00}'W_aW_{11}')\\
& = \phi_0(W_{00}')\phi_0(W_a)\phi_0(W_{11}')\\
& = \phi_0(W_{00}')\phi_0(W_{11}')\\
& = \phi_0(W_{00}')\phi_0(W_{01}')\phi_0(W_{10}')\phi_0(W_{11}')\\
& = \phi_0(W_{00}')\phi_0(W_b)\phi_0(W_{01}')\phi_0(W_{10}')\phi_0(W_c)\phi_0(W_{11}')\\
& = \phi_0(W_0)\phi_0(W_1).
\end{array}
$$

In the simpler case where $W_{01}$ does not begin with an $\alpha$-pure word (hence $W_{10}$ does not end with an $\alpha$-pure word) we let $W_{01}' = W_{01}$, $W_{10}' = W_{10}$ and both $W_b$ and $W_c$ be the empty word and the equalities above will all hold.  In the case there does not exist $\alpha < \kappa_0$ such that both $W_{00}$ ends with a nonempty $\alpha$-pure word and $W_{11}$ begins with an $\alpha$-pure word we let $W_{00}' \equiv W_{00}$, $W_{11}' \equiv W_{11}$ and $W_a \equiv E$.  It may still be the case that $W_{00}$ ends with a nonempty $\beta$-pure word and $W_{01}$ begins with a nonempty $\beta$-pure word, $\beta < \kappa_0$, and for this we define

$$
\begin{array}{ll}
W_{01}' & \equiv W_{01} \upharpoonright_p (\pindex(W_{01}) \setminus \{\lambda_2\})\\
W_b & \equiv (W_{00}\upharpoonright_p \{\lambda_0\})(W_{01}\upharpoonright_p \{\lambda_2\})\\
W_{10}' &  \equiv W_{10} \upharpoonright_p (\pindex(W_{10}) \setminus \{\lambda_3\})
\end{array}
$$

\noindent and let $W_c$ be given by 

\[
\left\{
\begin{array}{ll}
(W_{10}\upharpoonright_p \{\lambda_3\})(W_{11}\upharpoonright_p \{\lambda_1\})
                                            & \text{in case } W_{11} \text{ begins with a nonempty } \beta\text{-pure}\\
& \text{word and }\lambda_3 = \min\pindex(W_{11});\\
W_{10}\upharpoonright_p \{\lambda_3\}                                        & \text{otherwise}.
\end{array}
\right.
\]

\noindent The case where $W_{11}$ and $W_{10}$ respectively begin and end with a $\beta$-pure word, for some $\beta < \kappa_0$ is analogous.  If none of these cases holds then we simply let $W_{00}' \equiv W_{00}$, $W_{01}' \equiv W_{01}$, $W_{10}' \equiv W_{10}$, $W_{11}' \equiv W_{11}$ and $W_a \equiv W_b \equiv W_c \equiv E$.  This exhausts all possibilities and the proof is complete (the arguments for $\phi_1$ are made in the analogous way).

\end{proof}

\begin{proposition}\label{obtainediso}  The homomorphisms $\phi_0$ and $\phi_1$ descend respectively to isomorphisms 

\begin{center}
$\Phi_0: \beth_0(\Pfine(\{W_x\}_{x\in X})) \rightarrow \beth_1(\Pfine(\{U_x\}_{x\in X}))$

$\Phi_1: \beth_1(\Pfine(\{U_x\}_{x\in X})) \rightarrow \beth_0(\Pfine(\{W_x\}_{x\in X}))$
\end{center}

\noindent with $\Phi_0 = \Phi_1^{-1}$.
\end{proposition}

\begin{proof}  If $W \in \Pfine(\{W_x\}_{x \in X})$ is a pure word the set $\pindex(W)$ is a singleton and for any decomposition of $\pindex(W)$ by Lemma \ref{elementsofthegeneratedsubgroup} the accompanying set $J$ will necessarily be empty.  Thus all pure words in $\Pfine(\{W_x\}_{x \in X})$ are in $\ker(\phi_0)$ and so we get the induced $\Phi_0$, and similarly we obtain an induced $\Phi_0$.

Notice that by Lemma \ref{elementsofthegeneratedsubgroup} each element of the group $\beth_0(\Pfine(\{W_x\}_{x\in X}))$ may be written as a product $[[W_0]][[W_1]] \cdots [[W_n]]$ where each $W_i$ is an element in $(\bigcup_{x\in X}\pchunk(W_x))^{\pm 1}$.  For each $0 \leq j \leq n$ we select $x_j$ and $i_j$  and interval $\Lambda_j \subseteq \pindex(W_{x_j})$ such that $W_j \equiv (W_{x_j}\upharpoonright_p \Lambda_j)^{i_j}$.  Now

$$
\begin{array}{ll}
\Phi_1 \circ \Phi_0([[W_0]] \cdots [[W_n]]) & = \prod_{j=0}^n\Phi_1[[(U_{x_j}\upharpoonright_p \varpropto(\Lambda_j, \iota_{x_j}))^{i_j}]]\vspace*{2mm}\\
& = \prod_{j = 0}^n (\Phi_1([[U_{x_j}\upharpoonright_p \varpropto(\Lambda_j, \iota_{x_j})]]))^{i_j}\vspace*{2mm}\\
& = \prod_{j = 0}^n [[W_{x_j} \upharpoonright_p \varpropto(\varpropto(\Lambda_j, \iota_{x_j}), \iota_{x_j}^{-1})]]^{i_j}\vspace*{2mm}\\
& = \prod_{j = 0}^n [[W_{x_j} \upharpoonright_p \Lambda_j]]^{i_j}\vspace*{2mm}\\
& = \prod_{j = 0}^n [[W_j]]
\end{array}
$$

\noindent where the fourth equality holds by Lemma \ref{almostidentified} (the word $W_{x_j} \upharpoonright_p \varpropto(\varpropto(\Lambda_j, \iota_{x_j}), \iota_{x_j}^{-1})$ is obtained from the word $W_{x_j} \upharpoonright_p \Lambda_j$ by deleting finitely many pure subwords, namely those associated with the set $\Lambda_j \setminus \varpropto(\varpropto(\Lambda_j, \iota_{x_j}), \iota_{x_j}^{-1})$.)  Thus $\Phi_1 \circ \Phi_0$ is the identity map, and that $\Phi_0 \circ \Phi_1$ is also the identity map follows from the same reasoning.  The proposition is proved.

\end{proof}
\end{subsection}

\begin{subsection}{Extensions of coherent collections}\label{extendingcoherence}

By Proposition \ref{obtainediso}, the problem of finding an isomorphism between cone groups is reduced to that of finding a coherent collection of coi triples $\{\coi(W_x, U_x, \iota_x)\}_{x\in X}$ such that $\beth_0(\Pfine(\{W_x\}_{x\in X})) = \Co_{\kappa_0}$ and $\beth_1(\Pfine(\{U_x\}_{x\in X})) = \Co_{\kappa_1}$.  Thus, in this and all remaining subsections we approach the problem of extending collections of coi triples.  We still assume that $\kappa_0, \kappa_1 \geq 2$ and that the coi collections are from $\Red_{\kappa_0}$ to $\Red_{\kappa_1}$.

\begin{lemma}\label{findsomerepresentative}   Let $\{\coi(W_x, \iota_x, U_x)\}_{x\in X}$ be coherent.  If $W \in \Pfine(\{W_x\}_{x\in X})$ then there exists a $U \in \Red_{\kappa_1}$ and coi $\iota$ from $W$ to $U$ such that $\{\coi(W_x, \iota_x, U_x)\}_{x\in X} \cup \{(W, \iota, U)\}$ is coherent.  Moreover if $W$ is nonempty the domain (and range) of $\iota$ can be made to be nonempty.
\end{lemma}

\begin{proof}  If $W$ is empty then we let $U$ and $\iota$ be empty.  Else we choose subintervals $I_0, \cdots I_n$ in $\pindex(W)$ as in Lemma \ref{elementsofthegeneratedsubgroup}, let $J = \{0 \leq j \leq n \mid |I_j| > 1\}$, select $x_j \in X$ and $i_j \in \{-1, 1\}$ and intervals $\Lambda_j \subseteq \pindex(W_{x_j})$ with $W \upharpoonright_p I_j \equiv (W_{x_j} \upharpoonright_p \Lambda_j)^{i_j}$.  Let $J' \subseteq J$ be given by

\begin{center}

$J' = \{j\in J\mid (U_{x_j} \upharpoonright_p \varpropto(\Lambda_j, \iota_{x_j}))^{i_j} \not\equiv E\}$.
\end{center}

\noindent For each $j \in J'$ let $U_j' \equiv (U_{x_j} \upharpoonright_p \varpropto(\Lambda_j, \iota_{x_j}))^{i_j}$.  For every $0 \leq j \leq n$ with $j\notin J'$ we let $U_j' \equiv a_{0, 0}$.

The word $\prod_{j = 0}^n U_j'$ is probably not reduced, and so we will make slight modifications in order to obtain a reduced word.  We know that each subword $U_j'$ is reduced and nonempty.  Let $U_n \equiv U_n'$.  Let $0 \leq j < n$ be given.  There are a couple of possibilities:

\begin{itemize}
\item $\pindex(U_j')$ has a maximal element and $\pindex(U_{j + 1}')$ has a minimal element and both $U_j'\upharpoonright_p \{\max\pindex(U_j')\}$ and $U_{j + 1}' \upharpoonright_p \{\min\pindex(U_{j + 1}')\}$ are $\alpha$-pure for some $\alpha < \kappa_1$;

\item $\pindex(U_j')$ has a maximal element and $\pindex(U_{j + 1}')$ has a minimal element and both $U_j'\upharpoonright_p \{\max\pindex(U_j')\}$ and $U_{j + 1}' \upharpoonright_p \{\min\pindex(U_{j + 1}')\}$ are not $\alpha$-pure for some $\alpha < \kappa_1$; or

\item $\pindex(U_j')$ does not have a maximal element and $\pindex(U_{j + 1}')$ does not have a minimal element.
\end{itemize}

\noindent In the middle case we let $U_j \equiv U_j'$.  In the first or last case we choose $\alpha_j' < \kappa_1$ such that $U_j'$ does not end with an $\alpha_j'$-pure word (here we are using the fact that $\kappa_1 \geq 2$) and let $U_j \equiv U_j'a_{\alpha', 0}$.  The word $U_jU_{j + 1}'$ is reduced, and so the word $U_jU_{j+1}$ is reduced (since $U_{j + 1}$ is nonempty), and so the word $U \equiv \prod_{j = 0}^n U_j$ is reduced.

We now define the coi $\iota$ from $W$ to $U$ in a very natural way.  If $j \in J'$ then we let the domain of $\iota_{x_j}$ be $\Lambda_j'$, and in particular $\Close(\Lambda_j', \pindex(W_{x_j}))$.  Let $\Lambda_j'' \subseteq I_j$ be the image of $\Lambda_j' \cap \Lambda_j$ under the order isomorphism given by $W \upharpoonright_p I_j \equiv (W_{x_j} \upharpoonright_p \Lambda_j)^{i_j}$.  Similarly we let $\Theta_j'' \subseteq \pindex(U_j') \subseteq \pindex(U_j)$ be the image of $\iota(\Lambda_j \cap \Lambda_j')$ under the order isomorphism given by $U_j' \equiv (U_{x_j} \upharpoonright_p \varpropto(\Lambda_j, \iota_{x_j}))^{i_j}$.  Define $\iota_j: \Lambda_j'' \rightarrow \Theta_j''$ to be the order isomorphism given by the restriction to $\Lambda_j''$ of the composition of the order isomorphism given by $W \upharpoonright_p I_j \equiv (W_{x_j} \upharpoonright_p \Lambda_j)^{i_j}$ with $\iota$ with the order isomorphism given by $(U_{x_j} \upharpoonright_p \varpropto(\Lambda_j, \iota_{x_j}))^{i_j} \equiv U_j'$.  It is easy to check that $\Close(\Lambda_j'', I_j)$, $\Close(\Theta_j'', \pindex(U_j))$.

If $0\leq j \leq n$ and $j \notin J'$ then $I_j$ is finite and nonempty, as is $\pindex(U_j)$, and we simply select elements $\lambda \in I_j$ and $\lambda' \in \pindex(U_j)$ and let $\Lambda_j'' = \{\lambda\}$, $\Theta_j'' = \{\lambda_j'\}$ and $\iota_j: \Lambda_j'' \rightarrow \Theta_j''$ be the unique function.  Clearly $\Close(\Lambda_j'', I_j)$, $\Close(\Theta_j'', \pindex(U_j))$.

Let $\Lambda'' = \bigcup_{j=0}^n \Lambda_j''$ and $\Theta'' = \bigcup_{j = 0}^n \Theta_j''$, and notice that $\Close(\Lambda'', \pindex(W))$ and $\Close(\Theta'', \pindex(U))$ by Lemma \ref{basiccloseproperties} (iii).  Let $\iota: \Lambda'' \rightarrow \Theta''$ be the unique extension of the $\iota_j$.  Now $\coi(W, \iota, U)$.

We check that $\{\coi(W_x, \iota_x, U_x)\}_{x\in X} \cup \{\coi(W, \iota, U)\}$ is coherent.  Suppose that $y\in X$ and intervals $I \subseteq \pchunk(W)$ and $I' \subseteq \pchunk(W_y)$ and $i \in \{-1, 1\}$ are such that $W \upharpoonright_p I  \equiv (W_y \upharpoonright_p I')^i$.  Let $L \subseteq \{0, \ldots, n\}$ denote the set of those $j$ such that $I_j \cap I \neq \emptyset$.  For each $j \in L \cap J$ we have $W\upharpoonright_p (I_j \cap I) \equiv (W_{x_j}\upharpoonright_p \Lambda_j^*)^{i_j}$ for the obvious choice of interval $\Lambda_j^* \subseteq \Lambda_j \subseteq \pchunk(W_{x_j})$.  Thus $(W_{x_j}\upharpoonright_p \Lambda_j^*)^{i\cdot i_j} \equiv W_y \upharpoonright_p I_j'$ for the obvious choice of interval $I_j' \subseteq I'$.  By the coherence of $\{\coi(W_x, \iota_x, U_x)\}_{x\in X}$ we therefore have

$$
\begin{array}{ll}
[[U \upharpoonright_p \varpropto(I, \iota)]] & = \prod_{j \in L}[[U \upharpoonright_p \varpropto(I_j \cap I, \iota)]]\\
& = \prod_{j\in L \cap J'} [[U \upharpoonright_p \varpropto(I_j \cap I, \iota)]]\\
& = \prod_{j \in L\cap J'} [[U_{x_j} \upharpoonright_p \varpropto(\Lambda_j^*, \iota_{x_j})]]^{i_j}\\
& = \prod_{j\in (L \cap J')^i} [[U_{y}\upharpoonright_p \varpropto(I_j', \iota_y)]]^i\\
& = [[(U_y \upharpoonright_p \varpropto(I', \iota_y))^{i}]].
\end{array}
$$

If we select intervals $I, I' \subseteq \pindex(W)$ and $i\in \{-1, 1\}$ such that $W \upharpoonright_p I \equiv (W \upharpoonright_p I')^i$ then a similar strategy of finitely decomposing $I$ and $I'$ is employed to show $[[U\upharpoonright (I, \iota)]] = [[(U\upharpoonright_p (I', \iota))^{i}]]$.

The check that if $U \upharpoonright_p Q \equiv (U_z \upharpoonright_p Q')^i$, where $z\in X$, then the appropriate elements of $\Co_{\kappa_0}$ are equal is similar to that above, with slight modifications (note that although $U \notin \Pfine(\{U_x\}_{x\in X}))$ is possible, the word $U$ is a finite concatenation of words in $\Pfine(\{U_x\}_{x \in X})$ and pure words).  Similarly if $Q, Q' \subseteq \pindex(U)$, and the proof is complete.
\end{proof}

We introduce some extra notation for convenience.  For a not necessarily reduced word $W$ we let 

\begin{center}
$\|W\| = \sup\{\frac{1}{n+1} \mid n = \proj_1(W(i)) \text{ for some }i\in \overline{W}\}$
\end{center}

\noindent where the supremum is considered in the set of nonnegative reals.  As examples we have $\|E\| = 0$ and $\|a_{\alpha, 5}^{-1}a_{\alpha', 10}\| = \frac{1}{6}$.  By comparison to earlier notation, we have $d(W) = \frac{1}{\|W\|} - 1$.

\begin{lemma}\label{makeitsmaller}  Suppose that $\kappa_0$ and $\kappa_1$ are cardinal numbers greater than or equal to $2$.  Suppose that $\{\coi(W_x, \iota_x, U_x)\}_{x\in X}$ is coherent, $z \in X$ and that $\epsilon>0$ is a real number.  Then there exists a $U \in \Red_{\kappa_1}$ with $\|U\| <\epsilon$ and coi $\iota$ from $W_z$ to $U$ such that $\{\coi(W_x, \iota_x, U_x)\}_{x\in X} \cup \{\coi(W_z, \iota, U)\}$ is coherent.  Moreover the domain (and codomain) of $\iota$ may be chosen to be nonempty provided $\iota_z$  satisfies this property.
\end{lemma}

\begin{proof}  If $W_z$ is empty then let $U$ be empty and $\iota = \emptyset$.  Otherwise let  $U_z \equiv_p \prod_{\lambda \in \pindex(U_z)} U_{\lambda}$ and $J = \{\lambda \in \pindex(U_z) \mid \|U_{\lambda}\| \geq \epsilon\}$.  Since $U_z$ is a word, we know that $J$ is finite.  Let $N \in \omega$ be large enough that $\frac{1}{N+1} < \epsilon$.  For each $\lambda \in \pindex(U_z)$ we let

\[
U_{\lambda}' \equiv \left\{
\begin{array}{ll}
U_{\lambda}
                                            & \text{if } \lambda \notin J, \\
a_{\alpha, N}                                        & \text{if }\lambda \in J\text{ and }U_{\lambda}\text{ is }\alpha\text{-pure}.
\end{array}
\right.
\]

We let $U \equiv \prod_{\lambda \in \pindex(U_x)} U_{\lambda}'$.  It is easy to see that $U$ is reduced (a cancellation in $U$ would necessarily include the pairing of a letter $a_{\alpha, N} \equiv U_{\lambda}$, with $\lambda \in J$, with a letter in $U_{\lambda'}'$ where $\lambda'$ is the immediate successor or immediate predecessor of $\lambda$ in $\pindex(U_x)$, and thus $U_{\lambda}'$ and $U_{\lambda'}'$ are both $\alpha$-pure, so $U_{\lambda}$ and $U_{\lambda'}$ are as well, a contradiction).  Moreover $U \equiv_p \prod_{\lambda \in \pindex(U_x)} U_{\lambda}'$ and clearly $\|U\| < \epsilon$.  Letting $\iota = \iota_z$ it is immediate that $\iota$ is a coi from $W_z$ to $U$.  The rather intuitive fact that $\{\coi(W_x, \iota_x, U_x)\}_{x\in X} \cup \{\coi(W_z, \iota, U)\}$ is coherent is proved along similar lines used in earlier proofs.
\end{proof}

\begin{lemma}\label{avoidpchunk}  Suppose that $\kappa_1 \geq 2$ and that $|X| < 2^{\aleph_0}$.  Given $N \in \omega \setminus \{0\}$ and ordinal $\alpha < \kappa_1$ there exists an $\alpha$-pure word $U \in \Red_{\kappa_1}$ using only positive letters such that $\|U\| = \frac{1}{N}$, and $U(\max(\overline{U})) = a_{\alpha, N - 1} = U(\min(\overline{U}))$, and $U \notin \Pfine(\{U_x\}_{x\in X})$.
\end{lemma}

\begin{proof}  Assume the hypotheses.  We will let $\overline{U} = [0, 1] \cap \mathbb{Q}$.  It is easy to see that the set of all functions $f: ([0, 1]\cap \mathbb{Q}) \rightarrow \{a_{\alpha, n}\}_{n \geq N - 1}$  such that $f(0) = f(1) = a_{\alpha, N -1}$ and the restriction $f \upharpoonright (0, 1) \cap \mathbb{Q}$ is injective is of cardinality $2^{\aleph_0}$, and each such function is an element of $\Red_{\kappa_1}$ since there are no inverse letters with which to perform a cancellation.  On the other hand we have by Lemma \ref{notmanypures} that there are $< 2^{\aleph_0}$ pure elements in $\Pfine(\{U_x\}_{x\in X})$.  The lemma follows immediately.

\end{proof}

\end{subsection}

\begin{subsection}{$\omega$-type concatenations}\label{omegaconcat}

In this subsection we prove the following:

\begin{proposition}\label{omegawords}  Suppose that $\kappa_0$ and $\kappa_1$ are cardinal numbers greater than or equal to $2$.  Suppose that $\{\coi(W_x, \iota_x, U_x)\}_{x\in X}$ is coherent, that $\pindex(W) \equiv \prod_{n\in \omega} I_n$ with each $I_n \neq \emptyset$, $W \upharpoonright_p I_n \in \Pfine(\{W_x\}_{x\in X})$, and $W \notin \Pfine(\{W_x\}_{x\in X})$.  Suppose also that $|X| < 2^{\aleph_0}$.  Then there exists $U \in \Red_{\kappa_1}$ and coi $\iota$ from $W$ to $U$ such that $\{\coi(W_x, \iota_x, U_x)\}_{x\in X} \cup \{\coi(W, \iota, U)\}$ is coherent.
\end{proposition}

\begin{proof}  For each $n\in \omega$ write $W_n \equiv W\upharpoonright_p I_n$.  As $W_0 \in \Pfine(\{W_x\}_{x\in X})$ is nontrivial we select a word $U_0 \in \Red_{\kappa_1}$ and coi $\iota_0$ from $W_0$ to $U_0$ such that the domain of $\iota_0$ is nonempty and such that $\{\coi(W_x, \iota_x, U_x)\}_{x\in X} \cup \{\coi(W_0, \iota_0, U_0)\}$ is coherent, by Lemma \ref{findsomerepresentative}.  Assuming that $\{\coi(W_i, \iota_j, U_j)\}_{j \leq m}$ have already been chosen such that $\|U_j\| < \frac{\|U_{j - 1}\|}{2}$, each $\iota_j$ has nonempty domain and $\{\coi(W_x, \iota_x, U_x)\}_{x\in X} \cup \{\coi(W_i, \iota_j, U_j)\}_{j \leq m}$ is coherent, we use Lemmas \ref{findsomerepresentative} and \ref{makeitsmaller} to select $U_{m + 1} \in \Red_{\kappa_1}$ and coi $\iota_{m+1}$ from $W_{m + 1}$ to $U_{m + 1}$ so that $\iota_{m + 1}$ has nonempty domain, $\|U_{m+1}\| < \frac{\|U_m\|}{2}$ and $\{\coi(W_x, \iota_x, U_x)\}_{x\in X} \cup \{\coi(W_i, \iota_j, U_j)\}_{j \leq m + 1}$ is coherent.

The collection $\{\coi(W_x, \iota_x, U_x)\}_{x\in X} \cup \{\coi(W_i, \iota_j, U_j)\}_{j \in \omega}$ is coherent by Lemma \ref{ascendingchaincoi}.  For each $j \in \omega$ we will construct a word $V_j \in \Red_{\kappa_1}$ with $1 \leq |\pindex(V_j)| \leq 2$.  Select $\alpha_j < \kappa_1$ such that the word $U_j$ does not end with an $\alpha_j$-pure subword.  This is possible since $\kappa_1 \geq 2$ and $U_j$ can end in at most one pure subword (and might possibly not end in a pure subword).  By Lemma \ref{avoidpchunk} we select an $\alpha_j$-pure word $V_j' \in \Red_{\kappa_1} \setminus  \Pfine(\{U_x\}_{x\in X} \cup \{U_i\}_{i \in \omega})$ which uses only positive letters such that $\|V_j'\| = \|U_j\|$ and $\overline{V_j'}$ has maximum and minimum elements and $V_j'(\max(\overline{V_j'})) = a_{\alpha_j, \|U_j\| - 1} = V_j'(\max(\overline{V_j'}))$.  If $U_{j+1}$ begins with an $\alpha_j$-pure subword then, again by Lemma \ref{avoidpchunk}, select $V_j'' \in \Red_{\kappa_1} \setminus  \Pfine(\{U_x\}_{x\in X} \cup \{U_i\}_{i \in \omega})$ which uses only positive letters such that $\|V_j''\| = \|U_j\|$ and $\overline{V_j''}$ has maximum and minimum elements and $V_j''(\max(\overline{V_j''})) = a_{\alpha_j, \|U_j\| - 1} = V_j''(\max(\overline{V_j''}))$ and $V_j''$ is pure but not $\alpha_j$ pure.  If $U_{j + 1}$ does not begin with an $\alpha_j$-pure subword then let $V_j'' = E$.  Let $V_j = V_j'V_j''$.

We know for each $n\in \omega$ that $U_n$, $V_n'$ and $V_n''$ are each reduced.  By how $V_n'$ was selected, we know that $U_nV_n'$ is reduced since any cancellation would need to pair letters in $V_n'$ with those in $U_n$, and $U_n$ does not end in an $\alpha_j$-pure word.  Similarly, $U_nV_n'V_n'' \equiv U_nV_n$ is reduced.

Since $\|U_nV_n\| \leq \frac{1}{2^n}$ we have that the expression $\prod_{n\in \omega} U_nV_n \equiv U_0V_0U_1V_1 \cdots$ is a word.  Let $\mathcal{S}$ be a cancellation on the word $U =\prod_{n \in \omega} U_nV_n$.  If any elements of $\overline{U_0V_0} \subseteq \overline{U}$ appear in $\mathcal{S}$ then $\max \overline{V_0}$ appears and is paired with some element of $\overline{\prod_{n \geq 1}U_nV_n}$,  since $U_0V_0$ is reduced.  But $\max(\overline{V_0})$ cannot be paired with any element of $\overline{\prod_{n \geq 1}U_nV_n}$ since $\|V_0 \upharpoonright \{\max(\overline{V_0})\}\| >\|\prod_{n \geq 1}U_nV_n\|$.  Thus no elements of $\overline{U_0V_0}$ appear in $\mathcal{S}$.  But by the same token, no elements of $\overline{U_1V_1}$ appear in $\mathcal{S}$, and by induction no elements of any $\overline{U_nV_n}$ can appear.  Thus $\mathcal{S} = \emptyset$ and evidently $U$ is reduced.

We note as well that by how $V_n'$ and $V_n''$ were chosen we can write $\pindex(U) \equiv \prod_{n \in \omega}\pindex(U_n)\pindex(V_n)$, and $1 \leq |\pindex(V_n)| \leq 2$.  Let $\iota$ be the function $\iota = \bigcup_{j \in \omega} \iota_j$.  By Lemma \ref{basiccloseproperties} (iii) the domain of $\iota$ is close in $\pindex(W)$ and the range of $\iota$ is close in $U$, and thus we may write $\coi(W, \iota, U)$.  We will show that $\{\coi(W_x, \iota_x, U_x)\}_{x\in X} \cup \{\coi(W_i, \iota_j, U_j)\}_{j \in \omega} \cup \{\coi(W, \iota, U)\}$ is coherent, from which it will immediately follow that $\{\coi(W_x, \iota_x, U_x)\}_{x\in X} \cup \{\coi(W, \iota, U)\}$ is coherent.

Suppose that $y \in X \cup \omega$, $\Lambda_0 \subseteq \pindex(W)$ and $\Lambda_1 \subseteq \pindex(W_y)$ are intervals and $i \in \{-1, 1\}$ are such that $W \upharpoonright_p \Lambda_0 \equiv (W_y \upharpoonright_p \Lambda_1)^{i}$.  If the set $\{n\in \omega \mid I_n \cap \Lambda_0 \neq \emptyset\}$ is infinite, then by the fact that $\Lambda_0$ is an interval there exist $m \in \omega$ and intervals $I_m', I_m'' \subseteq I_m$, with $I_m'$ possibly empty, such that $I_m \equiv I_m'I_m''$ and $\Lambda_0 \equiv I_m''\prod_{n = m+1}^{\infty}I_n$.  Certainly $(W_y \upharpoonright_p \Lambda_1)^{i} \in \Pfine(\{W_x\}_{x\in X} \cup \{W_n\}_{n\in \omega})$, and since $W_n \in  \Pfine(\{W_x\}_{x\in X}$ for each $n$ we have in fact that $\Pfine(\{W_x\}_{x\in X} \cup \{W_n\}_{n\in \omega}) =  \Pfine(\{W_x\}_{x\in X})$.  Therefore we have $W \upharpoonright_p \Lambda_0 \equiv (W_y \upharpoonright_p \Lambda_1)^{i} \in \Pfine(\{W_x\}_{x\in X})$.  But also $(\prod_{n=0}^{m-1} W_n)W\upharpoonright_pI_m' \in  \Pfine(\{W_x\}_{x\in X})$.  Thus $W \equiv ((\prod_{n=0}^{m-1} W_n)W\upharpoonright_pI_m' ) (W\upharpoonright_p \Lambda_0) \in \Pfine(\{W_x\}_{x\in X})$, contrary to the assumptions of our lemma.

Thus we suppose that $y \in X \cup \omega$, $\Lambda_0 \subseteq \pindex(W)$ and $\Lambda_1 \subseteq \pindex(W_y)$ are intervals and $i \in \{-1, 1\}$ are such that $W \upharpoonright_p \Lambda_0 \equiv (W_y \upharpoonright_p \Lambda_1)^{i}$ and know from this that the set $K = \{n\in \omega \mid I_n \cap \Lambda_0 \neq \emptyset\}$ is finite.  If $K = \emptyset$ then $\Lambda_0 = \emptyset = \Lambda_1$ and $[[U \upharpoonright_p \varpropto(\Lambda_0, \iota)]] = [[E]] = [[(U_y \upharpoonright_p \varpropto(\Lambda_1,\iota_y))^{i}]]$.  If $K$ has cardinality $1$ then we let $K = \{m\}$ and we can write $I_m \equiv I_m'\Lambda_0I_m''$ where either or both of $I_m'$ and $I_m''$ may be empty.  Since $\{\coi(W_x, \iota_x, U_x)\}_{x\in X} \cup \{\coi(W_j, \iota_j, U_j)\}_{j \in \omega}$ is coherent, we have

$$
\begin{array}{ll}
[[U \upharpoonright_p \varpropto(\Lambda_0, \iota)]] & = [[U_m \upharpoonright_p \varpropto(\Lambda, \iota_m)]]\\
& = [[(U_y \upharpoonright_p \varpropto(\Lambda_1,\iota_y))^{i}]].
\end{array}
$$

\noindent If $K$ is of cardinality at least $2$ then we let $m_a$ and $m_b$ be respectively the minimal and maximal elements and write $I_{m_a} \equiv I_{m_a}'I_{m_a}''$, $I_{m_b} \equiv I_{m_b}'I_{m_b}''$ (where either or both of $I_{m_a}'$ and $I_{m_b}''$ may be empty) and $\Lambda_0 \equiv I_{m_a}''I_{m_a + 1} \cdots I_{m_b - 1}I_{m_b}'$.  As $W \upharpoonright_p \Lambda_0 \equiv (W_y \upharpoonright_p \Lambda_1)^{i}$, there exist subintervals $J_0, \ldots, J_{m_b - m_a}$ of $\Lambda_1$ such that $W \upharpoonright_p I_{j} \equiv (W_y \upharpoonright_p J_j- m_a)^i$ for $m_a <  j < m_b$ and $W\upharpoonright_p I_{m_a}'' \equiv (W_y \upharpoonright_p J_0)^i$ and $W \upharpoonright_p I_{m_b}' \equiv (W_y \upharpoonright_p J_{m_b - m_a})^{i}$.  Since $\{\coi(W_x, \iota_x, U_x)\}_{x\in X} \cup \{\coi(W_j, \iota_j, U_j)\}_{j \in \omega}$ is coherent, we have

$$
\begin{array}{ll}
[[U \upharpoonright_p \varpropto(\Lambda_0, \iota)]] & = [[U_{m_a} \upharpoonright_p \varpropto(I_{m_a}'',\iota_{m_a})]][[U_{m_a + 1} \upharpoonright_p \varpropto(I_{m_a + 1}, \iota_{m_a + 1}]]\\
& \cdots [[U_{m_b - 1} \upharpoonright_p \varpropto(I_{m_b - 1}, \iota_{m_b  - 1})]][[U \upharpoonright_p \varpropto(I_{m_b}', \iota_{m_b})]]\\
& = \prod_{j\in \{0, \ldots, m_b - m_a\}^i} [[(U_y \upharpoonright_p \varpropto(J_j, \iota_y))^i]]\\
& = [[(U_y \varpropto(\Lambda_1, \iota_y))^i]].
\end{array}
$$

Suppose now that $\Lambda_0, \Lambda_1 \subseteq \pindex(W)$ are intervals and $i\in \{-1, 1\}$ are such that $W \upharpoonright_p \Lambda_0 \equiv (W \upharpoonright_p \Lambda_1)^i$.  Let $K_0 = \{n \in \omega \mid I_n \cap \Lambda_0 \neq \emptyset\}$ and $K_1 = \{n \in \omega \mid I_n \cap \Lambda_1 \neq \emptyset\}$.

\noindent\textbf{Case 1. $K_0$ is infinite.}  In this case, if $K_1$ is finite then $W\upharpoonright_p \Lambda_0 \in \Pfine(\{W_x\}_{x\in X})$, and we have already seen that this implies $W \in \Pfine(\{W_x\}_{x\in X})$ since $K_0$ is infinite, and this is a contradiction.  Thus $K_1$ must be infinite in this case.  If $i = -1$ then $W\upharpoonright_p \Lambda_0 \equiv (W\upharpoonright_p \Lambda_1)^{-1}$, which implies that the word $W$ ends in a nonempty word $V^{-1}$, where $V \in \pchunk(W_{\min(J_1)})$.  Thus $W$ has a nontrivial terminal subword which is in $\Pfine(\{W_x\}_{x\in X})$, from which we derive a contradiction as before.  Thus $i = 1$ and $W \upharpoonright_p \Lambda_0 \equiv W \upharpoonright_p \Lambda_1$, and both $\Lambda_0$ and $\Lambda_1$ are infinite terminal intervals in $\pindex(W)$.  If without loss of generality $\Lambda_1$ is a proper subinterval of $\Lambda_0$, then since $W \upharpoonright_p \Lambda_0 \equiv W \upharpoonright_p \Lambda_1$ we can select a proper terminal subinterval $\Lambda_2 \subseteq \Lambda_1$ such that $W \upharpoonright_p \Lambda_1 \equiv W \upharpoonright_p \Lambda_2$, and inductively we select proper terminal subinterval $\Lambda_{i + 1} \subseteq \Lambda_i$ with $W \upharpoonright_p \Lambda_i \equiv W \upharpoonright_p \Lambda_{i + 1}$.  Thus, letting $\lambda \in \Lambda_0 \setminus \Lambda_1$ we see that the nonempty $W \upharpoonright_p \{\lambda\}$ occurs infinitely often as a subword of $W$, so that $W$ is not a word, a contradiction.  Thus $\Lambda_0 = \Lambda_1$ and $[[U \upharpoonright_p \varpropto(\Lambda_0, \iota)]] = [[(U\upharpoonright_p \varpropto(\Lambda_1, \iota)^i)]]$.

\noindent\textbf{Case 2. $K_0$ is finite.}  In this case we know that $K_1$ is also finite (by applying the argument in Case 1, since $W \upharpoonright_p \Lambda_1 \equiv (W \upharpoonright_p \Lambda_0)^i$).  Thus $W \upharpoonright_p \Lambda_0 \in \Pfine(\{W_n\}_{n\in \omega})$.  If $K_0 = \emptyset$ then so also $K_1 = \emptyset = \Lambda_0 = \Lambda_1$ and it is easy to see that $[[U\upharpoonright_p \varpropto(\Lambda_0, \iota)]] = [[E]] = [[(U\upharpoonright_p \varpropto(\Lambda_1, \iota))^{i}]]$.  In case $K_0 \neq \emptyset$, from the correspondence $W \upharpoonright_p \Lambda_0 \equiv (W \upharpoonright_p \Lambda_1)^i$ we decompose $\Lambda_0 \equiv \Theta_0\Theta_1\cdots \Theta_m$ and $\Lambda_1 \equiv \Theta_0'\Theta_1'\cdots\Theta_m'$ so that $W \upharpoonright_p \Theta_j \equiv (W \upharpoonright_p \Theta_{f(j)}')^i$ where

\[
f(j) = \left\{
\begin{array}{ll}
j
                                            & \text{if } i = 1, \\
m-j                                        & \text{if }i = -1.
\end{array}
\right.
\]

\noindent and each $\Theta_j$ is a subinterval of one of $I_{\min(K_0)}, \ldots, I_{\max(K_0)}$ and each $\Theta_j'$ is a subinterval of one of $I_{\min(K_1)}, \ldots, I_{\max(K_1)}$.  Let $f_0: \{0, \ldots, m\} \rightarrow \{\min(K_0), \ldots, \max(K_0)\}$ be the non-decreasing surjective function given by $\Theta_j \subseteq I_{f_0(j)}$ and similarly let $f_1: \{0, \ldots, m\} \rightarrow \{\min(K_1), \ldots, \max(K_1)\}$ be given by $\Theta_j' \subseteq I_{f_1(j)}$.  We have

$$
\begin{array}{ll}
[[U\upharpoonright_p \varpropto(\Lambda_0, \iota)]] & = \prod_{j = 0}^m[[U_{f_0(j)}\upharpoonright_p \varpropto(\Theta_j, \iota_{f_0(j)})]]\\
& = \prod_{j = 0}^m [[(U_{f_1(f(j))} \upharpoonright_p \varpropto(\Theta_{f(j)}, \iota_{f_1(f(j))}))^i]]\\
& = [[(U\upharpoonright_p \varpropto(\Lambda_1, \iota))^i]]
\end{array}
$$

\noindent where the first and third equalities hold by performing a deletion of finitely many pure words in $\Red_{\kappa_1}$ and the second equality holds by the coherence of the collection $\{\coi(W_n, \iota_n, U_n)\}_{n\in \omega}$.  This completes Case 2 and this part of the argument.

Next we suppose that $y \in X \cup \omega$, $\Lambda_0 \subseteq \pindex(U)$ and $\Lambda_1 \subseteq \pindex(U_y)$ are intervals and $i \in \{-1, 1\}$ are such that $U \upharpoonright_p \Lambda_0 \equiv (U_y \upharpoonright_p \Lambda_1)^{i}$.  Recalling that $U \equiv \prod_{n\in \omega} (U_nV_n)$ and none of the nonempty p-chunks of $V_n$ are in $\Pfine(\{U_x\}_{x\in X}\cup \{U_n\}_{n\in \omega})$ we see that $\Lambda_0 \subseteq \pindex(U_n)$ for some $n\in \omega$.  From the coherence of $\{\coi(W_n, \iota_n, U_n)\}_{n\in \omega} \cup \{\coi(W_x, \iota_x, U_x)\}_{x\in X}$ it is easy to see that $[[W \upharpoonright_p \varpropto(\Lambda_0, \iota^{-1})] = [[(W_y \upharpoonright_p \varpropto(\Lambda_1,\iota_y^{-1}))^i]]$.

Finally suppose intervals $\Lambda_0, \Lambda_1 \subseteq \pindex(U)$ and $i\in \{-1, 1\}$ are such that $U \upharpoonright_p \Lambda_0 \equiv (U \upharpoonright_p \Lambda_1)^i$.  Recall that $U \equiv \prod_{n\in \omega} U_nV_n$ with $$\pindex(U) \equiv \prod_{n\in \omega} \pindex(U_n)\pindex(V_n)$$ and for all $n\in \omega$ we have $\|U_n\| = \|V_n\| \geq 2 \|U_{n + 1}\|$ and $V_n$ uses only positive letters, satisfies $1 \leq |\pindex(V_n)| \leq 2$ and every nonempty p-chunk of $V_n$ is not an element of $\Pfine(\{U_x\}_{x \in X} \cup \{U_n\}_{n \in \omega})$.

If there exists $\lambda \in \Lambda_0$ and $n\in \omega$ such that $\lambda \in \pindex(V_n)$ then $i = 1$ since every pure p-chunk of $U$ which is not in $\Pfine(\{U_x\}_{x \in X}\{U_n\}_{n \in \omega})$ is a p-chunk in some $V_m$ and therefore has positive letters only.  Furthermore the order isomorphism $h: \Lambda_0 \rightarrow \Lambda_1$ induced by the word equivalence $U \upharpoonright_p \Lambda_0 \equiv U\upharpoonright_p \Lambda_1$ must have $h(\lambda) = \lambda$ , for if $U \upharpoonright_p \{\lambda\}$ is, say, $\alpha$-pure then $U\upharpoonright_p \{\lambda\}$ is the unique $\alpha$-pure p-chunk of $U$ which has value $\|U \upharpoonright_p \{\lambda\}\|$ under the function $\|\cdot\|$.  But this implies that $h$ is the identity function since if, say, $\lambda' < \lambda$ and $h(\lambda') < \lambda'$ then $\lambda' < h^{-1}(\lambda') < h^{-2}(\lambda') < \cdots < \lambda$ and so the word $U \upharpoonright_p \Lambda_0$ has infinitely many disjoint occurrences of subwords equivalent to $U\upharpoonright_p \{\lambda'\}$, which contradicts the fact that $U$ is a word.  Thus $\Lambda_0 = \Lambda_1$ and obviously $[[W \upharpoonright_p \varpropto(\Lambda_0, \iota^{-1})]] = [[W \upharpoonright_p \varpropto(\Lambda_1, \iota^{-1})]]$.

On the other hand if $\Lambda_0 \cap \pindex(V_n) = \emptyset$ for all $n \in \omega$ then $\Lambda_0 \subseteq \pindex(U_m)$ for some $m \in \omega$.  Thus $U \upharpoonright_p \Lambda_0 \in \Pfine(\{U_x\}_{x \in X} \cup \{U_n\}_{n \in \omega})$, so $\Lambda_1 \cap \pindex(V_n) = \emptyset$ for all $n \in \omega$ as well.  Thus $\Lambda_1 \subseteq \pindex(U_{m'})$ for some $m' \in \omega$.  Then

$$
\begin{array}{ll}
[[W\upharpoonright_p \varpropto(\Lambda_0, \iota^{-1})]] & = [[W_m \upharpoonright_p \varpropto(\Lambda_0, \iota_m^{-1})]]\\
& = [[(W_{m'} \upharpoonright_p \varpropto(\Lambda_1, \iota_{m'}^{-1}))^{i}]]\\
& =  [[(W \upharpoonright_p \varpropto(\Lambda_1, \iota^{-1}))^i]]
\end{array}
$$

\noindent since $\{\coi(W_n, \iota_n, U_n)\}_{n\in \omega}$ is coherent.  The proposition is proved.
\end{proof}

\end{subsection}

\begin{subsection}{$\mathbb{Q}$-type concatenations}\label{Qconcat}

In this subsection we will devote our attention to proving the following:

\begin{proposition}\label{Qtype}  Suppose that $\kappa_0$ and $\kappa_1$ are cardinal numbers greater than or equal to $2$.  Suppose that $\{\coi(W_x, \iota_x, U_x)\}_{x\in X}$ is coherent, that $\pindex(W) \equiv \prod_{q \in \mathbb{Q}} I_q$ with each $I_q \neq \emptyset$, $W \upharpoonright_p I_q \in \Pfine(\{W_x\}_{x\in X})$ for each $q \in \mathbb{Q}$, and $W \upharpoonright_p \bigcup\Lambda \notin \Pfine(\{W_x\}_{x\in X})$ for each interval $\Lambda \subseteq \mathbb{Q}$ with more than one point.  Suppose also that $|X| < 2^{\aleph_0}$.  Then there exists $U \in \Red_{\kappa_1}$ and coi $\iota$ from $W$ to $U$ such that $\{\coi(W_x, \iota_x, U_x)\}_{x\in X} \cup \{\coi(W, \iota, U)\}$ is coherent.
\end{proposition}

\begin{proof}  Let $\{W_n\}_{n \in \omega}$ be a list such that for each $q\in \mathbb{Q}$ we have some $n\in \omega$ for which either $W \upharpoonright_p I_q \equiv W_n$ or $W \upharpoonright_p I_q \equiv W_n^{-1}$, and $n \neq n'$ implies $W_n \not\equiv W_{n'} \not\equiv W_n^{-1}$.  Notice that indeed such a list must be infinite, for otherwise there is some $q' \in \mathbb{Q}$ such that $\{q \in \mathbb{Q} \mid W \upharpoonright_p I_q \equiv W \upharpoonright_p I_{q'}\}$ is infinite, which contradicts the fact that $W$ is a word.  By assumption $\{W_n\}_{n \in \omega}\subseteq \Pfine(\{W_x\}_{x \in X})$.  Select $U_0 \in \Red_{\kappa_1}$ and coi $\iota_0$ from $W_0$ to $U_0$ with nonempty domain such that $\{\coi(W_x, \iota_x, U_x)\}_{x\in X} \cup \{\coi(W_0, \iota_0, U_0)\}$ is coherent by Lemma \ref{findsomerepresentative}.  Assuming we have chosen $U_n$ and $\iota_n$ we select $U_{n + 1} \in \Red_{\kappa_1}$ and coi $\iota_{n + 1}$ from $W_{n + 1}$ to $U_{n + 1}$ such that $\|U_{n + 1}\| \leq \frac{1}{2}\|U_n\|$, the domain of $\iota_{n + 1}$ is nonempty, and $\{\coi(W_x, \iota_x, U_x)\}_{x\in X} \cup \{\coi(W_j, \iota_j, U_j)\}_{j = 0}^{n+1}$ is coherent by Lemmas \ref{findsomerepresentative} and \ref{makeitsmaller}.  By Lemma \ref{ascendingchaincoi} the collection $\{\coi(W_x, \iota_x, U_x)\}_{x\in X} \cup \{\coi(W_j, \iota_j, U_j)\}_{n \in \omega}$ is coherent.

For each $m \in \omega$ select ordinals $\alpha_{m, b}, \alpha_{m, c} < \kappa_1$ such that $U_m$ does not begin with an initial subword which is $\alpha_{m, b}$-pure and $U_m$ does not end with a terminal subword which is $\alpha_{m, c}$-pure.  By Lemma \ref{avoidpchunk} we select $\alpha_{m, b}$-pure word $V_{m, b}$ which uses only positive letters such that $\|V_{m, b}\| = \|U_m\|$, and $V_{m, b}(\max(\overline{V_{m, b}})) = a_{\frac{1}{\|U_m\|} - 1, \alpha_{m, b}} = V_{m, b}(\min(\overline{V_{m, b}}))$ and $V_{m, b} \notin \Pfine(\{U_{x \in X}\}_{x\in X} \cup \{U_n\}_{n\in \omega})$.  Similarly select an $\alpha_{m, c}$-pure word $V_{m, c}$ which uses only positive letters such that $\|V_{m, c}\| = \|U_m\|$, and $V_{m, c}(\max(\overline{V_{m, c}})) = a_{\frac{1}{\|U_m\|} - 1, \alpha_{m, c}} = V_{m, c}(\min(\overline{V_{m, c}}))$ and $V_{m, c} \notin \Pfine(\{U_{x \in X}\}_{x\in X} \cup \{U_n\}_{n\in \omega})$.

Define functions $f_0: \mathbb{Q} \rightarrow \omega$ and $f_1: \mathbb{Q} \rightarrow \{\pm 1\}$ by $W \upharpoonright_p I_q \equiv W_{f_0(q)}^{f_1(q)}$.  For each $m \in \omega$ the preimage $f_0^{-1}(m)$ is nonempty (by how the list $\{W_n\}_{n \in \omega}$ was chosen) and finite (since $W$ is a word).  For each $q \in \mathbb{Q}$ let $U_q \equiv (V_{f_0(q), b}U_{f_0(q)}V_{f_0(q), c})^{f_1(q)}$ and $U \equiv \prod_{q\in \mathbb{Q}} U_q$.  Notice that this is a word since for each real number $\epsilon > 0$ the set $\{q \in \mathbb{Q} \mid \|U_q\| \geq \epsilon\}$ is finite.  It is easy to see that each $U_q$ is reduced and that moreover $\pindex(U_{f_0(q)}^{f_1(q)})$ is a subinterval of $\pindex(U_q)$ and $|\pindex(U_q) \setminus \pindex(U_{f_0(q)}^{f_1(q)})| = 2$.

\begin{lemma}\label{QUisreduced}  $U$ is reduced.
\end{lemma}

\begin{proof}  For each $n \in \omega$ we let $J_n = \{q\in \mathbb{Q} \mid \|U_q\| = \frac{1}{n+1}\}$.  We see that each $J_n$ is finite since $U$ is a word.  For any cancellation $\mathcal{S}$ on $U$ we define $L_n(\mathcal{S})$ to be the set of those $q\in J_n$ for which there exists $i\in \overline{U_q}$ which occurs in some ordered pair in $\mathcal{S}$.  Define $L_n'(\mathcal{S}) \subseteq L_n(\mathcal{S})$ to be the set of all $q\in L_n(\mathcal{S})$ for which there exists a unique $q' \in L_n(\mathcal{S})$ such that $\mathcal{S}$ pairs each element in $\overline{U_q}$ with an element in $\overline{U_{q'}}$ and each element in $\overline{U_{q'}}$ with an element in $\overline{U_q}$.  Our strategy will be to assume for contradiction that a nonempty cancellation over $U$ exists and then to inductively modify the cancellation into a cancellation which witnesses a cancellation over $W$, contradicting the reducedness of $W$.

Suppose that $\mathcal{S}_0$ is a nonempty cancellation over $U$ and let $n_0$ be minimal such that $L_{n_0}(\mathcal{S}) \neq \emptyset$.  If $L_{n_0}(\mathcal{S}_0) = L_{n_0}'(\mathcal{S}_0)$ then we write $\mathcal{S}_1 = \mathcal{S}_0$ and move on to the next step of our induction.  If $L_{n_0}(\mathcal{S}_0) \neq L_{n_0}'(\mathcal{S}_0)$ then we write $L_{n_0}(\mathcal{S}_0) \setminus L_{n_0}'(\mathcal{S}_0) = \{q_0, \ldots, q_k\}$ with $q_m < q_{m+1}$ under the ordering on $\mathbb{Q}$.  Define a relation $E$ by writing $E(q_{m_0}, q_{m_1})$, where $q_{m_0}, q_{m_1} \in L_{n_0}(\mathcal{S}_0) \setminus L_{n_0}'(\mathcal{S}_0)$, if there exist $i_0 \in \overline{U_{q_{m_0}}}$ and $i_1 \in \overline{U_{q_{m_1}}}$ such that $\langle i_0, i_1 \rangle \in \mathcal{S}_0$.  Since each $U_q$ is reduced we see that $E(q_m, q_m)$ is false for all $0 \leq m \leq k$.  Also, $E(q_{m_0}, q_{m_1})$ implies that $q_{m_0} < q_{m_1}$ since $\langle i_0, i_1 \rangle \in \mathcal{S}_0$ implies $i_0 < i_1$ in $\overline{U}$.  By how each $U_q$ is defined, we see that $U_q(\min(\overline{U_q})) = U_q(\max(\overline{U_q})) \in \{a_{\alpha_{n_0}, n_0}^{\pm 1}\}$ for each $q \in L_{n_0}(\mathcal{S}_0)$.  For $q' \in \bigcup_{n > n_0}L_n(\mathcal{S}_0)$ we have $\|U_{q'}\| < \frac{1}{n_0 + 1}$.  Since $U_q$ is reduced for each $q\in L_{n_0}(\mathcal{S}_0)$, we see that for each $q \in L_{n_0}(\mathcal{S}_0)$ at least one of $\max(\overline{U_q})$ or $\min(\overline{U_q})$ must appear in some element of $\mathcal{S}_0$.  Moreover, by how $L_n'(\mathcal{S}_0)$ is defined, for each $q \in L_{n_0}(\mathcal{S}_0) \setminus L_{n_0}'(\mathcal{S}_0)$ at least one of $\max(\overline{U_q})$ or $\min(\overline{U_q})$ must appear in $\mathcal{S}_0$ and be paired with some element in $\overline{U_{q'}}$ for some $q' \in  L_{n_0}(\mathcal{S}_0) \setminus (L_{n_0}'(\mathcal{S}_0) \cup \{q\})$.

Thus we see that each $q \in L_{n_0}(\mathcal{S}_0) \setminus L_{n_1}'(\mathcal{S}_0)$ must appear as a first or second coordinate in the relation $E$.  Notice as well that if $E(q_{m_0}, q_{m_1})$ and $E(q_{m_2}, q_{m_3})$ where $q_{m_0} < q_{m_2} \leq q_{m_1}$ then $q_{m_0} < q_{m_2} < q_{m_3} \leq q_{m_1}$ by property (4) of cancellations (see Definition \ref{cancellation}).  Similarly if $E(q_{m_0}, q_{m_1})$ and $E(q_{m_2}, q_{m_3})$ hold and $q_{m_0} \leq q_{m_3} < q_{m_1}$ then we have $q_{m_0} \leq q_{m_2} < q_{m_3} < q_{m_1}$.  Since the set $L_{n_0}(\mathcal{S}_0) \setminus L_{n_1}'(\mathcal{S}_0)$ is finite, we therefore have some $0 \leq m < k$ such that $E(q_m, q_{m+1})$.  Again, since $U_{q_m}$ and $U_{q_{m + 1}}$ are each reduced we must have $\langle \max(\overline{U_m}), \min(\overline{U_{m+1}})\rangle \in \mathcal{S}_0$.  Thus $U_{q_m} \equiv (U_{q_{m+1}})^{-1}$ and we let $f: \overline{U_{q_m}} \rightarrow \overline{U_{q_{m+1}}}$ be an order reversing bijection with $U_{q_{m+1}}(f(i)) =  (U_{q_{m}}(i))^{-1}$ witnessing this equivalence.

We let $\mathcal{S}_0^{(1)}$ be given by

$$
\begin{array}{ll}
\mathcal{S}_0^{(1)} & = \{\langle i_0, i_1\rangle \in \mathcal{S}_0 \mid i_0, i_1 \notin \overline{U_{q_m}}\cup \overline{U_{q_{m+1}}}\}\\
& \cup \{\langle i_0, f(i_0) \rangle \mid i_0 \in \overline{U_{q_m}}\}\\
& \cup \{\langle i_0, i_1 \rangle \in \overline{U} \times \overline{U} \mid (\exists i_2 \in \overline{U_{q_m}}) \langle i_0, i_2\rangle, \langle f(i_2), i_1\rangle \in \mathcal{S}_0\}\\
& \cup \{\langle i_0, i_1 \rangle \in \overline{U} \times \overline{U} \mid (\exists i_2 \in \overline{U_{q_m}})\langle i_1, i_2\rangle, \langle i_0, f(i_2) \rangle \in \mathcal{S}_0\}\\
& \cup \{\langle i_0, i_1\rangle \in \overline{U} \times \overline{U} \mid (\exists i_2 \in \overline{U_{q_m}})\langle i_2, i_1\rangle, \langle f(i_2), i_0\rangle \in \mathcal{S}_0\}.
\end{array}
$$

\noindent It is straightforward to see that $\mathcal{S}_0^{(1)}$ is a cancellation and that $L_n(\mathcal{S}_0^{(1)}) \subseteq L_n(\mathcal{S}_0)$ for all $n\in \omega$.  But also $L_{n_0}'(\mathcal{S}_0^{(1)}) = L_{n_0}'(\mathcal{S}_0) \sqcup \{q_m, q_{m+1}\}$.  Iterating the argument to produce $\mathcal{S}_0^{(2)}$, $\mathcal{S}_0^{(3)}$, etc. so as to make $L_{n_0}'(\mathcal{S}_0^{(j+1)})$ strictly include $L_{n_0}'(\mathcal{S}_0^{(j)})$ and have $L_{n_0}(\mathcal{S}_0^{(j+1)}) \subseteq L_{n_0}(\mathcal{S}_0^{(j)})$, we see, since $L_{n_0}(\mathcal{S}_0)$ is finite, that eventually $L_{n_0}'(\mathcal{S}_0^{(j)}) = L_{n_0}(\mathcal{S}_0^{(j)})$.  Set $\mathcal{S}_1 = \mathcal{S}_0^{(j)}$.

Notice that $\mathcal{S}_1$ does not pair any element of $\overline{U_q}$ with $\overline{U_{q'}}$ when $q \in L_{n_0}(\mathcal{S}_1)$ and $q' \notin L_{n_0}(\mathcal{S}_1)$.  Letting $n_1 \in \omega$ be minimal such that $n_1 > n_0$ and $L_{n_1}(\mathcal{S}_1) \neq \emptyset$ (an $n > n_0$ with  $L_n(\mathcal{S}_1) \neq \emptyset$ must exist since $\mathbb{Q}$ is order dense), we may thus repeat the arguments as before to create $\mathcal{S}_2$ such that $L_{n_1}(\mathcal{S}_2) = L_{n_1}'(\mathcal{S}_2)$ and also $\mathcal{S}_2$ agrees with $\mathcal{S}_1$ on $L_{n_0}(\mathcal{S}_1) = L_{n_0}(\mathcal{S}_2)$. Select $n_2 > n_1$ which is minimal such that $L_{n_2}(\mathcal{S}_2) \neq \emptyset$, produce $\mathcal{S}_3$, and continue this process inductively.  Let $\mathcal{S}_{\infty}$ equal $\{\langle i_0, i_1 \rangle \mid (\exists m\in \omega) i_0, i_1 \in \bigcup_{q \in L_{n_m}}U_q\text{ and }\langle i_0, i_1 \rangle \in \mathcal{S}_{m+1}\}$ and we have $\mathcal{S}_{\infty}$ is a cancellation such that $L_n(\mathcal{S}_{\infty}) = L_n'(\mathcal{S}_{\infty})$ for all $n \in \omega$ and $\mathcal{S}_{\infty} \neq \emptyset$.

But now let $\mathcal{S}' = \{\langle q_0, q_1 \rangle \mid (\exists i_0 \in\overline{U_{q_0}}, i_1 \in \overline{U_{q_1}})\langle i_0, i_1 \rangle \in \mathcal{S}_{\infty}\}$ and notice that $\mathcal{S}'$ is a pairing of a subset of elements in $\mathbb{Q}$ that satisfies the comparable properties (1) - (4) of Definition \ref{cancellation}, and $\langle q_0, q_1 \rangle \in \mathcal{S}'$ implies that $U_{q_0} \equiv (U_{q_1})^{-1}$.  Then $W_{q_0} \equiv (W_{q_1})^{-1}$ for $\langle q_0, q_1 \rangle \in \mathcal{S}'$ and it is easy to use $\mathcal{S}'$ to define a nonempty cancellation $\mathcal{S}$ on $W$, and we have a contradiction.
\end{proof}

Now that we know that $U$ is reduced, it is easy to see that $\pindex(U) \equiv \prod_{q \in \mathbb{Q}} \pindex(U_q) \equiv \prod_{q\in \mathbb{Q}} (\pindex(V_{f_0(q), b})\pindex(U_{f_0(q)})\pindex(V_{f_0(q), c}))^{f_1(q)}$.  We define the coi $\iota$ from $W$ to $U$ in the very natural way using the collection $\{\coi(W_n, \iota_n, U_n)\}_{n\in \omega}$.  Namely let $W_q$ denote the subword $W \upharpoonright_p I_q$, and recall that $W_{f_0(q)}^{f_1(q)} \equiv W_q$ and $U_q \equiv (V_{f_0(q)}U_{f_0(q)}V_{f_0(q)})^{f_1(q)}$.  Let $g:\pindex(U_{f_0(q)}^{f_1(q)}) \rightarrow \pindex(U_q)$ denote the order embedding given by this last equivalence.  Let $\iota_q$ be the function whose domain $\dom(\iota_q)$ is the image of $\dom(\iota_{f_0(q)})$ under the order isomorphism $f: \pindex(W_{f_0(q)}^{f_1(q)}) \rightarrow \pindex(W_q)$, whose image lies in $\pindex(U_q)$ and such that $\iota_q(i) = g \circ \iota_{f_0(q)} \circ f^{-1}(i)$.

Notice that $\iota_q$ is an order isomorphism between its domain and image since $\iota_{f_0(q)}$ is order preserving and exactly one of the following holds:

\begin{itemize}

\item $f$ is an order isomorphism between $\pindex(W_q)$ and $\pindex(W_{f_0(q)})$ and $g$ is an order embedding from $\pindex(U_{f_0(q)})$ to $\pindex(U_q)$;

\item $f$ gives an order reversing bijection between $\pindex(W_q)$ and $\pindex(W_{f_0(q)})$ and $g$ gives an order reversing embedding from $\pindex(U_{f_0(q)})$ to $\pindex(U_q)$.
\end{itemize}

Moreover since $\Close(\dom(\iota_n), \pindex(W_n))$ it is easy to see that the relation $\Close(\dom(\iota_q), \pindex(W_q))$ holds.  Also, since $|\pindex(V_{f_0(q), b})| = 1 = |\pindex(V_{f_0(q), c})|$ we easily see that $\Close(\im(\iota_q), \pindex(U_q))$.  Now we let $\iota$ be the order isomorphism given by $\iota = \bigcup_{q\in \mathbb{Q}} \iota_q$.  By Lemma \ref{basiccloseproperties} (iii) we have $\Close(\dom(\iota), \pindex(W))$ and $\Close(\im(\iota), \pindex(U_q))$, so $\iota$ is a coi from $W$ to $U$.  We check the coherence of $\{\coi(W_x, \iota_x, U_x)\}_{x\in X} \cup \{\coi(W_n, \iota_n, U_n)\}_{n \in \omega} \cup \{\coi(W, \iota, U)\}$, which will immediately imply the coherence of  $\{\coi(W_x, \iota_x, U_x)\}_{x\in X} \cup \{\coi(W, \iota, U)\}$.

Suppose that $x_0\in X \cup \omega$, $\Lambda_0 \subseteq \pindex(W)$ and $\Lambda_1 \subseteq \pindex(W_{x_0})$ are intervals, and $i \in \{-1, 1\}$ are such that $W \upharpoonright_p \Lambda_0 \equiv (W_{x_0} \upharpoonright_p \Lambda_1)^i$.  Notice that $\Lambda_0$ must be a subinterval of some $\pindex(W_q)$ since $\mathbb{Q}$ is order dense, $W \upharpoonright_p \Lambda \notin \Pfine(\{W_x\}_{x\in X})$ for each interval $\Lambda \subseteq \mathbb{Q}$ with more than one point and $(W_{x_0} \upharpoonright_p \Lambda_1)^i \in \Pfine(\{W_x\}_{x\in X} \cup \{W_n\}_{n\in \omega}) = \Pfine(\{W_x\}_{x \in X})$.  But letting $f: \pindex(W_{f_0(q)}^{f_1(q)}) \rightarrow \pindex(W_q)$ be the natural order isomorphism and $\Lambda_0' \subseteq \pindex(W_{f_0(q)}^{f_1(q)})$ be the interval given by $f^{-1}(\Lambda_0)$ it is easy to see that

$$
\begin{array}{ll}
[[U \upharpoonright_p \varpropto(\Lambda_0, \iota)]] & = [[(U_{f_0(q)} \upharpoonright_p \varpropto(\Lambda_0', \iota_{f_0(q)}))^{f_1(q)}]]\\
& = [[(U_{x_0} \upharpoonright_p \varpropto(\Lambda_1, \iota_{x_0}))^i]]
\end{array}
$$

\noindent by how the function $\iota_q$ was defined (for the first equality) and the coherence of $\{\coi(W_x, \iota_x, U_x)\}_{x\in X} \cup \{\coi(W_n, \iota_n, U_n)\}_{n \in \omega}$ (for the second equality).

Next, suppose that $\Lambda_0, \Lambda_1 \subseteq \pindex(W)$ are intervals and $i \in \{-1, 1\}$ are such that $W\upharpoonright_p \Lambda_0 \equiv (W\upharpoonright_p \Lambda_1)^i$.  Let $J_0 = \{q \in \mathbb{Q} \mid \pindex(W_q) \cap \Lambda_0 \neq \emptyset\}$ and $J_1 = \{q\in \mathbb{Q} \mid \pindex(W_q) \cap \Lambda_1 \neq \emptyset\}$.  Clearly each of $J_0$ and $J_1$ are intervals in $\mathbb{Q}$.  If, say $J_0$ is empty or a singleton then $W\upharpoonright_p \Lambda_0 \in \Pfine(\{W_x\}_{x \in X})$, and so $J_1$ is not infinite (since we are assuming $W \upharpoonright_p \Lambda \notin \Pfine(\{W_x\}_{x\in X})$ for each interval $\Lambda \subseteq \mathbb{Q}$ with more than one point.)  Similarly if $J_1$ is empty or a singleton then $J_0$ is finite (hence a singleton or empty).  In case $J_0$ is finite we can argue as before, using the coherence of the collection $\{\coi(W_n, \iota_n, U_n)\}_{n \in \omega}$ to obtain $[[U \upharpoonright_p \varpropto(\Lambda_0, \iota)]] = [[(U \upharpoonright_p \varpropto(\Lambda_1, \iota))^i]]$.

Suppose now that $J_0$ (and therefore also $J_1$) is infinite.  Since $J_0$ is order dense and $W \upharpoonright_p \Lambda \notin \Pfine(\{W_x\}_{x\in X})$ for each interval $\Lambda \subseteq \mathbb{Q}$ with more than one point, we notice that $J_0$ has a minimum if and only if the word $W \upharpoonright_p \Lambda_0$ has a nonempty initial subword which is an element of $\Pfine(\{W_x\}_{x\in X})$.  Also, if $J_0$ has minimum $q$ then $W\upharpoonright_p (\pindex(W_q) \cap \Lambda_0)$ is the maximal initial subword of $W \upharpoonright_p \Lambda_0$ which is an element in $\Pfine(\{W_x\}_{x\in X})$.  Similarly $J_0$ has a maximum if and only if the word $W \upharpoonright_p \Lambda_0$ has a nonempty terminal subword which is an element of $\Pfine(\{W_x\}_{x\in X})$, and if $J_0$ has maximum $q$ then $W\upharpoonright_p (\pindex(W_q) \cap \Lambda_0)$ is the maximal terminal subword of $W \upharpoonright_p \Lambda_0$ which is an element in $\Pfine(\{W_x\}_{x\in X})$.  Let $J_0' \subseteq J_0$ be the subinterval which consists of $J_0$ minus any maximum or minimum that $J_0$ might have.  By similar reasoning, we see that for each $q \in J_0'$ the subword $W_q$ is a maximal subword of $W\upharpoonright_p \Lambda_0$ which is an element of $\Pfine(\{W_x\}_{x\in X})$.

The comparable claims hold for $J_1$; for example $J_1$ has a minimum if and only if the word $W \upharpoonright_p \Lambda_1$ has a nonempty initial subword which is an element of $\Pfine(\{W_x\}_{x\in X})$, and if $q \in J_1$ is minimal then $W\upharpoonright_p (\pindex(W_q) \cap \Lambda_1)$ is the maximal initial subword of $W \upharpoonright_p \Lambda_1$ which is an element in $\Pfine(\{W_x\}_{x\in X})$.  Define the interval $J_1' \subseteq J_1$ similarly.  As $W \upharpoonright_p \Lambda_0 \equiv (W \upharpoonright_p \Lambda_1)^i$ we see that if $i = 1$ 

\begin{itemize}

\item $J_0$ has a minimum if and only if $J_1$ has;

\item $J_0$ has a maximum if and only if $J_1$ has;

\item if $q_0 = \min(J_0)$ and $q_1 = \min(J_1)$ then $W \upharpoonright_p (\Lambda_0 \cap \pindex(W_{q_0})) \equiv W \upharpoonright_p (\Lambda_1 \cap \pindex(W_{q_1}))$;

\item if $q_0 = \max(J_0)$ and $q_1 = \max(J_1)$ then $W \upharpoonright_p (\Lambda_0 \cap \pindex(W_{q_0})) \equiv W \upharpoonright_p (\Lambda_1 \cap \pindex(W_{q_1}))$;

\item there is an order isomorphism $h: J_0' \rightarrow J_1'$ such that $W_{h(q)} \equiv W_{g}$
\end{itemize}

\noindent and if $i = -1$

\begin{itemize}

\item $J_0$ has a minimum if and only if $J_1$ has a maximum;

\item $J_0$ has a maximum if and only if $J_1$ has a minimum;

\item if $q_0 = \min(J_0)$ and $q_1 = \max(J_1)$ then $W \upharpoonright_p (\Lambda_0 \cap \pindex(W_{q_0})) \equiv (W \upharpoonright_p (\Lambda_1 \cap \pindex(W_{q_1})))^{-1}$;

\item  if $q_0 = \max(J_0)$ and $q_1 = \min(J_1)$ then $W \upharpoonright_p (\Lambda_0 \cap \pindex(W_{q_0})) \equiv (W \upharpoonright_p (\Lambda_1 \cap \pindex(W_{q_1})))^{-1}$;

\item there is an order reversing bijection $h: J_0' \rightarrow J_1'$ such that $W_{h(q)} \equiv (W_q)^{-1}$.

\end{itemize}

From this and how the $\iota_q$ were defined it is clear that $$U \upharpoonright_p (\Lambda_0 \cap \bigcup_{q \in J_0'}\pindex(W_q), \iota) \equiv (U \upharpoonright_p (\Lambda_1 \cap \bigcup_{q \in J_1'} \pindex(W_q)))^i.$$  Now if, for example, $i = -1$ and $J_0$ has maximum and minimum then we see that

$$
\begin{array}{ll}
[[U \upharpoonright_p (\Lambda_0, \iota)]]\\
= [[U \upharpoonright_p (\Lambda_0 \cap \pindex(W_{\min(J_0)}), \iota)]][[U \upharpoonright_p (\Lambda_0 \cap \bigcup_{q \in J_0'}\pindex(W_q), \iota)]]\\
\cdot [[U \upharpoonright_p (\Lambda_0 \cap \pindex(W_{\max(J_0)}), \iota)]]\\
= [[U \upharpoonright_p (\Lambda_0 \cap \pindex(W_{\min(J_0)}), \iota)]]][[(U \upharpoonright_p (\Lambda_1 \cap \bigcup_{q \in J_1'}\pindex(W_q), \iota))^{-1}]]\\
\cdot [[U \upharpoonright_p (\Lambda_0 \cap \pindex(W_{\max(J_0)}), \iota)]]\\
=  [[(U \upharpoonright_p (\Lambda_0 \cap \pindex(W_{\max(J_1)}), \iota))^{-1}]]][[(U \upharpoonright_p (\Lambda_1 \cap \bigcup_{q \in J_1'}\pindex(W_q), \iota))^{-1}]]\\
\cdot [[U \upharpoonright_p (\Lambda_0 \cap \pindex(W_{\max(J_0)}), \iota)]]\\
=  [[(U \upharpoonright_p (\Lambda_0 \cap \pindex(W_{\max(J_1)}), \iota))^{-1}]]]\\
\cdot [[(U \upharpoonright_p (\Lambda_1 \cap \bigcup_{q \in J_1'}\pindex(W_q), \iota))^{-1}]][[(U \upharpoonright_p (\Lambda_0 \cap \pindex(W_{\min(J_0)}), \iota))^{-1}]]\\
=  [[(U \upharpoonright_p (\Lambda_1 \cap \bigcup_{q \in J_1'} \pindex(W_q)))^{-1}]]

\end{array}
$$

\noindent where the first and last equalities follow from deleting finitely many pure p-chunks, the second equality follows from $U \upharpoonright_p (\Lambda_0 \cap \bigcup_{q \in J_0'}\pindex(W_q), \iota) \equiv (U \upharpoonright_p (\Lambda_1 \cap \bigcup_{q \in J_1'} \pindex(W_q)))^i$, and the third and fourth follow from the fact that the collection $\{\coi(W_n, \iota_n, U_n)\}_{n \in \omega}$ is coherent.  All other possibilities can be similarly argued.

Next we suppose that $x_0 \in X \cup \omega$ and $\Lambda_0 \subseteq \pindex(U), \Lambda_1 \subseteq \pindex(U_{x_0})$ are intervals and $i \in \{-1, 1\}$ are such that $U \upharpoonright_p \Lambda_0 \equiv (U_{x_0} \upharpoonright_p \Lambda_1)^i$.  As $(U_{x_0} \upharpoonright_p \Lambda_1)^i \in \Pfine(\{U_x\}_{x \in X} \cup \{U_n\}_{n \in \omega})$, and $V_{m, b}, V_{m, c} \notin \Pfine(\{U_x\}_{x \in X} \cup \{U_n\}_{n \in \omega})$ for all $m \in \omega$ we see that $\Lambda_0$ must be a subinterval of some $\pindex(U_q)$, and more particularly a subinterval of $\pindex(U_{f_0(q)}^{f_1(q)})$.  By how $\iota_q$ was defined, and since $\{\coi(W_x, \iota_x, U_x)\}_{x\in X} \cup \{\coi(W_n, \iota_n, U_n)\}_{n\in \omega}$ is coherent it follows that

\begin{center}

$[[W \upharpoonright_p \varpropto(\Lambda_0, \iota^{-1})]] = [[(W_{x_0} \upharpoonright_p \varpropto(\Lambda_1, \iota_{x_0}^{-1}))^i]]$.

\end{center}

Finally, supposing that intervals $\Lambda_0, \Lambda_1 \subseteq \pindex(U)$ and $i \in \{-1, 1\}$ are such that $U\upharpoonright_p \Lambda_0 \equiv (U \upharpoonright_p \Lambda_1)^i$ we define $J_0, J_0', J_1, J_1'$ the same as before.  One sees that  $J_0$ has a minimum if and only if $U \upharpoonright_p \Lambda_0$ has a nonempty initial subword which is a pure p-chunk (i.e. a word $V_{m, b}^{\pm 1}$ or $V_{m, c}^{\pm 1}$ for some $m\in \omega$) or which is in $\Pfine(\{U_x\}_{x \in X} \cup \{U_n\}_{n \in \omega})$, and similar such adjustments for maxima and $J_1$.  Also for each $q \in J_0'$ (or $q \in J_1'$) we have that $U_{f_0(q)}^{f_1(q)}$ is a maximal subword of $U$ which is in $\Pfine(\{U_x\}_{x \in X} \cup \{U_n\}_{n \in \omega})$, and each of $V_{f_0(q), b}^{f_1(q)}$ and $V_{f_0(q), c}^{f_1(q)}$ is a maximal p-chunk of $U$ such that all of the nonempty p-chunks are not in $\Pfine(\{U_x\}_{x \in X} \cup \{U_n\}_{n \in \omega})$.  The bijection $h: J_0' \rightarrow J_1'$ which is an order isomorphism in case $i = 1$, or an order reversal in case $i = -1$, such that $U_{h(q)} \equiv (U_q)^i$ once again can be seen and the argument then follows as before showing that

\begin{center}

$[[W \upharpoonright_p \varpropto(\Lambda_0, \iota^{-1})]] = [[(W \upharpoonright_p \varpropto(\Lambda_1, \iota^{-1}))^i]]$.

\end{center}

\end{proof}

\end{subsection}

\begin{subsection}{Arbitrary extensions}\label{arbitraryext}

In this subsection we will prove the following proposition and then complete the proof of Theorem \ref{bigisomorphism} as well as prove Theorem \ref{elementaryequiv}.

\begin{proposition}\label{arbitraryextensions}   Suppose that $\kappa_0$ and $\kappa_1$ are cardinal numbers greater than or equal to $2$.  Suppose that $\{\coi(W_x, \iota_x, U_x)\}_{x\in X}$ is coherent and that $|X| < 2^{\aleph_0}$.  Then given $W\in \Red_{\kappa_0}$ there exists $U \in \Red_{\kappa_1}$ and coi $\iota$ from $W$ to $U$ such that $\{\coi(W_x, \iota_x, U_x)\}_{x\in X} \cup \{\coi(W, \iota, U)\}$ is coherent.
\end{proposition}

\begin{proof}  Assume the hypotheses.  If $W$ is the empty word $E$ then we let $U \equiv E$ and $\iota$ be the empty function.  This clearly satisfies the conclusion of the proposition.  Thus we may now assume that $W$ is not $E$ and so $\pindex(W)$ is nonempty.  For each $\lambda \in \pindex(W)$ we let $\iota_{\lambda}$ be the empty function, and $\iota_{\lambda}$ is a coi from $W \upharpoonright_p \{\lambda\}$ to $E$.  It is quite trivial to see that $\mathcal{T}_0 = \{\coi(W_x, \iota_x, U_x)\}_{x\in X} \cup \{\coi(W\upharpoonright_p \{\lambda\}, \iota_{\lambda}, E)\}_{\lambda \in \pindex(W)}$ is coherent.  Let $\prec$ be a well-order on the set $\pindex(W)$ and if $\mathcal{T}$ is a collection of coi then we let $h(\mathcal{T})$ denote the set of first words listed in the ordered triples (for example $h(\mathcal{T}_0) = \{W_x\}_{x\in X} \cup \{W\upharpoonright_p \{\lambda\}\}_{\lambda \in \pindex(W)}$).

\noindent \textbf{Step 1.}  We'll define a function $f_0$ from a subset of the set $\aleph_1$ of countable ordinals to $\pindex(W)$, as well as a function $f_1$ with the same domain as $f_0$ and codomain the set of two letters $\{L, R\}$ and $f_2$ a function with the same domain as $f_0$ and codomain the set of intervals in $\pindex(W)$, and also extend the coi collection.  If each $\lambda \in \pindex(W)$ is contained in a maximal interval $I \subseteq \pindex(W)$ such that $W \upharpoonright_p I \in h(\mathcal{T}_{\zeta})$ then we cease our construction and proceed to Step 2.  If it is not the case that each $\lambda \in \pindex(W)$ is contained in a maximal interval $I \subseteq \pindex(W)$ such that $W \upharpoonright_p I \in h(\mathcal{T}_{\zeta})$ then we select a minimal such $\lambda$ under the well-ordering $\prec$ and let $f_0(\zeta) = \lambda$.  At least one of two possibilities must hold:

\noindent \textbf{Case i.}  If there is a sequence $\{I_m\}_{m \in \omega}$ such that $\lambda = \min(I_m)$ and $I_m$ is strictly included in $I_{m+1}$ for all $m \in \omega$ with $W \upharpoonright_p I_m \in \Pfine(h(\mathcal{T_{\zeta}}))$ but $W \upharpoonright_p \bigcup_{m\in \omega} I_m \notin \Pfine(h(\mathcal{T_{\zeta}}))$ then we let $f_1(\zeta) = L$ and $f_2(\zeta) = \bigcup_{m\in \omega} I_m$.  By Proposition \ref{omegawords} we select a $U_{\zeta} \in \Red_{\kappa_1}$ and coi $\iota_{\zeta}$ from $W\upharpoonright_p f_2(\zeta)$ to $U_{\zeta}$ such that $\mathcal{T}_{\zeta + 1} = \mathcal{T}_{\zeta} \cup \{\coi(W\upharpoonright_p f_2(\zeta), \iota_{\zeta}, U_{\zeta})\}$ is coherent.

\noindent \textbf{Case ii.}  If such a sequence as in Case i does not exist then there exists a sequence $\{I_m\}_{m \in \omega}$ such that $\lambda = \max(I_m)$ and $I_m$ is strictly included in $I_{m+1}$ for all $m \in \omega$ with $W \upharpoonright_p I_m \in \Pfine(h(\mathcal{T_{\zeta}}))$ but $W \upharpoonright_p \bigcup_{m\in \omega} I_m \notin \Pfine(h(\mathcal{T_{\zeta}}))$.  In this case we let $f_1(\zeta) = R$ and $f_2(\zeta) = \bigcup_{m\in \omega} I_m$.  By Proposition \ref{omegawords} applied to the word $W^{-1}$ we select a $U_{\zeta} \in \Red_{\kappa_1}$ and coi $\iota_{\zeta}$ from $W\upharpoonright_p f_2(\zeta)$ to $U_{\zeta}$ such that $\mathcal{T}_{\zeta + 1} = \mathcal{T}_{\zeta} \cup \{\coi(W\upharpoonright_p f_2(\zeta), \iota_{\zeta}, U_{\zeta})\}$ is coherent.

Iterating this recursion and letting $\mathcal{T}_{\zeta} = \bigcup_{\zeta_0 < \zeta} \mathcal{T}_{\zeta_0}$ when $\zeta$ is a limit ordinal, we define the functions $f_0, f_1, f_2$ over an increasingly large initial segment of $\aleph_1$.  We claim, however, that this recursion must terminate at some stage, and thus move us into Step 2.  If, otherwise, the recursion did not terminate then the functions $f_0, f_1, f_2$ are defined on all of $\aleph_1$.  Since the codomains, $\pindex(W)$ and $\{L, R\}$, of $f_0$ and $f_1$ are countable there exists some $\lambda \in \pindex(W)$ and, say $R\in \{L, R\}$, and uncountable $J \subseteq \aleph_1$ such that $f_0(J) = \{\lambda\}$ and $f_1(J) = \{R\}$.  But notice that $f_2(\zeta_0)$ is strictly included into $f_2(\zeta_1)$ whenever $\zeta_0, \zeta_1 \in J$ satisfy $\zeta_0 < \zeta_1$, and this is impossible since the set $\pindex(W)$ is countable.

\noindent \textbf{Step 2.}  From Step 1 we obtain a coherent collection $\mathcal{T}_{\zeta}$ of coi, with $| \mathcal{T}_{\zeta}| < 2^{\aleph_0}$, and each $\lambda \in \pindex(W)$ includes into a maximal interval $I_\lambda \subseteq \pindex(W)$ with respect to the property that $W \upharpoonright_p I_{\lambda} \in \Pfine(h(\mathcal{T}_{\zeta}))$.  The collection $\Lambda$ of all such maximal intervals has a natural induced ordering and is necessarily order dense, for if there existed distinct $I_{\lambda}$ and $I_{\lambda'}$ between which there are no elements in $\Lambda$ then the word $W \upharpoonright_p I_{\lambda} \cup I_{\lambda'} \in \Pfine(h(\mathcal{T}_{\zeta}))$, contradicting maximality.  Let $\Lambda'$ be the interval in $\Lambda$ which excludes $\min(\Lambda)$ and $\max(\Lambda)$ if either or both exist.  If $\Lambda'$ is not empty then it is order isomorphic to $\mathbb{Q}$, and in either case by Proposition \ref{Qtype} we may add, if necessary, a single coi triple to $\mathcal{T}_{\zeta}$ to obtain a coherent collection $\mathcal{T}_{\zeta}'$ such that $W\upharpoonright_p (\bigcup \Lambda') \in \Pfine(h(\mathcal{T}_{\zeta}'))$.  Next, since $W \upharpoonright_p \min(\Lambda), W\upharpoonright_p \max(\Lambda) \in \Pfine(h(\mathcal{T}_{\zeta}'))$ if either of $\min(\Lambda)$ or $\max(\Lambda)$ exists, we have that $W \in \Pfine(h(\mathcal{T}_{\zeta}'))$ as $W$ is the concatenation of one or two or three words in $\Pfine(h(\mathcal{T}_{\zeta}'))$.  By Lemma \ref{findsomerepresentative} we select $U \in \Red_{\kappa_1}$ and coi $\iota$ such that $\mathcal{T}_{\zeta}' \cup \{\coi(W, \iota, U)\}$ is coherent.  Then $\{\coi(W_x, \iota_x, U_x)\}_{x\in X} \cup \{\coi(W, \iota, U)\}$ is coherent and our proposition is proved.
\end{proof}

\begin{proof}[Proof of Theorem \ref{bigisomorphism}]  Let $\kappa$ be a cardinal such that $2 \leq \kappa \leq 2^{\aleph_0}$.  It is easy to see from Theorem \ref{howmany} that $|\Red_2| = |\Red_{\kappa}| = 2^{\aleph_0}$.  Thus we let $\prec$ well-order $\Red_2$ in such a way that each element has fewer than $2^{\aleph_0}$ predecessors.  Similarly let $\prec'$ well-order $\Red_{\kappa}$ in such a way that each element has fewer than $2^{\aleph_0}$ predecessors.  We inductively define a coherent collection $\{\coi(W_{\zeta}, \iota_{\zeta}, U_{\zeta})\}_{\zeta < 2^{\aleph_0}}$ of coi triples from $\Red_2$ to $\Red_{\kappa}$.

Recall that each ordinal $\zeta$ may be written uniquely as an ordinal sum $\zeta = \beta + m$ where $\beta$ is either $0$ or a limit ordinal and $m \in \omega$, and so $\zeta$ can be considered even or odd depending on the parity of $m$.  Suppose that we have defined coherent $\{\coi(W_{\zeta}, \iota_{\zeta}, U_{\zeta})\}_{\zeta < \mu}$ for all $\mu < \nu < 2^{\aleph_0}$.  By Lemma \ref{ascendingchaincoi} we know $\{\coi(W_{\zeta}, \iota_{\zeta}, U_{\zeta})\}_{\zeta < \nu}$ is coherent.  If $\nu$ is even then by Lemma \ref{avoidpchunk} we select a word $W_{\nu} \notin \Pfine(\{W_{\zeta}\}_{\zeta < \nu})$ which is minimal such under $\prec$ and by Proposition \ref{arbitraryextensions} select a $U_{\nu} \in \Red_{\kappa}$ and coi $\iota_{\nu}$ such that  $\{\coi(W_{\zeta}, \iota_{\zeta}, U_{\zeta})\}_{\zeta < \nu + 1}$ is coherent (using $\kappa_0 =2$ and $\kappa_1 = \kappa$).  Similarly if $\nu$ is odd then by Lemma \ref{avoidpchunk} we select a word $U_{\nu} \notin \Pfine(\{U_{\zeta}\}_{\zeta < \nu})$ which is minimal such under $\prec'$ and by Proposition \ref{arbitraryextensions} select a $W_{\nu} \in \Red_{\kappa}$ and coi $\iota_{\nu}$ such that  $\{\coi(W_{\zeta}, \iota_{\zeta}, U_{\zeta})\}_{\zeta < \nu + 1}$ is coherent (using $\kappa_0 = \kappa$ and $\kappa_1 = 2$).

Notice that $\Pfine(\{W_{\zeta}\}_{\zeta < 2^{\aleph_0}}) = \Red_2$ and $\Pfine(\{U_{\zeta}\}_{\zeta < 2^{\aleph_0}}) = \Red_{\kappa}$.  Thus by Proposition \ref{obtainediso} we have an isomorphism $\Phi: \Co_2 \rightarrow \Co_{\kappa}$.

\end{proof}

We will derive Theorem \ref{elementaryequiv} as a consequence of Theorem \ref{bigisomorphism}.  Instead of defining the notions of elementary equivalence and elementary subsumption, we will trust the reader to know these concepts or to look them up.  We will rely on the following classical result.

\begin{lemma}\label{test}  Suppose $\mathcal{U}_0$ is a submodel of $\mathcal{U}_1$ such that for every $a_0, \ldots, a_{n-1}\in U_0$ and $a\in U_1$ there exists an automorphism $\phi: \mathcal{U}_1 \rightarrow \mathcal{U}_1$ such that $\phi(a_i) = a_i$ for all $i<n$ and $\phi(a) \in U_0$.  Then $\mathcal{U}_0$ is an elementary submodel of $\mathcal{U}_1$.
\end{lemma}

\begin{proof}[Proof of Theorem \ref{elementaryequiv}]  Certainly if $\gamma = \kappa$ or if $2 \leq \gamma \leq \kappa \leq 2^{\aleph_0}$ then we have $\Co_{\gamma} \simeq \Co_{\kappa}$ (using Theorem \ref{bigisomorphism} in the second case) and the isomorphism is an elementary embedding.  We may therefore assume that $2^{\aleph_0}\leq \gamma < \kappa$, for the result will follow for $2 \leq \gamma < 2^{\aleph_0} < \kappa$ as well by the fact that $\Co_{\gamma} \simeq \Co_{2^{\aleph_0}}$ in this case.

The map $\psi_{\gamma, \kappa}: \Co_{\gamma} \rightarrow \Co_{\kappa}$ given by $[[W]] \mapsto [[W]]$ is easily seen to be an injection and we consider $\Co_{\gamma}$ as the substructure of $\Co_{\kappa}$ consisting of those $[[W]]$ which have a representative utilizing only letters with first coordinate $< \gamma$.  Any bijection $f: \kappa \rightarrow \kappa$ induces a bijection $F_f:\Ao_{\kappa} \rightarrow \Ao_{\kappa}$ given by $a_{\alpha, n}^{\pm 1}\mapsto  a_{f(\alpha), n}^{\pm 1}$ which induces a bijection $\mathcal{F}_f: \W_{\kappa} \rightarrow \W_{\kappa}$ given by $W \mapsto \prod_{i\in \overline{W}}F_f(W(i))$.  This $\mathcal{F}_f$ induces an automorphism $\theta_{f}: \Red_{\kappa} \rightarrow \Red_{\kappa}$ given by $W \mapsto \mathcal{F}_f(W)$ which descends to an automorphism $\overline{\theta_f}: \Co_{\kappa}\rightarrow \Co_{\kappa}$.

\begin{lemma}\label{lotsofauts}  Suppose $\gamma \leq \kappa$ with $\gamma$ uncountable.  If $X \subseteq \Co_{\gamma}$ and $Y \subseteq \Co_{\kappa}$ with $|X|, |Y|<\gamma$ there exists a bijection $f: \kappa \rightarrow \kappa$ such that $\overline{\theta_f}(x) = x$ for all $x\in X$ and $\overline{\theta_f}(Y) \subseteq \Co_{\gamma}$.
\end{lemma}

\begin{proof}  Assume the hypotheses.  For each $x\in X$ fix a representative $W_x \in x$ such that $\pro_0(W) \subseteq \gamma$.  For each $y\in Y$ fix a representative $W_y$.  Since each set $\pro_0(W_x)$ is at most countable, the set $\bigcup_{x\in X}\pro_0(W_x)$ is of cardinality at most $\aleph_0\cdot |X|$.  Similarly the set $\bigcup_{y\in Y} \pro_0(W_y)$ is of cardinality at most $\aleph_0 \cdot |Y|$.  

Since $\gamma$ is uncountable, $\bigcup_{x\in X}\pro_0(W_x) \subseteq \gamma$ is of cardinality less than $\gamma$ and $\bigcup_{y\in Y} \pro_0(W_y)\subseteq \kappa$ is also of cardinality less than $\gamma$, we can easily select a bijection $f: \kappa \rightarrow \kappa$ which fixes the elements in $\bigcup_{x\in X}\pro_0(W_x)$ and such that $f(\bigcup_{y\in Y} \pro_0(W_y)) \subseteq \gamma$.  The automorphism $\overline{\theta_f}$ satisfies the desired properties.
\end{proof}

The proof of Theorem \ref{elementaryequiv} is now complete by appealing to Lemma \ref{test}.

\end{proof}

We note that the map $f\mapsto \overline{\theta_f}$ gives a homomorphic injection from the full symetric group on the set $\kappa$,  $S_{\kappa}$, to the automorphism group $\Aut(\pi_1(\GS_{\kappa}))$.  Since $\pi_1(\GS_2) \simeq \pi_1(\GS_{2^{\aleph_0}})$ we immediately get the following, which is not obvious a priori:

\begin{corollary}  The group $\Aut(\pi_1(\GS_2))$ includes a subgroup isomorphic to the full symmetric group $S_{2^{\aleph_0}}$ on a set of size continuum.
\end{corollary}

This corollary also follows by combining Theorem \ref{bigisomorphism2} of the current paper with \cite[Theorem B]{Cor}.

\end{subsection}
\end{section}

\begin{section}{Theorem \ref{bigisomorphism2}}\label{last}

In this section we prove Theorem \ref{bigisomorphism2}.  Many of the notions and strategies used in the proof of Theorem \ref{bigisomorphism} will be used here with slight adaptions.  In many cases the adaptions are so slight that we will simply state a result and point to the comparable result in Section \ref{isom} for the proof.

Subsection \ref{HAbackground} will give some preliminary setup and notation.  Subsection \ref{basicextension} provides some discussion of elementary earlier results which are revisited in the current setting.  In subsection \ref{omegaconcatnewsetting} we show how to extend a coherent collection of coi triples, in our new setting, so as to include $\omega$-concatenations of words which have already appeared in the collection, and subsection \ref{Qconcatnewsetting} gives the comparable results for $\mathbb{Q}$-concatenations.  In subsection \ref{arbitraryextensionsnewsetting} we prove Theorem \ref{bigisomorphism2}.

\begin{subsection}{Background for $\HA$}\label{HAbackground}  The harmonic archipelago space $\HA$ is a disk in which we have raised thin hills of height $1$ whose hill-bases limit to a single point on the boundary (see Figure \ref{harmonicarchipelagofig} in the introduction).  The fundamental group $\pi_1(\HA)$ admits a description using infinitary words, in similar flavor to the fundamental groups mentioned thus far.  For references and proofs of our characterization the interested reader may consult \cite[Section 2]{CHM}.

We consider the set $\W_c$ of words on the alphabet $\{c_n^{\pm 1}\}_{n \in \omega}$, defined up to $\equiv$.  The $\sim$ equivalence relation on $\W_c$ is defined as before and the group $\W_c/ \sim$ is isomorphic to the fundamental group $\pi_1(\Ea)$ of the Hawaiian earring.  Defining reduced words, cancellations, etc., precisely as before, the set $\Red_c$ of reduced words in $\W_c$ is isomorphic to the group $\W_c/\sim$.  We will say that a word $W \in \Red_c$ is \emph{$m$-pure}, with $m\in \omega$, if all of its letters are in $\{c_m^{\pm 1}\}$; more particularly $W$ is $m$-pure if and only if it is either $E$ or of form $c_m^j$ with $j \in \mathbb{Z} \setminus \{0\}$.  Similarly $W \in \Red_c$ is pure if it is $m$-pure for some $m \in \omega$.  Let $\Pure(\Red_c)$ denote the set of pure subwords of $\Red_c$.

The group $\pi_1(\HA)$ is isomorphic to $\Red_c/\langle\langle \Pure(\Red_c)\rangle\rangle$.  One can visualize this by considering the continuous map from $\Ea$ to $\HA$ where the wedge point of $\Ea$ maps to the point on the boundary of $\HA$ which is the limit of the hill-bases, and the $n$-th circle of $\Ea$ maps so as to move along the boundary, wrap around the $n$-th hill, and then follow the same path backwards along the boundary.  The induced map on the fundamental group produces a surjection to $\pi_1(\HA)$ and the kernel corresponds to $\langle\langle\Pure(\Red_c)\rangle\rangle$.

As was done before, for a word $W \in \Red_c$ we can select maximal nonempty intervals of $\overline{W}$ such that the restriction of $W$ to the interval is pure.  Also, define the p-decomposition of $\overline{W}$, of $W$, $\pchunk(W)$, p-fine, and extend the notation $W \equiv_p \prod_{\lambda \in \Lambda} W_{\lambda}$, $\pindex(W)$, etc. to the words of $\Red_c$ with respect to the words $\Pure(\Red_c)$.  There is little room for confusion between such notions already defined for groups $\Red_{\kappa}$ and these new notions for $\Red_c$ since the words in $\Red_c$ use an alphabet with letter ``c'' and the letters have only one subscript and the letters in $\Red_{\kappa}$ use the letter ``a'' and have two subscripts.  Of course, the motivations for these notions are the same in both cases: from a reduced word we may delete a pure subword, and after reducing we obtain a word which represents a loop which is homotopic to that represented by the original word.

The following hold by the same proofs as their counterparts in subsection \ref{pchunks}, but substituting $m$-pure for some $m \in \omega$ for any mention of $\alpha$-pure.

\begin{lemma}\label{pchunkmultiplicationHA}  Suppose that $W, U \in \Red_c$ are such that $W \equiv_p \prod_{\lambda \in \Lambda} W_{\lambda}$ and $U \equiv_p \prod_{\lambda' \in \Lambda'} U_{\lambda'}$.  Then there exists a (possibly empty) initial interval $I \subseteq \Lambda$, a (possibly empty) terminal interval $I' \subseteq \Lambda'$ such that either:

\begin{enumerate}[(i)] \item $\Red(WU) \equiv_p \prod_{\lambda \in I}W_{\lambda}\prod_{\lambda'\in I'}U_{\lambda'}$; or

\item there exist $\lambda_0\in \Lambda$ which is the least element strictly above all elements in $I$, $\lambda_1\in \Lambda'$ which is the greatest element strictly below all elements of $I'$ and 

\begin{center}
$\Red(WU) \equiv_p (\prod_{\lambda \in I}W_{\lambda})V(\prod_{\lambda'\in I'}U_{\lambda'})$ 
\end{center}

\noindent where $V \equiv \Red(W_{\lambda_0}U_{\lambda_1})\not\equiv E$ is pure.

\end{enumerate}

\end{lemma}

\begin{lemma}\label{elementsofthegeneratedsubgroupHA}  Suppose that $X \subseteq \Red_c$.  For each nonempty element $W$ of the subgroup $\Pfine(X) \leq \Red_c$ if $W \equiv_p \prod_{\lambda \in \Lambda} W_{\lambda}$ then there exist nonempty intervals $I_0, \ldots, I_n$ in $\Lambda$ such that

\begin{enumerate}[(i)]

\item $\Lambda = \prod_{i = 0}^n I_i$; and

\item for each  $0 \leq i \leq n$ at least one of the following holds:

\begin{enumerate}[(a)]

\item $I_i$ is a singleton $\{\lambda\}$ such that $W_{\lambda}$ is the reduction of a finite concatenation of pure p-chunks of elements in $X^{\pm 1}$;

\item $\prod_{\lambda \in I_i} W_{\lambda}$ is a p-chunk of some element in $X^{\pm 1}$.

\end{enumerate}

\end{enumerate}
\end{lemma}

\begin{lemma}\label{fineHA}  If $X \subseteq \Red_c$ then the subgroup $\langle \bigcup_{U\in X} \pchunk(U)\rangle \leq \Red_c$ is p-fine.  This is the smallest p-fine subgroup including the set $X$.
\end{lemma}

The analogue of Lemma \ref{notmanypures} also holds but it is not useful in the setting $\Red_c$ since the set $\Pure(\Red_c)$ of pure words is countable.  This limitation represents the principal difficulty in proving Theorem \ref{bigisomorphism2}.

Let $\beth_c: \Red_c \rightarrow \Red_c/\langle\langle\Pure(\Red_c)\rangle\rangle$ denote the quotient map and $[[W]]$ denote the equivalence class of $W \in \Red_c$ in $\Red_c/\langle\langle\Pure(\Red_c)\rangle\rangle$.  For words $W \in \Red_c$ and $U \in \Red_2$ we'll write, as before, $\coi(W, \iota, U)$ to denote that $\iota$ is a coi between $\pindex(W)$ and $\pindex(U)$ and say that $\coi(W, \iota, U)$ is a \emph{coi triple from $\Red_c$ to $\Red_2$}.  Coherence of a collection of coi triples from $\Red_c$ to $\Red_2$ is defined in the same way as before and the following analogue to Proposition \ref{obtainediso} follows from the same arguments.

\begin{proposition}\label{obtainedisonew}  From a coherent collection $\{\coi(W_x, \iota_x, U_x)\}_{x\in X}$ of coi triples from $\Red_c$ to $\Red_2$ we obtain isomorphisms

\begin{center}
$\Phi_0: \beth_c(\Pfine(\{W_x\}_{x \in X})) \rightarrow \beth_2(\Pfine(\{U_x\}_{x\in X}))$
\end{center}

\noindent and

\begin{center}
$\Phi_1:  \beth_2(\Pfine(\{U_x\}_{x\in X})) \rightarrow \beth_c(\Pfine(\{W_x\}_{x \in X}))$
\end{center}

\noindent such that $\Phi_0 = \Phi_1^{-1}$.

\end{proposition}

\end{subsection}

\begin{subsection}{Some basic extension results}\label{basicextension}

\begin{lemma} \label{findsomerepresentativenew}  Let $\{\coi(W_x, \iota_x, U_x)\}_{x\in X}$ be a coherent collection of coi triples from $\Red_c$ to $\Red_2$.

\begin{enumerate}

\item  If $W \in \Pfine(\{W_x\}_{x\in X})$ then there exists a $U \in \Red_{\kappa_1}$ and coi $\iota$ from $W$ to $U$ such that $\{\coi(W_x, \iota_x, U_x)\}_{x\in X} \cup \{(W, \iota, U)\}$ is coherent.  Moreover if $W$ is nonempty the domain (and range) of $\iota$ can be made to be nonempty.

\item  If $U \in \Pfine(\{U_x\}_{x\in X})$ then there exists a $W \in \Red_{\kappa_1}$ and coi $\iota$ from $W$ to $U$ such that $\{\coi(W_x, \iota_x, U_x)\}_{x\in X} \cup \{(W, \iota, U)\}$ is coherent.  Moreover if $U$ is nonempty the domain (and range) of $\iota$ can be made to be nonempty.

\end{enumerate}

\end{lemma}

\begin{proof}  Claim (1) is proved in the same way as Lemma \ref{findsomerepresentative}, almost word for word.  We prove claim (2), and the reader will notice that the reasoning is quite similar in this case as well.  If $U$ is empty then we let $W$ and $\iota$ be empty.  Else we choose subintervals $I_0, \ldots, I_n$ in $\pindex(U)$ as in Lemma \ref{elementsofthegeneratedsubgroup}, let $J = \{0 \leq j \leq n \mid |I_j| > 1\}$, select $x_j \in X$ and $i_j \in \{-1, 1\}$ and intervals $\Lambda_j \subseteq \pindex(U_{x_j})$ with $U \upharpoonright_p I_j \equiv (U_{x_j} \upharpoonright_p \Lambda_j)^{i_j}$.  Let $J' \subseteq J$ be given by

\begin{center}

$J' = \{j\in J\mid (W_{x_j} \upharpoonright_p \varpropto(\Lambda_j, \iota_{x_j}))^{i_j} \not\equiv E\}$.
\end{center}

\noindent For each $j \in J'$ let $W_j' \equiv (W_{x_j} \upharpoonright_p \varpropto(\Lambda_j, \iota_{x_j}))^{i_j}$.  For every $0 \leq j \leq n$ with $j\notin J'$ we let $W_j' \equiv c_0$.

The word $\prod_{j = 0}^n W_j'$ is probably not reduced, and so we will make slight modifications in order to obtain a reduced word.  We know that each subword $W_j'$ is reduced and nonempty.  Let $W_n \equiv W_n'$.  Let $0 \leq j < n$ be given.  There are a couple of possibilities:

\begin{itemize}
\item $\pindex(W_j')$ has a maximal element and $\pindex(W_{j + 1}')$ has a minimal element and both $W_j'\upharpoonright_p \{\max\pindex(W_j')\}$ and $W_{j + 1}' \upharpoonright_p \{\min\pindex(W_{j + 1}')\}$ are $m$-pure for some $m \in \omega$;

\item $\pindex(W_j')$ has a maximal element and $\pindex(W_{j + 1}')$ has a minimal element and both $U_j'\upharpoonright_p \{\max\pindex(W_j')\}$ and $W_{j + 1}' \upharpoonright_p \{\min\pindex(W_{j + 1}')\}$ are not $m$-pure for some $m \in \omega$; or

\item $\pindex(W_j')$ does not have a maximal element and $\pindex(W_{j + 1}')$ does not have a minimal element.
\end{itemize}

\noindent In the middle case we let $W_j \equiv W_j'$.  In the first or last case we choose $m_j \in \omega$ such that $W_j'$ does not end with an $m_j$-pure word and let $W_j \equiv W_j'c_{m_j}$.  The word $W_jW_{j + 1}'$ is reduced, and so the word $W_jW_{j+1}$ is reduced (since $W_{j + 1}$ is nonempty), and so the word $W \equiv \prod_{j = 0}^n W_j$ is reduced.  Moreover $\pindex(W) \equiv \prod_{j = 0}^n\pindex(W_j)$.

We now define the coi $\iota$ from $W$ to $U$ in a very natural way.  If $j \in J'$ then we let the domain of $\iota_{x_j}$ be $\Lambda_j'$, and in particular $\Close(\Lambda_j', \pindex(W_{x_j}))$.  Let $\Lambda_j'' \subseteq I_j$ be the image of $\Lambda_j' \cap \Lambda_j$ under the order isomorphism given by $W \upharpoonright_p I_j \equiv (W_{x_j} \upharpoonright_p \Lambda_j)^{i_j}$.  Similarly we let $\Theta_j'' \subseteq \pindex(U_j') \subseteq \pindex(U_j)$ be the image of $\iota(\Lambda_j \cap \Lambda_j')$ under the order isomorphism given by $U_j' \equiv (U_{x_j} \upharpoonright_p \varpropto(\Lambda_j, \iota_{x_j}))^{i_j}$.  Define $\iota_j: \Lambda_j'' \rightarrow \Theta_j''$ to be the order isomorphism given by the restriction to $\Lambda_j''$ of the composition of the order isomorphism given by $W \upharpoonright_p I_j \equiv (W_{x_j} \upharpoonright_p \Lambda_j)^{i_j}$ with $\iota$ with the order isomorphism given by $(U_{x_j} \upharpoonright_p \varpropto(\Lambda_j, \iota_{x_j}))^{i_j} \equiv U_j'$.  It is easy to check that $\Close(\Lambda_j'', I_j)$, $\Close(\Theta_j'', \pindex(U_j))$.

If $0\leq j \leq n$ and $j \notin J'$ then $I_j$ is finite and nonempty, as is $\pindex(U_j)$, and we simply select elements $\lambda \in I_j$ and $\lambda' \in \pindex(U_j)$ and let $\Lambda_j'' = \{\lambda\}$, $\Theta_j'' = \{\lambda_j'\}$ and $\iota_j: \Lambda_j'' \rightarrow \Theta_j''$ be the unique function.  Clearly $\Close(\Lambda_j'', I_j)$, $\Close(\Theta_j'', \pindex(U_j))$.

Let $\Lambda'' = \bigcup_{j=0}^n \Lambda_j''$ and $\Theta'' = \bigcup_{j = 0}^n \Theta_j''$, and notice that $\Close(\Lambda'', \pindex(W))$ and $\Close(\Theta'', \pindex(U))$ by Lemma \ref{basiccloseproperties} (iii).  Let $\iota: \Lambda'' \rightarrow \Theta''$ be the unique extension of the $\iota_j$.  Now $\coi(W, \iota, U)$.

We check that $\{\coi(W_x, \iota_x, U_x)\}_{x\in X} \cup \{\coi(W, \iota, U)\}$ is coherent.  Suppose that $y\in X$ and intervals $I \subseteq \pchunk(W)$ and $I' \subseteq \pchunk(W_y)$ and $i \in \{-1, 1\}$ are such that $W \upharpoonright_p I  \equiv (W_y \upharpoonright_p I')^i$.  Let $L \subseteq \{0, \ldots, n\}$ denote the set of those $j$ such that $I_j \cap I \neq \emptyset$.  For each $j \in L \cap J$ we have $W\upharpoonright_p (I_j \cap I) \equiv (W_{x_j}\upharpoonright_p \Lambda_j^*)^{i_j}$ for the obvious choice of interval $\Lambda_j^* \subseteq \Lambda_j \subseteq \pchunk(W_{x_j})$.  Thus $(W_{x_j}\upharpoonright_p \Lambda_j^*)^{i\cdot i_j} \equiv W_y \upharpoonright_p I_j'$ for the obvious choice of interval $I_j' \subseteq I'$.  By the coherence of $\{\coi(W_x, \iota_x, U_x)\}_{x\in X}$ we therefore have

$$
\begin{array}{ll}
[[U \upharpoonright_p \varpropto(I, \iota)]] & = \prod_{j \in L}[[U \upharpoonright_p \varpropto(I_j \cap I, \iota)]]\\
& = \prod_{j\in L \cap J'} [[U \upharpoonright_p \varpropto(I_j \cap I, \iota)]]\\
& = \prod_{j \in L\cap J'} [[U_{x_j} \upharpoonright_p \varpropto(\Lambda_j^*, \iota_{x_j})]]^{i_j}\\
& = \prod_{j\in (L \cap J')^i} [[U_{y}\upharpoonright_p \varpropto(I_j', \iota_y)]]^i\\
& = [[(U_y \upharpoonright_p \varpropto(I', \iota_y))^{i}]].
\end{array}
$$

If we select intervals $I, I' \subseteq \pindex(W)$ and $i\in \{-1, 1\}$ such that $W \upharpoonright_p I \equiv (W \upharpoonright_p I')^i$ then a similar strategy of finitely decomposing $I$ and $I'$ is employed to show $[[U\upharpoonright (I, \iota)]] = [[(U\upharpoonright_p (I', \iota))^{i}]]$.

The check that if $U \upharpoonright_p Q \equiv (U_z \upharpoonright_p Q')^i$, where $z\in X$, then the appropriate elements of $\Red_c/\langle\langle\Pure(\Red_c)\rangle\rangle$ are equal is similar to that above.  Similarly if $Q, Q' \subseteq \pindex(U)$, and the proof is complete.
\end{proof}

\begin{lemma}\label{makeitsmallernew}  Suppose that $\{\coi(W_x, \iota_x, U_x)\}_{x\in X}$ is a coherent collection of coi triples from $\Red_c$ to $\Red_2$, $z \in X$ and that $\epsilon>0$ is a real number.  Then there exists a $U \in \Red_2$ with $\|U\| <\epsilon$ and coi $\iota$ from $W_z$ to $U$ such that $\{\coi(W_x, \iota_x, U_x)\}_{x\in X} \cup \{\coi(W_z, \iota, U)\}$ is coherent.  Moreover the domain (and codomain) of $\iota$ may be chosen to be nonempty provided those of $\iota_z$ are.

Similarly for any $y \in X$ there exists a $W \in \Red_c$ with $\|W\| < \epsilon$ and coi $\iota$ from $W$ to $U_y$ such that $\{\coi(W_x, \iota_x, U_x)\}_{x\in X} \cup \{\coi(W, \iota, U_y)\}$ is coherent, and the domain and codomain of $\iota$ may be chosen to be nonempty provided those of $\iota_y$ are.
\end{lemma}

\begin{proof}  If $W_z$ is empty then let $U$ be empty and $\iota = \emptyset$.  Otherwise let  $U_z \equiv_p \prod_{\lambda \in \pindex(U_z)} U_{\lambda}$ and $J = \{\lambda \in \pindex(U_z) \mid \|U_{\lambda}\| \geq \epsilon\}$.  Since $U_z$ is a word, we know that $J$ is finite.  Let $N \in \omega$ be large enough that $\frac{1}{N+1} < \epsilon$.  For each $\lambda \in \pindex(U_z)$ we let

\[
U_{\lambda}' \equiv \left\{
\begin{array}{ll}
U_{\lambda}
                                            & \text{if } \lambda \notin J, \\
a_{\alpha, N}                                        & \text{if }\lambda \in J\text{ and }U_{\lambda}\text{ is }\alpha\text{-pure}.
\end{array}
\right.
\]

We let $U \equiv \prod_{\lambda \in \pindex(U_x)} U_{\lambda}'$.  It is easy to see that $U$ is reduced (a cancellation in $U$ would necessarily include the pairing of a letter $a_{\alpha, N} \equiv U_{\lambda}$, with $\lambda \in J$, with a letter in $U_{\lambda'}'$ where $\lambda'$ is the immediate successor or immediate predecessor of $\lambda$ in $\pindex(U_x)$, and thus $U_{\lambda}'$ and $U_{\lambda'}'$ are both $\alpha$-pure, so $U_{\lambda}$ and $U_{\lambda'}$ are as well, a contradiction).  Moreover $U \equiv_p \prod_{\lambda \in \pindex(U_z)} U_{\lambda}'$ and clearly $\|U\| < \epsilon$.  Letting $\iota = \iota_z$ it is immediate that $\iota$ is a coi from $W_z$ to $U$.  The rather intuitive fact that $\{\coi(W_x, \iota_x, U_x)\}_{x\in X} \cup \{\coi(W_z, \iota, U)\}$ is coherent is proved along similar lines used in earlier proofs.

Now let $y \in X$ be given.  If $U_y$ is empty then let $W$ and $\iota$ be empty.  Else we write $W_y \equiv_p \prod_{\lambda \in \pindex(W)} W_{\lambda}$ and $J = \{\lambda \in \pindex(W_y) \mid \|W_{\lambda}\| \geq \epsilon\}$, and so $J$ is finite.  Select $N \in \omega$ large enough that $\frac{1}{N + 1} < \epsilon$.  Write $J = \{\lambda_0, \lambda_1, \ldots, \lambda_n\}$ where $\lambda_j < \lambda_{j+1}$ under the order on $\pindex(W_y)$.  Select $m_0 \in \omega$ with $m_0 > N$ such that $W_y \upharpoonright_p \{\lambda \in \pindex(W_y) \mid \lambda < \lambda_0\}$ does not end with a nonempty $m_0$-pure subword and $W_y \upharpoonright_p \{\lambda \in \pindex(W_y) \mid \lambda_0 < \lambda\}$ does not begin with a nonempty $m_0$-pure subword.  If $0 < j < n$ and we have already selected $m_{j - 1}$ then select $m_j \in \omega$ with $m_j > m_{j + 1}$ such that $W_y \upharpoonright_p \{\lambda \in \pindex(W_y) \mid \lambda < \lambda_j\}$ does not end with a nonempty $m_j$-pure subword and $W_y \upharpoonright_p \{\lambda \in \pindex(W_y) \mid \lambda_j < \lambda\}$ does not begin with a nonempty $m_j$-pure subword.  Assuming we have selected $m_j$ for all $0 \leq j < n$ we select $m_n \in \omega$ with $m_n > m_{n-1}$ such that $W_y \upharpoonright_p \{\lambda \in \pindex(W_y) \mid \lambda < \lambda_n\}$ does not end with a nonempty $m_n$-pure subword and $W_y \upharpoonright_p \{\lambda \in \pindex(W_y) \mid \lambda_n < \lambda\}$ does not begin with a nonempty $m_j$-pure subword.  Letting

\[
W_{\lambda}' \equiv \left\{
\begin{array}{ll}
W_{\lambda}
                                            & \text{if }\lambda \notin J, \\
c_{m_j}                                        & \text{if }\lambda = \lambda_j \in J
\end{array}
\right.
\]

\noindent and $W \equiv \prod_{\lambda \in \pindex(W_y)} W_{\lambda}'$ it is easy to see that $W$ is reduced, that the equivalence $W \equiv_p \prod_{\lambda\in \pindex(W_y)}W_{\lambda}'$ holds, and $\|W\| <\epsilon$.  Letting $\iota = \iota_y$ one can easily perform the tedious check that $\{\coi(W_x, \iota_x, U_x)\}_{x\in X} \cup \{\coi(W, \iota, U_z)\}$ is coherent.

\end{proof}

The remaining material in this subsection is geared towards allowing us to define words which avoid a p-fine subgroup.

\begin{definition}  Given a word $W \in \Red_c$ we let $\sigma(W): \pindex(W) \rightarrow \omega$ be defined by letting $\sigma(W)(\lambda) = n$ where $W \upharpoonright_p \{\lambda\}$ is an $n$-pure word.
\end{definition}

\begin{lemma}\label{intervalsareboring}  Suppose that $\Theta$ is a totally ordered set and $f_0: \Theta \rightarrow \omega$ is a function.  Also suppose that $V \in \Red_c$ and $\iota_0, \iota_1: \Theta \rightarrow \pindex(V)$ are order embeddings with $\iota_0(\Theta)$ and $\iota_1(\Theta)$ being intervals in $\pindex(V)$, and $\sigma(V)(\iota_0(\theta)) = \sigma(V)(\iota_1(\theta)) = f_0(\theta)$ for all $\theta \in \Theta$.  If $\iota_0(\theta') = \iota_1(\theta')$ for some $\theta' \in \Theta$ then $\iota_0 = \iota_1$.
\end{lemma}

\begin{proof}  Assume the hypotheses and let $\lambda' = \iota_0(\theta') = \iota_1(\theta')$.  Letting $\theta < \theta'$ in $\Theta$ be given, there is a unique $\lambda < \lambda'$ such that $\sigma(V)(\lambda) = f_0(\theta)$ and $|\{\lambda'' \in \pindex(V)\mid \lambda < \lambda'' < \lambda', \sigma(V)(\lambda'') = \sigma(V)(\lambda)\}| = |\{\theta'' \in \Theta \mid \theta < \theta'' < \theta', f_0(\theta'') = f_0(\theta)\}|$.  Since $\iota_0(\Theta)$ and $\iota_1(\Theta)$ are intervals in $\pindex(V)$ and $\sigma(V)(\iota_0(\theta_0)) = \sigma(V)(\iota_1(\theta_0)) = f_0(\theta_0)$ for all $\theta_0 \in \Theta$, it must be that $\lambda = \iota_0(\theta) = \iota_1(\theta)$.  Thus $\iota_0(\theta) = \iota_1(\theta)$ for all $\theta < \theta'$, and that $\iota_0(\theta) = \iota_1(\theta)$ for $\theta > \theta'$ follows similarly.

\end{proof}

\begin{lemma}\label{injectionsarefew}  Suppose that $\Theta$ is a totally ordered set and $f_0: \Theta \rightarrow \omega$ is a function.  If $V \in \Red_c$ then there are finitely many order embeddings $\iota: \Theta \rightarrow \pindex(V)$ with $\iota(\Theta)$ an interval and $\sigma(V)(\iota(\theta)) = f_0(\theta)$ for all $\theta \in \Theta$.
\end{lemma}

\begin{proof}  If $\Theta$ is empty then there is exactly one order embedding to $\pindex(V)$, namely the empty function.  If $\Theta$ is not empty, then fix $\theta' \in \Theta$.  Notice that there are only finitely many $\lambda' \in \pindex(V)$ such that $f_0(\theta') = \sigma(V)(\lambda')$ (since $V$ is a word), and any order embedding $\iota: \Theta \rightarrow \pindex(V)$ with $\iota(\Theta)$ an interval in $\pindex(V)$ and $\sigma(V)(\iota(\theta)) = f_0(\theta)$ for all $\theta \in \Theta$ and $\iota(\theta') = \lambda'$ is unique by Lemma \ref{intervalsareboring}.  Thus the conclusion holds.

\end{proof}

\begin{lemma}\label{blocksubword}  Suppose that $\{W_x\}_{x \in X} \subseteq \Red_c$ with $|X| < 2^{\aleph_0}$, that $\Theta$ is a totally ordered set and $f_0: \Theta \rightarrow \omega$ is a function.  If $f_1: \omega \rightarrow \Theta$ is an injective function (not necessarily preserving order) and $f_2: f_1(\omega) \rightarrow \{-1, 1\}$ is a function then there exists a function $g: f_1(\omega) \rightarrow \omega \setminus \{0\}$ such that there exists no $W \in(\bigcup_{x \in X} \pchunk(W_x))^{\pm 1}$ with $\iota: \Theta \equiv \pindex(W)$, $\sigma(W)(\iota(\theta)) = f_0(\theta)$ and $W \upharpoonright_p \{\iota(\theta)\} \equiv c_{f_1(\theta)}^{f_2(\theta)g(\theta)}$ for all $\theta \in f_1(\omega)$.

\end{lemma}

\begin{proof}  Let $\{\iota_y\}_{y \in Y}$ be the collection of all order embeddings with domain $\Theta$, codomain an element in $\{\pindex(W_x^{\pm 1})\}_{x \in X}$, say $\pindex(W_{x_y}^{i_{y}})$ where $i_{y}\in \{-1, 1\}$, $\iota_y(\Theta)$ an interval in $\pindex(W_{x_y}^{i_{y}})$, and $f_0(\theta) = \sigma(W_{x_y}^{i_y})(\iota_y(\theta))$.  We assume that the indexing $Y$ has no duplications: $y_0 \neq y_1$ implies that $\iota_{y_0} \neq \iota_{y_1}$.  Notice that $|Y| < 2^{\aleph_0}$ since by Lemma \ref{injectionsarefew} for each $x \in X$ there can be only finitely many $y \in Y$ with $x = x_y$.

The set of all functions $g: f_1(\omega) \rightarrow \omega \setminus \{0\}$ is of cardinality $2^{\aleph_0}$ and so it is possible to select $g: f_1(\omega) \rightarrow \omega \setminus \{0\}$ such that for each $y \in Y$ there exists $\theta_y \in f_1(\omega)$ such that $W_{x_y}^{i_y} \upharpoonright_p \{\iota_y(\theta_y)\} \not\equiv c_{f_1(\theta_y)}^{f_2(\theta_y)g(\theta_y)}$.  Clearly this $g$ safisfies the conclusion.
\end{proof}

\begin{lemma}\label{avoidinc}  Let $\{W_x\}_{x \in X} \subseteq \Red_c$ with $|X| < 2^{\aleph_0}$.  There exists a word $V \in \Red_c \setminus \Pfine(\{W_x\}_{x \in X})$.
\end{lemma}

\begin{proof}  Let $\{Z_m\}_{m \in \omega}$ be a collection of disjoint subsets of $\omega$ such that $|Z_m| = m + 1$ and all elements of $Z_m$ are below all elements of $Z_{m + 1}$ under the order on $\omega$.  We define words $V_m$ to be such that $\pindex(V_m)$ is equal to the set $Z_m$ under the restricted order from $\omega$, and $\sigma(V_m)(k) = k$.  Thus $V_m \equiv c_{k_{m, 0}}^{l_{m, 0}}c_{k_{m, 1}}^{l_{m, 1}}\cdots c_{k_{m, m}}^{l_{m, m}}$ where $Z_m = \{k_{m, 0}, \ldots, k_{m, m}\}$ and $k_{m, j} < k_{m, j + 1}$ and the exponents $l_{m, 0}, \ldots, l_{m, m}$ have yet to be determined.  The word $V$ is given by the product $V \equiv \prod_{m \in \omega} V_m$.  However the undetermined exponents in each $V_m$ are chosen it is clear that $V$ is reduced and provided the undetermined exponents are nonzero and we have $\pindex(V) \equiv \prod_{m \in \omega} \pindex(V_m)$.

So far we have determined $\pindex(V)$ and $\sigma(V)$, and we set $f_0: \pindex(V) \rightarrow \omega$ equal to $\sigma(V)$.  Let $f_{1, 0}: \omega \rightarrow \pindex(W)$ be the function where $f_{1, 0}(m) = k_{m, 0}$, let $f_{1, 1}: \omega \rightarrow \pindex(W)$ be the function where $f_{1, 1}(m) = k_{m + 1, 1}$ (i.e. the second element in $\pindex(V_{m + 1})$), $f_{1, 2}: \omega \rightarrow \pindex(W)$ has $f_{1, 2}(m) = k_{m + 2, 2}$ (the third element in $\pindex(V_{m + 2})$), etc.  Obviously each $f_{1, n}$ is injective and $f_{1, n_0}(\omega) \cap f_{1, n_1}(\omega) = \emptyset$ when $n_0 \neq n_1$.  For each $n \in \omega$ we let $f_{2, n}: f_{1, n}(\omega) \rightarrow \{-1, 1\}$ be the constant map to $1$.  Applying Lemma \ref{blocksubword}, for each $n \in \omega$ we select $g_n: f_{1, n}(\omega) \rightarrow \omega \setminus \{0\}$ such that there exists no $W \in (\bigcup_{x \in X} \pchunk(W_x))^{\pm 1}$ with $\iota: \pindex(\prod_{m = n}^{\infty}V_m) \equiv \pindex(W)$, $\sigma(W)(\iota(k)) = f_0(k)$ and $W \upharpoonright_p \{\iota(k)\} \equiv c_{f_{1, n}(k)}^{g_n(k)}$ for all $k \in f_{1, n}(\omega)$.

Let $V_m \equiv c_{k_{m, 0}}^{g_0(k_{m, 0})}c_{k_{m, 1}}^{g_1(k_{m, 1})}\cdots c_{k_{m, m}}^{g_m(k_{m, m})}$.  Now we have determined $V$.  If it is the case that $V \in \Pfine(\{W_x\}_{x \in X})$ then by Lemma \ref{elementsofthegeneratedsubgroupHA} there is a terminal interval $\Lambda \subseteq \Pfine(V)$ and $x \in X$ and $i\in \{-1, 1\}$ such that $V \upharpoonright_p \Lambda \in \pchunk(W_x^i)$.  As $\Lambda$ is a terminal interval in $\pindex(V)$, it is cofinite in $\pindex(V)$, and so we select $n \in \omega$ such that $\pindex(\prod_{m = n}^{\infty}V_m) \subseteq \Lambda$.  Select interval $I \subseteq \pindex(W_x^i)$ with $V \upharpoonright_p \Lambda \equiv W_x^i \upharpoonright_p I$ and let $\iota: \Lambda \rightarrow I$ be the induced order isomorphism.  Notice that $\sigma(W_x^i)(\iota(k)) = f_0(k)$ for all $k \in \pindex(\prod_{m = n}^{\infty}V_m)$.  Then by how $g_n$ was defined we have some $k \in f_{1, n}(\omega)$ such that $W_x^i \upharpoonright_p \{\iota(k)\} \not\equiv c_k^{g_n(k)} \equiv V \upharpoonright_p \{k\}$, a contradiction.

Notice that we have even shown that for each $n \in \omega$ the subword $\prod_{m = n}^{\infty} V_m$ is not an element of $\Pfine(\{W_x\}_{x \in X})$.

\end{proof}

\end{subsection}

\begin{subsection}{$\omega$-type concatenations.}\label{omegaconcatnewsetting}  In this subsection we prove the following.

\begin{proposition}\label{omegawordsnew}  Suppose that $\{\coi(W_x, \iota_x, U_x)\}_{x \in X}$ is a coherent collection of coi triples from $\Red_c$ to $\Red_2$ and that $|X| < 2^{\aleph_0}$ and $\Pure(\Red_c) \subseteq \Pfine(\{W_x\}_{x \in X})$.

\begin{enumerate}

\item  Suppose $W \in \Red_c$ with $\pindex(W) \equiv \prod_{n\in \omega} I_n$ and each $I_n \neq \emptyset$, $W \upharpoonright_p I_n \in \Pfine(\{W_x\}_{x\in X})$, and $W \notin \Pfine(\{W_x\}_{x\in X})$.  Then there exists $U \in \Red_2$ and coi $\iota$ from $W$ to $U$ such that $\{\coi(W_x, \iota_x, U_x)\}_{x\in X} \cup \{\coi(W, \iota, U)\}$ is coherent.

\item  Suppose $U \in \Red_2$ with $\pindex(U) \equiv \prod_{n \in \omega} I_n$ and each $I_n \neq \emptyset$, $U \upharpoonright_p I_n \in \Pfine(\{U_x\}_{x \in X})$, and $U \notin \Pfine(\{U_x\}_{x \in X})$.  Then there exists $W \in \Red_c$ and coi $\iota$ from $W$ to $U$ such that $\{\coi(W_x, \iota_x, U_x)\}_{x\in X} \cup \{\coi(W, \iota, U)\}$ is coherent.

\end{enumerate}

\end{proposition}

\begin{proof}  Claim (1) has the same proof as Proposition \ref{omegawords}, almost word for word, and so we do not write the proof for this (one constructs an appropriate word, shows that it is an element in $\Red_2$, defines the coi in the natural way and argues regarding coherence precisely in the same way as in that proposition).  For claim (2) we assume the hypotheses.  Let $\{Z_m\}_{m \in \omega}$ be a collection of disjoint subsets of $\omega$ such that $|Z_m| = m + 1$ and all elements of $Z_m$ are below all elements of $Z_{m + 1}$ under the order on $\omega$.

Let $U_m \equiv U \upharpoonright_p I_m$ for each $m \in \omega$.  By Lemmas \ref{findsomerepresentativenew} and \ref{makeitsmallernew} we select a word $W_0 \in \Red_c$ and coi $\iota_0$ from $W_0$ to $U_0$ such that $\{\coi(W_x, \iota_x, U_x)\}_{x\in X} \cup \{\coi(W_0, \iota_0, U_0)\}$ is coherent, the domain of $\iota_0$ is nonempty, and $\|W_0\| < \frac{1}{\max(Z_0) + 1}$.  Assuming we have defined $W_0, \ldots, W_m$ and $\iota_0, \ldots, \iota_m$ we apply Lemmas \ref{findsomerepresentativenew} and \ref{makeitsmallernew} to find $W_{m + 1}$ and coi $\iota_{m + 1}$ from $W_{m + 1}$ to $U_{m + 1}$ such that $\{\coi(W_x, \iota_x, U_x)\}_{x\in X} \cup \{\coi(W_0, \iota_0, U_0), \ldots, \coi(W_{m + 1}, \iota_{m + 1}, U_{m + 1})\}$ is coherent, the domain of $\iota_{m + 1}$ is nonempty, and $\|W_{m + 1}\| < \frac{1}{\max(Z_{m + 1}) + 1}$.

We define words $V_m$ to be such that $\pindex(V_m)$ is equal to the set $Z_m$ under the restricted order from $\omega$, and $\sigma(V_m)(k) = k$.  Thus $V_m \equiv c_{k_{m, 0}}^{l_{m, 0}}c_{k_{m,1}}^{l_{m, 1}}\cdots c_{k_{m, m}}^{l_{m, m}}$ where $Z_m = \{k_{m, 0}, \ldots, k_{m, m}\}$ and $k_{m, j} < k_{m, j + 1}$ and the exponents $l_{m, 0}, \ldots, l_{m, m}$ have yet to be determined.  The word $W$ is given by the product $W \equiv \prod_{m \in \omega} W_mV_m \equiv W_0V_0W_1V_1\cdots$.  Provided the undetermined exponents in each $V_m$ are chosen so as to all be nonzero, the word $W$ is reduced (by arguing as in Proposition \ref{omegawords}) and we have $\pindex(W) \equiv \prod_{m \in \omega} \pindex(W_m)\pindex(V_m)$.

So far we have determined $\pindex(W)$ and $\sigma(W)$.  Let $f_{1, 0}: \omega \rightarrow \pindex(W)$ be the function where $f_{1, 0}(m) = \min\pindex(V_m)$, let $f_{1, 1}: \omega \rightarrow \pindex(W)$ be the function where $f_{1, 1}(m)$ is the second element in $\pindex(V_{m + 1})$, $f_{1, 2}: \omega \rightarrow \pindex(W)$ has $f_{1, 2}(m)$ being the third element in $\pindex(V_{m + 2})$, etc.  Obviously each $f_{1, n}$ is injective and $f_{1, n_0}(\omega) \cap f_{1, n_1}(\omega) = \emptyset$ when $n_0 \neq n_1$.  For each $n \in \omega$ we let $f_{2, n}: f_{1, n}(\omega) \rightarrow \{-1, 1\}$ be the constant map to $1$.  Applying Lemma \ref{blocksubword}, for each $n \in \omega$ we select $g_n: f_{1, n}(\omega) \rightarrow \omega \setminus \{0\}$ such that assigning $f_{1, n}(k)$ the exponent $g_n(k)$ guarantees that $\prod_{m = n}^\infty W_mV_m$ is not in $(\bigcup_{x\in X}\pchunk(W_x) \cup \bigcup_{j \in \omega}\pchunk(W_j))^{\pm 1}$.

Thus we let $V_m \equiv c_{k_{m, 0}}^{g_0(k_{m, 0})}c_{k_{m, 1}}^{g_1(k_{m, 1})}\cdots c_{k_{m, m}}^{g_m(k_{m, m})}$.  Now we have defined $W$, and $W$ is reduced with $\pindex(W) \equiv \prod_{m \in \omega} \pindex(W_m)\pindex(V_m)$.  Arguing as in \ref{avoidinc} we see that $W \notin \Pfine(\{W_x\}_{x \in X} \cup\{W_j\}_{j \in \omega})$ and indeed $\prod_{m = n}^{\infty} W_mV_m \notin \Pfine(\{W_x\}_{x \in X} \cup\{W_j\}_{j \in \omega})$ for each $n \in \omega$.  We let $\iota$ be the coi from $W$ to $U$ defined by $\iota = \bigcup_{m \in \omega} \iota_m$.

We check that $\{\coi(W_x, \iota_x, U_x)\}_{x\in X} \cup \{\coi(W_m, \iota_m, U_m)\}_{m \in \omega} \cup \{\coi(W, \iota, U)\}$ is coherent.  Suppose $z \in X \cup \omega$, $\Lambda_0 \subseteq \pindex(W)$ and $\Lambda_1 \subseteq \pindex(W_z)$ are intervals and $i \in \{-1, 1\}$ are such that $W \upharpoonright_p \Lambda_0 \equiv (W_z \upharpoonright_p \Lambda_1)^i$.  If $\{n \in \omega \mid \Lambda_0 \cap \pindex(W_n) \neq \emptyset\}$ is infinite, then it follows from the fact that $\Lambda_0$ is an interval in $\pindex(W)$ that $W \upharpoonright_p \Lambda_0$ has a word $\prod_{m = n}^{\infty}W_nV_n$ as a p-chunk, for some $n \in \omega$.  However this requires that $\prod_{m = n}^{\infty}W_nV_n \in \Pfine(\{W_x\}_{x \in X} \cup\{W_j\}_{j \in \omega})$, which is a contradiction.  Thus the set $\{n \in \omega \mid \Lambda_0 \cap \pindex(W_n) \neq \emptyset\}$ is finite, and it is straightforward to argue that 

\begin{center}

$[[U \upharpoonright_p \varpropto(\Lambda_0, \iota)]] = [[(U_z \upharpoonright_p \varpropto(\Lambda_1, \iota_z))^i]]$

\end{center}

\noindent from the fact that $\{\coi(W_x, \iota_x, U_x)\}_{x\in X} \cup \{\coi(W_m, \iota_m, U_m)\}_{m \in \omega}$ is coherent, as was done in Proposition \ref{omegawords}.

Suppose $\Lambda_0, \Lambda_1 \subseteq \pindex(W)$ are intervals and $i \in \{-1, 1\}$ are such that $W \upharpoonright_p \Lambda_0 \equiv (W \upharpoonright_p \Lambda_1)^i$.  We let $K_0 = \{n \in \omega \mid \Lambda_0 \cap \pindex(W_n) \neq \emptyset\}$ and $K_1 = \{n \in \omega \mid \Lambda_1 \cap \pindex(W_n) \neq \emptyset\}$.  If either of $K_0$ or $K_1$ is finite then from the fact that $\Pure(\Red_c) \subseteq \Pfine(\{W_x\}_{x \in X})$ we see that $W \upharpoonright_p \Lambda_0 \in \Pfine(\{W_x\}_{x \in X} \cup\{W_j\}_{j \in \omega})$ and so both of $K_0$ and $K_1$ are therefore finite.  If $K_0$ is finite then we see that 

\begin{center}

$[[U \upharpoonright_p \varpropto(\Lambda_0, \iota)]] = [[(U \upharpoonright_p \varpropto(\Lambda_1, \iota))^i]]$

\end{center}

\noindent from the coherence of $\{\coi(W_x, \iota_x, U_x)\}_{x\in X} \cup \{\coi(W_m, \iota_m, U_m)\}_{m \in \omega}$, by arguing as in Case 1 of Proposition \ref{omegawords}.  Thus we may assume that $K_0$ and $K_1$ are infinite.  As both are infinite, we see that $\Lambda_0$ and $\Lambda_1$ are each nonempty terminal intervals in $\pindex(W)$.  Since $\Pure(\Red_c) \subseteq \Pfine(\{W_x\}_{x \in X} \cup\{W_j\}_{j \in \omega})$, we know that every proper initial subword of $W \upharpoonright_p \Lambda_0$ is in $\Pfine(\{W_x\}_{x \in X} \cup\{W_j\}_{j \in \omega})$, and we also know that every nonempty terminal subword of $W \upharpoonright_p \Lambda_0$ is not in $\Pfine(\{W_x\}_{x \in X} \cup\{W_j\}_{j \in \omega})$.  The similar claims hold for $W \upharpoonright_p \Lambda_1$.  But since $W \upharpoonright_p \Lambda_0 \equiv (W \upharpoonright_p \Lambda_1)^i$ this implies that $i = 1$.  Thus $W \upharpoonright_p \Lambda_0 \equiv W \upharpoonright_p \Lambda_1$, and since $\Lambda_0$ and $\Lambda_1$ are terminal intervals in $\pindex(W)$ we know that at least one of $\Lambda_0 \subseteq \Lambda_1$ or $\Lambda_1 \subseteq \Lambda_0$ holds.  But we have already seen that no word may be $\equiv$ to a proper terminal subword of itself (see the proof of Proposition \ref{omegawords}) and so $\Lambda_0 = \Lambda_1$ and it immediately follows that 

\begin{center}

$[[U \upharpoonright_p \varpropto(\Lambda_0, \iota)]] = [[(U \upharpoonright_p \varpropto(\Lambda_1, \iota))^i]]$.

\end{center}

Finally, one analyzes the cases where $\Lambda_0 \subseteq \pindex(U)$ and $\Lambda_1 \subseteq \pindex(U_y)$ or $\Lambda_1 \subseteq \pindex(U)$ in the same way as above, using the coherence of the collection $\{\coi(W_x, \iota_x, U_x)\}_{x\in X} \cup \{\coi(W_m, \iota_m, U_m)\}_{m \in \omega}$ and the fact that for every proper initial subinterval $\Lambda$ of $\pindex(U)$ we have $U \upharpoonright_p \Lambda \in \Pfine(\{U_x\}_{x \in X} \cup \{U_j\}_{j \in J})$ and for every nonempty terminal interval $\Lambda$ of $\pindex(U)$ we have $U \upharpoonright_p \Lambda \notin \Pfine(\{U_x\}_{x \in X} \cup \{U_j\}_{j \in J})$.

\end{proof}

\end{subsection}

\begin{subsection}{$\mathbb{Q}$-type concatenations.}\label{Qconcatnewsetting}

We begin with an elementary result.

\begin{lemma}\label{denseinQ}  Suppose that $\{Y_n\}_{n \in \omega}$ is a collection of nonempty finite subsets of $\mathbb{Q}$ such that $\mathbb{Q} = \bigsqcup_{n \in \omega} Y_n$.  Then there exists a collection $\{N_k\}_{k \in \omega}$ such that $\bigsqcup_{k \in \omega} N_k = \omega$, each $N_k$ is infinite, and $\bigcup_{n \in N_k} Y_n$ is dense in $\mathbb{Q}$ for each $k \in \omega$.
\end{lemma}

\begin{proof}  Let $h: \omega \rightarrow \omega \times \omega$ be a bijection and define $h_1: \omega \rightarrow \omega$ by letting $h_1(m)$ be the second coordinate of $h(m)$.  Let $\{I_j\}_{j \in \omega}$ be an enumeration of all nonempty open intervals in $\mathbb{Q}$ with rational supremum and infimum.  We will inductively construct an increasing sequence $F_0 \subseteq F_1 \subseteq \cdots$ of finite subsets of $\omega$.  Let $F_0 = \emptyset$ and assuming that we have defined $F_{m - 1}$ we select $n_m \in \omega \setminus F_{m - 1}$ to be minimal such that $Y_{n_m} \cap I_{h_1(m)} \neq \emptyset$ and let $F_m = F_{m - 1} \cup \{n_m\}$.  Letting $N_k = \{n_{h^{-1}(k, j)}\}_{j \in \omega}$ it is easy to see that the conclusion holds.

\end{proof}

\begin{proposition}\label{Qtypenew}  Suppose that $\{\coi(W_x, \iota_x, U_x)\}_{x \in X}$ is a coherent collection of coi triples from $\Red_c$ to $\Red_2$ and that $|X| < 2^{\aleph_0}$ and $\Pure(\Red_c) \subseteq \Pfine(\{W_x\}_{x \in X})$.

\begin{enumerate}

\item  Suppose that $W \in \Red_c$ is such that $\pindex(W) \equiv \prod_{q \in \mathbb{Q}} I_q$ with each $I_q \neq \emptyset$, $W \upharpoonright_p I_q \in \Pfine(\{W_x\}_{x\in X})$ for each $q \in \mathbb{Q}$, and $W \upharpoonright_p \bigcup\Lambda \notin \Pfine(\{W_x\}_{x\in X})$ for each interval $\Lambda \subseteq \mathbb{Q}$ with more than one point.  Then there exists $U \in \Red_2$ and coi $\iota$ from $W$ to $U$ such that $\{\coi(W_x, \iota_x, U_x)\}_{x\in X} \cup \{\coi(W, \iota, U)\}$ is coherent.

\item  Suppose that $U \in \Red_2$ is such that $\pindex(U) \equiv \prod_{q \in \mathbb{Q}} I_q$ with each $I_q \neq \emptyset$, $U \upharpoonright_p I_q \in \Pfine(\{U_x\}_{x\in X})$ for each $q \in \mathbb{Q}$, and $U \upharpoonright_p \bigcup\Lambda \notin \Pfine(\{U_x\}_{x\in X})$ for each interval $\Lambda \subseteq \mathbb{Q}$ with more than one point.  Then there exists $W \in \Red_c$ and coi $\iota$ from $W$ to $U$ such that $\{\coi(W_x, \iota_x, U_x)\}_{x\in X} \cup \{\coi(W, \iota, U)\}$ is coherent.

\end{enumerate}

\end{proposition}

\begin{proof}  Claim (1) is proved as in Proposition \ref{Qtype} with almost no alteration.  For claim (2) we let $\{U_n\}_{n \in \omega}$ be a list such that for each $q \in \mathbb{Q}$ we have some $n \in \omega$ for which either $U \upharpoonright_p I_q \equiv U_n$ or $U \upharpoonright_p I_q \equiv U_n^{-1}$, and $n \neq n'$ implies $U_n \not\equiv U_{n'}\not\equiv U_n^{-1}$.  Such a list must be infinite, of course, as $U$ is a word.  We have by assumption that $\{U_n\}_{n \in \omega} \subseteq \Pfine(\{U_x\}_{x \in X})$.  Select $W_0 \in \Red_c$ and coi $\iota_0$ from $W_0$ to $U_0$, with nonempty domain and range, such that $\{\coi(W_x, \iota_x, U_x)\}_{x \in X} \cup\{\coi(W_0, \iota_0, U_0)\}$ is coherent and $\|W_0\| < 1$, using Lemmas \ref{findsomerepresentativenew} and \ref{makeitsmallernew}.  Generally select by Lemmas \ref{findsomerepresentativenew} and \ref{makeitsmallernew} a word $W_{n + 1} \in \Red_c$ and coi $\iota_{n + 1}$ from $W_{n + 1}$ to $U_{n + 1}$, with nonempty domain and range, so that $\{\coi(W_x, \iota_x, U_x)\}_{x \in X} \cup\{\coi(W_0, \iota_0, U_0), \ldots, \coi(W_{n + 1}, \iota_{n + 1}, U_{n + 1})\}$ is coherent and $\|W_{n + 1}\| < \frac{1}{n + 1}$.

Define functions $h_0: \mathbb{Q} \rightarrow \omega$ and $h_1: \mathbb{Q} \rightarrow \{-1, 1\}$ by $U \upharpoonright_p I_q \equiv U_{h_0(q)}^{h_1(q)}$.  For each $q \in \mathbb{Q}$ we will let $W_q \equiv (c_{h_0(q)}^{z_{h_0(q)}} W_{h_0(q)} c_{h_0(q)}^{z_{h_0(q)}})^{h_1(q)}$ and the nonzero integers $z_n$ are yet to be determined.  The word $W \equiv \prod_{q \in \mathbb{Q}} W_q$ will be reduced by the same argument as that for Lemma \ref{QUisreduced}, and $\pindex(W) \equiv \prod_{q \in \mathbb{Q}}\pindex(W_q)$.

Now that we have determined the values of $\sigma(W)$ we still need to fix the nonzero integers $z_n$.  For each $n \in \omega$ we let $Y_n$ be the preimage $h_0^{-1}(n)$.  We have $\omega = \bigsqcup_{n \in \omega} Y_n$ and each $Y_n$ is nonempty and finite.  By Lemma \ref{denseinQ} we select a collection $\{N_k\}_{k \in \omega}$ of infinite subsets of $\omega$ such that $\omega = \bigsqcup_{k \in \omega} N_k$ and $\bigcup_{n \in N_k} Y_n$ is dense in $\mathbb{Q}$ for each $k \in \omega$.  Let $\{J_j\}_{j \in \omega}$ be an enumeration of all nonempty open intervals in $\mathbb{Q}$ with rational infimum and supremum.  Then $N_j \cap J_j$ is dense in $J_j$ for each $j \in \omega$.  Fix $j \in \omega$.  Since $Y_n$ is finite for each $n \in N_j$ we can select an injection $F_{1, j}: \omega \rightarrow J_j$ such that $h_0(F_{1, j}(m_0)) \neq h_0(F_{1, j}(m_1))$ when $m_0 \neq m_1$.  Let $f_{1, j}(m)$ be the maximum element in $\pindex(W_{F_{1, j}(m)})$ (whose exponent is not yet determined).  By Lemma \ref{blocksubword} we pick a function $g_j: f_{1, j}(\omega) \rightarrow \omega \setminus \{0\}$ such that setting $z_{h_0(F_{1, j}(m))} = g_j(m)$ guarantees that the word $W \upharpoonright_p \bigcup J_j$ is not an element in $(\bigcup_{x \in X}\pchunk(W_x) \cup \bigcup_{n \in \omega} \pchunk(W_n))^{\pm 1}$.  Thus for all $n \in N_j$ we set $z_n = g_j(m)$ provided $h_0(F_{1, j}(m)) = n$ and set $z_n = 1$ if $n \notin h_0(F_{1, j}(\omega))$.

We have now determined the exponents $z_n$ and so the word $W$ is completely determined.  We have already noticed that the word $W$ is reduced and that $\pindex(W) \equiv \prod_{q \in \mathbb{Q}}\pindex(W_q)$.  Each $W_q$ is an element in $\Pfine(\{W_x\}_{x \in X} \cup \{W_n\}_{n \in \omega})$ since $c_{h_0(q)} \in \Pure(\Red_c) \subseteq \Pfine(\{W_x\}_{x \in X}) \subseteq \Pfine(\{W_x\}_{x \in X} \cup \{W_n\}_{n \in \omega})$.  We claim that each $W_q$ is a maximal subword of $W$ which is an element of $\Pfine(\{W_x\}_{x \in X} \cup \{W_n\}_{n \in \omega})$ in the sense that $\pindex(W_q)$ is an interval in $\pindex(W)$ and there is no interval $I$ in $\pindex(W)$ which properly includes $\pindex(W_q)$ such that $W \upharpoonright_p I \in \Pfine(\{W_x\}_{x \in X} \cup \{W_n\}_{n \in \omega})$.  Were it the case that such an $I$ existed, we would have an open interval $J_j \subseteq \mathbb{Q}$ such that $W \upharpoonright_p \bigcup J_j \in (\bigcup_{x \in X}\pchunk(W_x) \cup \bigcup_{n \in \omega}\pchunk(W_n))^{\pm 1}$ by Lemma \ref{elementsofthegeneratedsubgroupHA} and the fact that $\mathbb{Q}$ is order dense, but this was ruled out by how the exponents $\{z_n\}_{n \in N_j}$ were selected.

Now define the coi $\iota$ from $W$ to $U$ in the very natural way so that the restriction $\iota \upharpoonright_p \dom(\iota) \cap \pindex(W_q)$ commutes with $\iota_{h_0(q)}$ if $h_1(q) = 1$ and $\iota \upharpoonright_p \dom(\iota) \cap \pindex(W_q)$ commutes with the reverse of $\iota_{h_0(q)}$ defined on $W_{h_0(q)}^{-1}$ if $h_1(q) = - 1$.  The check that $\{\coi(W_x, \iota_x, U_x)\}_{x \in X} \cup \{\coi(W_n, \iota_n, U_n)\}_{n \in \omega} \cup \{\coi(W, \iota, U)\}$ is coherent now follows that used in Proposition \ref{Qtype}.

\end{proof}

\end{subsection}

\begin{subsection}{Arbitrary extensions.}\label{arbitraryextensionsnewsetting}  We complete the proof of Theorem \ref{bigisomorphism2}.  The following result is proved in precisely the same way as Proposition \ref{arbitraryextensions}, using Propositions \ref{omegawordsnew} and \ref{Qtypenew} in place of Propositions \ref{omegawords} and \ref{Qtype}, respectively.

\begin{proposition}\label{arbitraryextensionsnew}   Suppose that $\{\coi(W_x, \iota_x, U_x)\}_{x\in X}$ is a coherent collection of coi from $\Red_c$ to $\Red_2$, that $|X| < 2^{\aleph_0}$, and that $\Pfine(\Red_c) \subseteq \Pfine(\{W_x\}_{x \in X})$.

\begin{enumerate}

\item  Given $W\in \Red_c$ there exists $U \in \Red_2$ and coi $\iota$ from $W$ to $U$ such that $\{\coi(W_x, \iota_x, U_x)\}_{x\in X} \cup \{\coi(W, \iota, U)\}$ is coherent.

\item  Given $U\in \Red_2$ there exists $W \in \Red_c$ and coi $\iota$ from $W$ to $U$ such that $\{\coi(W_x, \iota_x, U_x)\}_{x\in X} \cup \{\coi(W, \iota, U)\}$ is coherent.

\end{enumerate}

\begin{proof}[Proof of Theorem \ref{bigisomorphism2}]  As $|\Red_2| = |\Red_c| = 2^{\aleph_0}$ we let $\prec_c$ well-order $\Red_c$ in such a way that each element has fewer than $2^{\aleph_0}$ predecessors and $\prec_2$ well-order $\Red_2$ in such a way that each element has fewer than $2^{\aleph_0}$ predecessors.  We inductively define a coherent collection $\{\coi(W_{\zeta}, \iota_{\zeta}, U_{\zeta})\}_{\zeta < 2^{\aleph_0}}$ of coi triples from $\Red_c$ to $\Red_2$.

Let $\{W_n\}_{n \in \omega}$ be an enumeration of $\Pure(\Red_c)$ and notice that the collection $\{\coi(W_n, \iota_n, E)\}_{n \in \omega}$ is coherent, where of course $\iota_n$ is the empty function.

Suppose that we have defined coherent $\{\coi(W_{\zeta}, \iota_{\zeta}, U_{\zeta})\}_{\zeta < \mu}$ for all $\mu < \nu < 2^{\aleph_0}$.  We know $\{\coi(W_{\zeta}, \iota_{\zeta}, U_{\zeta})\}_{\zeta < \nu}$ is coherent by reasoning as in Lemma \ref{ascendingchaincoi}.  If $\nu \geq \omega $ is even then by Lemma \ref{avoidinc} we select a word $W_{\nu} \notin \Pfine(\{W_{\zeta}\}_{\zeta < \nu})$ which is minimal such under $\prec_c$ and by Proposition \ref{arbitraryextensionsnew} select a $U_{\nu} \in \Red_2$ and coi $\iota_{\nu}$ such that  $\{\coi(W_{\zeta}, \iota_{\zeta}, U_{\zeta})\}_{\zeta < \nu + 1}$ is coherent.  Similarly if $\nu \geq \omega$ is odd then by Lemma \ref{avoidpchunk} we select a word $U_{\nu} \notin \Pfine(\{U_{\zeta}\}_{\zeta < \nu})$ which is minimal such under $\prec_2$ and by Proposition \ref{arbitraryextensionsnew} select a $W_{\nu} \in \Red_c$ and coi $\iota_{\nu}$ such that  $\{\coi(W_{\zeta}, \iota_{\zeta}, U_{\zeta})\}_{\zeta < \nu + 1}$ is coherent.

Now $\Pfine(\{W_{\zeta}\}_{\zeta < 2^{\aleph_0}}) = \Red_c$ and $\Pfine(\{U_{\zeta}\}_{\zeta < 2^{\aleph_0}}) = \Red_2$.  Thus by Proposition \ref{obtainedisonew} we have an isomorphism $\Phi: \Red_c/\langle\langle \Pfine(\Red_c)\rangle\rangle \rightarrow \Co_2$ and we are done.

\end{proof}

\end{proposition}

\end{subsection}

\end{section}

\section*{Acknowledgement}

The author thanks Jeremy Brazas for the beautiful pictures used in this article.

\end{document}